\documentclass[12pt]{article} 
\usepackage{graphicx}
\usepackage{amssymb}
\usepackage{mathtools}
\usepackage{amsmath}
\usepackage{amsthm}
\usepackage{color}
\usepackage[hidelinks]{hyperref}
\usepackage{csquotes}
\usepackage[pagewise]{lineno}
\usepackage[english]{babel}
\usepackage{enumitem}
\usepackage{graphicx,url} 
\usepackage{tocloft}
\usepackage{sectsty}
\usepackage{mathrsfs}
\usepackage{eufrak}
\usepackage{comment}

\renewcommand{\Re}{\operatorname{Re}}
\newtheorem{thm}{Theorem}[section]
\newtheorem{cor}{Corollary}[section]
\newtheorem{lem}[thm]{Lemma}
\newtheorem{prop}{Proposition}[section]
\newtheorem{defn}{Definition}[section]
\newtheorem{exmp}{Example}[section]
\newtheorem{rmk}{Remark}[section]

\numberwithin{equation}{section}

\sectionfont{\fontsize{13}{15}\selectfont}
\subsectionfont{\fontsize{12}{15}\selectfont}

\def \R{{\mathbb{R}}}
\def \P {\Phi(x)}
\def \om{\omega(x)}
\def \japxi{\langle \xi \rangle}
\def \japxik{\langle \xi \rangle_k}
\def \japx{\langle x \rangle}
\def \go{g_{\Phi,k}}
\def \gt{\tilde{g}_{\Phi,k}}

\def \zint(#1){Z_{int}(#1)}
\def \zmid(#1){Z_{mid}(#1)}
\def \zext(#1){Z_{ext}(#1)}

\def \sobo(#1){H^{#1}_{\Phi,k}(\R^n)}
\def \V(#1){\Vert_{\Phi,k;#1}}

\def \sobol(#1,#2){\mathcal{H}^{#1,#2}_{\Phi,k;\Theta}(\R^n)}
\def \Vl(#1,#2){\Vert_{\Phi,k;\Theta,#2,#1}}

\def \var{(t,x,\xi)}
\def \J{[0,T] \times \R^n \times \R^n}
\def \la{\langle}
\def \rak{\rangle_k}
\def \ra{\rangle}
\def \h{h(x,\xi)}
\def \T{\Theta(x,\xi)}
\def \c(#1,#2){\chi^{(#1)}_{#2}}
\def \t(#1){\left( \frac{\theta(t)}{t}\right)^{#1}}
\def \p(#1){e^{#1\psi(t)}}
\def \ro{\rho}

\def \G(#1,#2){G^{#1,#2}(\omega,\go)}
\def \Gt(#1,#2,#3){G^{#1,#2}(#3,\go)}
\def \Gint(#1,#2,#3,#4){G^{#1,#2}\{#4\}(\omega,\go)_{#3}^{(1)}}
\def \Gmid(#1,#2,#3,#4,#5,#6){G^{#1,#2}\{#3,#5,#6\}(\omega,\go)_{#4}^{(2)}}
\def \Gext(#1,#2,#3,#4,#5,#6,#7){G^{#1,#2}\{#3,#4,#6,#7\}(\omega,\go)_{#5}^{(3)}}

\def \GTT(#1,#2,#3){G^{#1,#2}(#3,\go)}
\def \GT(#1,#2){G^{#1,#2}(\omega,\gt)}
\def \GTint(#1,#2,#3,#4){G^{#1,#2}\{#4\} (\omega,\gt)_{#3}^{(1)}}
\def \GTmid(#1,#2,#3,#4,#5,#6){G^{#1,#2}\{#3,#5,#6\}(\omega,\gt)_{#4}^{(2)}}
\def \GText(#1,#2,#3,#4,#5,#6,#7){G^{#1,#2}\{#3,#4,#6,#7\}(\omega,\gt)_{#5}^{(3)}}

\def \zpi{   \theta (\h)}
\def \zpii{ \tilde{\theta}(\h) \theta (\h) e^{\psi(\h)}}
\def \zi{t_{x,\xi}}
\def \zii{\tilde t_{x,\xi}}

\def \cn(#1){\mathcal{N}_{#1}(t,x,D_x)}
\def \tcn(#1){\tilde{\mathcal{N}}_{#1}(t,x,D_x)}

\def \CHI{ \chi(t/t_{x,\xi})}
\def \CHIC{\big(1- \CHI\big)}

\providecommand{\keywords}[1]
{
	\small	
	\textbf{\text{Keywords:}} #1
}

\providecommand{\subclass}[1]
{
	\small	
	\textbf{\text{MSC(2010):}} #1
}

\usepackage{geometry}
\title{\Large  Energy Estimates and Global Well-posedness for a \\Broad Class of Strictly Hyperbolic Cauchy Problems with \\Coefficients Singular in Time\thanks{Dedicated to Bhagawan Sri Sathya Sai Baba on the Occassion of His 96th Birthday.}}

\author{\normalsize Rahul Raju Pattar\thanks{rahulrajupattar@sssihl.edu.in (Corresponding Author)} , N. Uday Kiran\thanks{nudaykiran@sssihl.edu.in}  \\
	\small Department of Mathematics and Computer Science\\
	\small Sri Sathya Sai Institute of Higher Learning, Puttaparthi, India \\
}

\date{}
\begin{document} 
	
	\maketitle
	
	\begin{abstract}
		The goal of this paper is to establish a global well-posedness for a broad class of strictly hyperbolic Cauchy problems with coefficients in $C^2((0,T];C^\infty(\R^n))$ growing polynomially in $x$ and singular in $t$. The problems we study are of strictly hyperbolic type with respect to a generic weight and a metric on the phase space. 
		The singular behavior is captured by the blow-up of the first and second $t$-derivatives of the coefficients
		which allows the coefficients to be either logarithmic-type or oscillatory-type near $t=0$. 
		To arrive at an energy estimate, we perform a conjugation by a pseudodifferential operator of the form $e^{\nu(t)\Theta(x,D_x)},$ where $\Theta(x,D_x)$ explains the quantity of the loss by linking it to the metric on the phase space and the singular behavior while $\nu(t)$ gives a scale for the loss.
		We call the conjugating operator as {\itshape{loss operator}} and
		depending on its order we report that the solution experiences zero, arbitrarily small, finite or infinite loss in relation to the initial datum.
		We also provide a counterexample and derive the anisotropic cone conditions in our setting.\\
		\keywords{Loss of Regularity $\cdot$ Strictly Hyperbolic Operator with non-regular Coefficients $\cdot$ Logarithmic Singularity $\cdot$ Metric on the Phase Space $\cdot$ Global Well-posedness $\cdot$ Pseudodifferential Operators}\\
		\subclass{35L15 $\cdot$ 35B65 $\cdot$ 35B30 $\cdot$ 35S05 }
	\end{abstract}
	
	\begin{center}
		\tableofcontents
	\end{center}
	
	
	\section{Introduction}
	
	A study of loss of derivatives in relation to the initial datum plays a prominent role in the modern theory of hyperbolic equations. It is now well established that the loss is caused due to either an irregularity, a singularity, or a degeneracy of the coefficients with respect to the time variable (see for instance \cite{CSK,Cico1,Cico2,hiro,GG,KuboReis,CicoLor,LorReis,Petkov}, and the references therein). Recently, many authors have also considered the interplay between the irregularities in time and certain suitable bounds on the space (see for instance \cite{AscaCappi1,AscaCappi2,NUKCori1,Rahul_NUK2,RahulNUK3}).   
	
	In this work, our main interest is on the optimality of loss when the coefficients of a second order strictly hyperbolic Cauchy problem are singular in $t$ and unbounded in $x$. In particular, our interest is either logarithmic-type (i.e. $|\ln t|^r, r {\geq 1}$)  blow-up (see Example \ref{ex2}) or infinitely many oscillations (see Example \ref{ex1}) near $t=0$ and the unboundedness in $x$ characterized by a generic weight function. We study the Cauchy  problem using a pseudodifferential calculus associated to a metric on the phase space \cite{Lern,nicRodi} that generalizes the SG metric \cite{AscaCappi1,AscaCappi2}. We investigate how the behavior in $x$ and the singularity in $t$ variables of the coefficients interact and influence the regularity of the solution.
	
	The notion of a metric on phase space $T^{*}\R^{n}\ (\cong \R^{2n})$ was first introduced by H\"ormander \cite{Horm} who studied the smooth functions $p(x,\xi)$ called symbols, using the metric $\japxi^{2\delta} |dx|^2+ \japxi^{-2\rho} |d\xi|^2$, $0\leq \delta<\rho\leq 1$. For some $m \in \R \text{ and } C_{\alpha\beta}>0$, the symbol $p(x,\xi)$ satisfies
	\begin{linenomath*}
		\[
		|\partial_\xi^\alpha \partial_x^\beta p(x,\xi)| \leq C_{\alpha\beta} \japxi^{m-\rho|\alpha|+\delta|\beta|}, 
		\]
	\end{linenomath*}
	where $\alpha,\beta \in \mathbb{N}_0^n\;(\mathbb{N}_0=\mathbb{N} \cup \{0\})$ are multi-indices and $\japxi=(1+|\xi|^2)^{1/2}$. In a more general framework created by Beals and Fefferman \cite{Feff,nicRodi}, one can consider the symbol $p(x,\xi)$ that satisfies the estimate for some $C_{\alpha\beta}>0,$
	\begin{linenomath*}
		\begin{equation*}		
			|\partial_\xi^\alpha \partial_x^\beta p(x,\xi)| \leq C_{\alpha\beta} M(x,\xi) \Psi(x,\xi)^{-|\alpha|} \Phi(x,\xi)^{-|\beta|},
		\end{equation*}
	\end{linenomath*}
	where the positive functions $M(x,\xi), \Psi(x,\xi)$ and $\Phi(x,\xi)$ are specially chosen. In such a case, as given in \cite{Lern}, we can consider the following Riemannian structure on the phase space 
	\begin{linenomath*}
		\begin{equation}
			\label{beals_fefferman}
			g_{x,\xi}=\frac{|dx|^2}{\Phi(x,\xi)^2}+\frac{|d\xi|^2}{\Psi(x,\xi)^2};
		\end{equation}
	\end{linenomath*}
	where $\Phi(x,\xi)$ and $\Psi(x,\xi)$ satisfy certain Riemannian and Symplectic geometric restrictions \cite{Lern}. 
	
	In this paper, we consider a metric (\ref{beals_fefferman}) with $\Phi(x,\xi)=\Phi(x)$ and $\Psi(x,\xi)=\japxik=(k^2+|\xi|^2)^{1/2}$, for a large positive parameter $k$, and the weight function $M(x,\xi)=\om^{m_2}\japxik^{m_1}$ for $m_1,m_2 \in \R.$ Here the functions $\om$ and $\P$ are positive monotone increasing in $|x|$ such that $1\leq \om \lesssim \P\lesssim \japx$ ($\japx := (1+\vert x\vert^2)^{1/2}$). Further, note that $\omega(x)$ and $\Phi(x)$ are associated with the weight and metric respectively, and they specify the structure of the differential equation in space variable. These functions will be discussed in detail in Section \ref{metric}.  
	
	Our interest is to study the behavior of solutions corresponding to a second order linear hyperbolic equation with coefficients that are polynomially growing $x$ and singular at $t=0$. Such equations have a well-known behavior of a loss of derivatives when the coefficients are bounded in space ($\R^n$) together with all their derivatives (see for example \cite{CSK,Cico1,CSR,KuboReis,GG}). The study of such equations is motivated by a behavioral study of blow-up solutions in quasilinear problems, see for example \cite[Section 4]{hiro}. Along with the extension of the results to a global setting ($x \in \R^n$ and the coefficients are allowed to grow polynomially in $x$), we also investigate the behavior at infinity of the solution in relation to the coefficients. By a loss of regularity of the solution in relation to the initial datum we mean a change of indices in an appropriate Sobolev space associated with the metric (\ref{beals_fefferman}).
	
	As a prototype of our model operator (see Section \ref{stmt}), consider the Cauchy problem:
	\begin{equation}
		\begin{rcases}
			\label{PO}
			\partial_t^2u - a(t,x)\Delta_xu = 0,  \quad (t,x) \in [0,T] \times \R^n, \\
			u(0,x)=u_1(x), \quad \partial_tu(0,x)=u_2(x),
		\end{rcases}
	\end{equation}
	where the coefficient $a(t,x)$ is in $ C^2((0,T];C^\infty(\R^n))$ and satisfies the following estimates
	\begin{alignat}{2}
		\label{ell}
		a(t,x) & \geq C_0 \; \om^2, \\ 
		\label{a}
		\vert \partial_x^\beta a(t,x) \vert &\leq C_{\beta}^{(1)} \; \tilde{\theta}(t)  \; \om^2\P^{-|\beta|}  ,\\
		\label{sing1}
		|\partial_x^\beta\partial_t a(t,x)| & \leq C_{\beta}^{(2)} \; \frac{\theta(t)}{t} \; \om^2\P^{-|\beta|}, \\
		\label{sing2}
		|\partial_x^\beta\partial_t^2 a(t,x)| & \leq C_{\beta}^{(3)} \left(\frac{\theta(t)}{t}\right)^2 e^{\psi(t)} \om^2\P^{-|\beta|}, 
	\end{alignat}
	where $ (t,x) \in (0,T] \times \R^n, \;\beta \in \mathbb{N}_0^n$, $C_0,C_{\beta}^{(j)}>0, j=1,2,3$. Here $\tilde \theta,\theta,\psi:(0,+\infty) \to (0,+\infty)$ are positive nonincreasing smooth functions such that $\tilde\theta(t),\theta(t),\psi(t) \geq 1$.
	\begin{rmk}
		From (\ref{sing1}), we have the following estimate
		\[
		\vert \partial_x^\beta a(T,x) - \partial_x^\beta a(t,x) \vert \leq  \int_{t}^{T} \vert \partial_x^\beta \partial_s a(s,x) \vert ds \leq C_{\beta}  \om^2\P^{-|\beta|}  \int_{t}^{T} \frac{\theta(s)}{s} ds.
		\]
		Implying
		\begin{equation*}
			\vert \partial_x^\beta a(t,x) \vert \leq C_{\beta} \om^2\P^{-|\beta|} \int_{t}^{T} \frac{\theta(s)}{s} ds.
		\end{equation*}
		Hence,  we define $\tilde \theta(t)$ as 
		\begin{align}\label{rel2}
			\tilde \theta(t) = \begin{cases}
				C\int_{t}^{T} \frac{\theta(s)}{s} ds,& \text{ when } a(t,\cdot) \text{ is unbounded} \\
				1,& \text{ otherwise,}
			\end{cases}
		\end{align}
		for some $C>0$. The definition (\ref{rel2}) of $\tilde\theta$ suggests that it grows atleast as $|\ln t|$ near $t=0$ and $|\tilde\theta(t)'| = \frac{\theta(t)}{t}$  when the coefficients are unbounded in $t$.
	\end{rmk}
	
	We define a function $\vartheta:(0,+\infty) \to (0,+\infty)$ as
	\begin{equation}\label{cq}
		\vartheta\left( \frac{1}{t}  \right) :=  \theta(t) (\tilde{\theta}(t) + \psi(t)).
	\end{equation}
	This function plays a crucial role in performing conjugation and in defining Sobolev spaces. The rate of growth of $\vartheta$ defines the quantity of the loss in regularity. For the purpose of pseudodifferential calculus in our context, we need $\vartheta$ to satisfy the following estimate
	\begin{equation}
		\label{ineq3}
		\left|  \frac{d^j}{dr^j}\vartheta(r)  \right| \leq C_j \frac{\vartheta(r)}{r^j}, \quad r \in \R^+
	\end{equation}
	for some $C_j>0.$ Note that the above estimate is natural for logarithmic-type functions.
	
	From the broad class of singular coefficients that we consider, below are certain examples of $a(t,x)$ satisfying (\ref{ell} - \ref{sing2}).  Let $n=1,$ $\kappa_1,\kappa_2 \in [0,1]$ and $T$ be sufficiently small.

	\begin{exmp} \label{ex1}
		$ a(t,x) = 4\japx^{2\kappa_1} \left(2+\sin \left( \japx^{1-\kappa_2}\right)\right) c(t) $
		where 
		\[
		c(t) = 2 + e^{-|\ln t|^{1-\alpha}} \sin \left( |\ln t|^{2\alpha} e^{|\ln t|^{1-\alpha}} \right), \quad \text{ for some } \alpha \in (0,1).
		\]
		Here $\om = 2\japx^{\kappa_1}, \P = \japx^{\kappa_2}, \tilde{\theta}(t) = 3 , \theta(t) = |\ln t|^\alpha,$ and $ \psi(t) = {|\ln t|^{1-\alpha}}$.
	\end{exmp}
		
	\begin{exmp} \label{ex2}
		$a(t,x) = \left( \ln(1+1/t) \right)^4$. Here $\om=\P=1, \tilde{ \theta}(t) = (\ln(1+1/t))^4, \theta(t) = (\ln(1+1/t))^3,$ and $ \psi(t) = 1$.
	\end{exmp}

	\begin{exmp} \label{ex3}
		$a(t,x) =9\japx^2  (2+ t\sin(1/t))$. Here $\om = 3 \japx, \P = \japx, \tilde{\theta}(t) =  \theta(t) =1 $ and $\psi(t) =\ln \left(1+ \frac{1}{t}\right) $.
	\end{exmp}
	At this juncture it is important to relate the conditions (\ref{sing1}) and (\ref{sing2}) to the oscillatory behavior studied in the literature, see for example \cite{CSR,Reis,KuboReis}.
    \begin{defn}[Oscillatory Behavior]\label{osci}
		Let $c=c(t) \in L^{\infty}([0,T]) \cap C^2((0,T])$ satisfy the estimate
		\begin{equation}\label{oscill}
			\left\vert \frac{d^j}{dt^j} c(t) \right\vert \leq C_j \left( \frac{1}{t} \Big(  \ln (1+ 1/t)  \Big)^{\gamma} \right)^j, \quad j=1,2.
		\end{equation}
		We say that the oscillating behavior of the function $c$ is
		\begin{itemize}
			\item very slow if $\gamma = 0$
			
			\item slow if $\gamma  \in(0,1)$
			
			\item fast if $\gamma = 1$
			
			\item very fast if (\ref{oscill}) is not satisfied for $\gamma =1.$ 
		\end{itemize} 
	\end{defn}
	\begin{exmp}
		If $c(t) = 2+ \sin \left( \ln \frac{1}{t} \right)^\alpha$ then the oscillations are very slow, slow, fast and very fast if $\alpha \leq 1, \alpha \in (1,2), \alpha=2$ and $\alpha>2,$ respectively.
	\end{exmp}
	
	We see that the estimate in (\ref{oscill}) characterizing the oscillatory behavior is a specific case of the singular behavior considered in this paper. It corresponds to the case $\tilde{\theta}(t) = 1, \theta(t)= \left(\ln(1+1/t) \right)^{\gamma}$ and $\psi(t) = 1$ in (\ref{a})-(\ref{sing2}).
	
	In \cite{Cico1}, Cicognani discussed the well-posedness of (\ref{PO}) when $a(t,x)$ is uniformly bounded in $x$ and satisfies
	\begin{equation}\label{sing3}
		\vert \partial_ta(t,x) \vert \leq \frac{C}{t^q}, \quad q \geq 1, \; (t,x) \in (0,T] \times \R^n ,
	\end{equation}
	without any further conditions on the second time derivative. When $q=1,$ this corresponds to the case $\om=\P=1, \theta(t) = 1$ and $\tilde{\theta}(t) = \ln(1+1/t)$ in (\ref{ell}-\ref{sing1}). The coefficient $a(t,x)$ is allowed to have logarithmic blow-up near $t=0$ for the case  $q=1$. The author reports $C^\infty$-well-posedness when $q=1$ and Gevrey-well-posedness when $q>1$ for 
	the Cauchy problem (\ref{PO}) with a finite and infinite loss of derivatives, respectively. These results have been extended by the authors to a global setting: the case $q=1$ in \cite{RahulNUK3} and $1<q<\frac{3}{2}$ in \cite{Rahul_NUK2}.
	
	In \cite[Theorem 8]{Reis}, Reissig considered the coefficients independent of $x$ and oscillating in $t$ i.e., $a=a(t)$. The author reports no loss, arbitrary small loss, finite loss and infinite loss of derivatives for the cases very slow, slow, fast and very fast oscillations, respectively. These results were {\itshape partially} extended to the case of coefficients depending on both $t$ and $x$ in \cite[Theorem 13]{Reis} and \cite[Theorem 1.2]{KuboReis} where the $C^\infty$ well-posedness is established through the construction of parametrix. But this extension requires the condition
	\begin{linenomath*}
		\[
		\vert \partial_t^j \partial_x^\beta a(t,x)\vert \leq C_{j,\beta} \bigg(\frac{1}{t}\bigg(\ln \frac{1}{t}\bigg)^\gamma\bigg)^j , \quad \text{ for all } j \in \mathbb{N}_0,\; \beta \in \mathbb{N}_0^n
		\]
	\end{linenomath*}
	for $\gamma \in [0,1].$ The second author along with Coriasco and Battisti extended these results to the SG setting in \cite{NUKCori1}. The authors report well-posedness in $\mathcal{S}'(\R^n)$ under the Levi conditions for the weakly hyperbolic Cauchy problem (\ref{PO}) with the coefficient $a(t,x)$ satisfying
	\begin{linenomath*}
		\[
		\vert \partial_t^j \partial_x^\beta a(t,x)\vert \leq C_{j,\beta} \bigg(\frac{1}{t}\bigg(\ln \frac{1}{t}\bigg)^\gamma\bigg)^j \japx^{2-|\beta|}, \quad \text{ for all } j \in \mathbb{N}_0,\; \beta \in \mathbb{N}_0^n.
		\]
	\end{linenomath*}
	We note here that these extensions partially settle the well-posedness issue for oscillatory behavior case as they require the coefficients to be bounded in $t$ and possess all the $t$-derivatives in $(0,T].$ 
	
	In our work, we allow the coefficients to be unbounded near $t=0$ and consider singular behavior specified by just the first and second $t$-derivatives of the coefficients as evident from the conditions (\ref{a}-\ref{sing2}). The main results Theorem \ref{result1} and Theorem \ref{result3} require no further smoothness in $t$-variable. Contributions of this paper are as follows
	\begin{enumerate}
		\item Analyze the effect of blow-up of coefficients near $t=0$ (that is the case when $\tilde\theta$ is unbounded) on the regularity of the solution. We consider coefficients blowing up atmost as $|\ln t|^r, r {\geq 1}.$ Note that these class of problems have not been addressed in the literature. Only the case of $r=1$ is considered in \cite{Cico1}.
		
		\item Analyze the finite loss of regularity when 
		\begin{equation}\label{ineq}
			\vartheta(1/t) \leq C_0^* \ln \left( 1 + \frac{1}{t} \right)
		\end{equation}
		This bridges the gap between \cite[Theorem 2]{Cico1} (i.e. the case $q=1$ in (\ref{sing3})) and \cite[Theorem 8]{Reis} in a {\itshape{global setting}} allowing logarithmic blow-up at $t=0$ and precisely quantify the loss of regularity (both decay and derivatives). Incidentally, this completely settles the well-posedness issue for the oscillatory behavior ($\gamma \in [0,1]$ in (\ref{oscill})) case.
		
		\item Analyze the infinite loss of regularity when (\ref{ineq}) is violated i.e.,
			\begin{equation}\label{ineq2}
				\begin{rcases}
				\begin{aligned}
					\lim_{t \to 0^+} \frac{\vartheta(1/t)}{\vert \ln t \vert} &= \infty, \quad \text{ and }  \\
					C_1^* \left(\ln \left( 1 + \frac{1}{t} \right) \right)^{\varrho_1} \leq \vartheta\left( \frac{1}{t} \right) &\leq C_2^* \left(\ln \left( 1 + \frac{1}{t} \right) \right)^{ \varrho_2}, 
				\end{aligned}
			\end{rcases}
			\end{equation}
			for some $1<\varrho_1 \leq \varrho_2.$ When the coefficients are independent of $x$, Colombini et al. \cite[Theorem 1.2 and 1.4]{CSR} have shown that one can not expect $C^\infty$ well-posedness in the case of {\itshape{very fast oscillations}} ($\gamma>1$ in (\ref{oscill})). In this paper, we quantify the infinite loss of regularity using infinite order pseudodifferential operators.
			
		\item Derive an optimal cone condition (Theorem \ref{result2}) for the solution of the Cauchy problem (\ref{eq1}) from the energy estimates derived in Theorem \ref{result1} and \ref{result3}. The $L^1$ integrability of logarithmic order singularity in (\ref{ineq2}) guarantees that the propagation speed is finite. We report that even the weight function governing the coefficients influences the geometry of the slope of the cone in such a manner that the slope grows as $|x|$ grows.
		
		\item Show that the set of coefficients which result in infinite loss of regularity is nonempty and residual in certain metric space by extending the work of Ghisi and Gobbino \cite[Section 4]{GG} to a global setting. The ideas here rely on  the spectral theorem for pseudodifferential operators on $\R^n$ and Baire category arguments.
	\end{enumerate}

	Our tool kit consists of metrics on the phase space, a localization technique in the extended phase space and the corresponding parameter depedent global symbol classes.  We study a strictly hyperbolic Cauchy problem using a metric of the form
	\begin{equation}\label{go}
		g_{\Phi,k} = \P^{-2} \vert dx\vert^2 + \japxik^{-2}\vert d\xi\vert^2.
	\end{equation}  
	Our localization technique relies upon the Planck function $\h=(\P\japxik)^{-1}$(see Section \ref{metric} for the definition) associated to the metric $g_{\Phi,k}$ for a subdivision of the extended phase space (see Section \ref{zones}). 
	
	To microlocally compensate the loss of regularity, we conjugate a first order system associated to the operator $P$ in (\ref{eq1}) by a pseudodifferential operator of the form
	\begin{equation}\label{infpdo}
		e^{\nu(t)\Theta(x,D_x)},
	\end{equation}
	where $\nu \in C([0,T]) \cap C^1((0,T])$ and $\Theta(x,\xi)=\sigma(\Theta(x,D_x))$ is defined as
	\begin{equation}\label{Theta}
		\Theta(x,\xi) := \vartheta(\P\japxik)= \theta\big(\h\big)\big(\tilde\theta(\h) + \psi(\h)\big)
	\end{equation}
	This is discussed in Section \ref{CONJ}. When (\ref{ineq2}) is satisfied, the operator in (\ref{infpdo}) is of infinite order in both $x$ and $D_x$.
	The operator $\Theta(x,D_x)$ explains the quantity of the loss by linking it to the metric on the phase space and the singular behavior while $\nu(t)$ gives a scale for the loss.
	We call the conjugating operator as {\itshape{loss operator}}.
	The symbol of the operator arising after the conjugation is governed by a metric $\gt$ that is conformally equivalent to the initial metric $\go$. The metric $\gt$ is of the form
	\begin{equation}\label{gt}
		\gt = \T^{2} \go.
	\end{equation} 
	In Section \ref{sb}, we define Sobolev spaces using the loss operator and depending  on its order we report that the solution experiences zero, arbitrarily small, finite or infinite loss in relation to the initial datum defined in the Sobolev spaces.

	\section{Tools}
	
	In this section we introduce the main tools of the paper. The first part is devoted to a class of metrics on the phase space that govern the geometry of the symbols in our consideration. We then introduce Sobolev spaces related to the metrics and the order of singularity. In the following subsection we device a localization technique based on the Planck function associated to the metric and finally we introduce the parameter dependent global symbol classes.
	
	\subsection{Our Choice of Metric on the Phase Space}\label{metric}
	The metrics we consider are of the form given in (\ref{go}) and (\ref{gt}) whose nature is essentially dependent on the properties of $\Phi.$ Before we discuss the structure of $\Phi$, let us review some notation and terminology used in the study of metrics on the phase space, see \cite[Chapter 2]{Lern} for details. Let us denote by $\sigma(X,Y)$ the standard symplectic form on $T^*\R^n\cong \R^{2n}$: if $X=(x,\xi)$ and $Y=(y,\eta)$, then
	\[
		\sigma(X,Y)=\xi \cdot y - \eta \cdot x.
	\]	
	We can identify $\sigma$ with the isomorphism of $\R^{2n}$ to $\R^{2n}$ such that $\sigma^*=-\sigma$, with the formula $\sigma(X,Y)= \langle \sigma X,Y\rangle$. Consider a Riemannian metric $g_X$ on $\R^{2n}$ (which is a measurable function of $X$) to which we associate the dual metric $g_X^\sigma$ given by
	\[ 
	g_X^\sigma(Y)= \sup_{0 \neq Y' \in \R^{2n}} \frac{\langle \sigma Y,Y'\rangle^2}{g_X(Y')}, \quad \text{ for all } Y \in \R^{2n}.
	\]
	
	Considering $g_X$ as a matrix associated to positive definite quadratic form on $\R^{2n}$, we have $g_X^\sigma=\sigma^*g_X^{-1}\sigma$.
One defines the Planck function \cite{nicRodi}, that plays a crucial role in the development of pseudodifferential calculus to be
	\[ 
		h_g(x,\xi) := \sup_{0\neq Y \in \R^{2n}} \Bigg(\frac{g_X(Y)}{g_X^\sigma(Y)}\Bigg)^{1/2}.
	\]
	 The uncertainty principle is quantified as the upper bound $h_g(x,\xi)\leq 1$. In the following, we make use of the strong uncertainty principle, that is, for some $\kappa>0$,
	 \[
	 	h_g(x,\xi) \leq (1+|x|+|\xi|)^{-\kappa}, \quad (x,\xi)\in \R^{2n}.
	 \]
	 The Planck function associated to the metric $\go$ is $\h = (\P\japxik)^{-1}$ while the one to the metric $\gt$ is $\tilde{h}(x,\xi) = \T(\P\japxik)^{-1}.$
	 
	 Basically, a pseudodifferential calculus is the datum of the metric satisfying some local and global conditions. In our case, it amounts to the conditions on $\P$. The symplectic structure and the uncertainty principle also play a natural role in the constraints imposed on $\Phi$. So we consider $\P=\Phi(|x|)$ to be a monotone increasing function of $|x|$ satisfying following conditions: 

	\begin{alignat*}{3}
		1 \; \leq & \quad \Phi(x) &&\lesssim  1+|x| && \quad \text{(sub-linear)} \\
		\vert x-y \vert \; \leq & \quad r\Phi(y) && \implies C^{-1}\Phi(y)\leq \Phi(x) \leq C \Phi(y)  && \quad \text{(slowly varying)} \\
		&\Phi(x+y) && \lesssim  \Phi(x)(1+|y|)^s && \quad \text{(temperate)}
	\end{alignat*}
	for all $x,y\in\R^n$ and for some $r,s,C>0$. Note that $C \geq 1$ in the slowly varying condition with $x=y$. 
	
	For the sake of calculations arising in the development of symbol calculus related to metrics $\go$ and $\gt$, we need to impose following additional conditions:
	\begin{align*}
		 \Phi(x+y) & \leq \Phi(x) + \Phi(y),   \qquad  (\text{Subadditive})\\
		|\partial_x^\beta \Phi(x)| & \lesssim \Phi(x) \japx^{-|\beta|}, \\
		\Phi(ax) &\leq \P,
	\end{align*}
	where $\beta \in \mathbb{Z}_+^n$ and $a \in [0,1]$. It can be observed that the above conditions are quite natural in the context of symbol classes. In our work, we need even the weight function $\om$ to satisfy the above properties given for $\Phi.$ In arriving at an energy estimate using the Sharp G\r{a}rding inequality (see Section \ref{Energy1} for details), we need 
	\begin{linenomath*}
		\[
		\om \lesssim \P, \quad x \in \R^n.
		\] 
	\end{linenomath*}

	\subsection{Sobolev Spaces}\label{sb}
	We now introduce the Sobolev spaces related to the metric $g_{\Phi,k}$.
	\begin{defn} \label{Sobol}   
		The Sobolev space $\sobol(s,\delta)$ for $s=(s_1,s_2) \in \R^2$ and $\delta \in \R$ is defined as
		\begin{linenomath*}
			\begin{equation}
				\label{Sobo3}
				\sobol(s,\delta) = \{v \in L^2(\R^n): e^{\delta\Theta(x,D_x)}\P^{s_2}\la D_x \rak^{s_1}v \in L^2(\R^{n}) \},
			\end{equation} 
		\end{linenomath*}
		equipped with the norm
		$
		\Vert v \Vl(s,\delta) = \Vert e^{\delta\Theta(\cdot,D)}\Phi(\cdot)^{s_2}\la D \rak^{s_1}v \Vert_{L^2} .
		$ 
	\end{defn}
	Here the operator $\Theta(x,D_x)$ is as in (\ref{Theta}). When $\vartheta$ satisfies the estimate in (\ref{ineq}), then the operator $e^{\delta\Theta(x,D_x)}$ is a finite order pseudodifferential operator. In that case the Sobolev spaces $\sobol(s,\delta)$ are of the form given by the following definition.
	\begin{defn} \label{Sobo}   
		The Sobolev space $\sobo(s)$ for $s=(s_1,s_2) \in \R^2$ is defined as
		\begin{linenomath*}
			\begin{equation}
				\label{Sobo2}
				\sobo(s) = \{v \in L^2(\R^n): \P^{s_2}\langle D_x \rak^{s_1}v \in L^2(\R^{n}) \},
			\end{equation} 
		\end{linenomath*}
		equipped with the norm
		$
		\Vert v \V(s) = \Vert \Phi(\cdot)^{s_2}\la D \rak^{s_1}v \Vert_{L^2} .
		$ 
		\end{defn}
		
		\begin{rmk}\label{sobormk}(Relation between $\sobol(s,\delta)$ and $\sobo(s)$)
			\begin{enumerate}
				\item If  $\vartheta$ is a bounded function then we have the equivalence $	\sobo(s) \equiv \sobol(s,\delta),$ as $\T$ is a bounded function in both $x$ and $\xi.$
				\item If $\lim\limits_{t \to 0^+} \frac{\vartheta(1/t)}{|\ln t|} = 0$, $	\sobo(s+\varepsilon e) \subseteq \sobol(s,\delta) \subseteq \sobo(s-\varepsilon e)$ for every $\varepsilon>0$.
				\item If $\vartheta(1/t) \equiv C_0 \ln(1+1/t)$ for some $C_0>0$, then $
				\sobo(s+C_0\delta e) \equiv \sobol(s,\delta).$ Here
				$\Theta(x,\xi) = C_0\ln (1+\P\japxik)$. 				
			\end{enumerate}
		\end{rmk}		
		
	\subsection{Subdivision of the Phase Space}\label{zones}
	One of the main tools in our analysis is the division of the extended phase space, $J=\J$, into three regions using the Planck function, $h(x,\xi)=(\P \japxik)^{-1}$, and the functions $\tilde\theta, \theta$ and $\psi$ which specify the order of singularity. We use these regions in the proof of Theorem \ref{result1} and Theorem \ref{result3} (see Section \ref{Proof1}) to handle the low regularity in $t$. We define $t_{x,\xi}$ and $\tilde t_{x,\xi}$ for a fixed $(x,\xi)$ as 
	\begin{linenomath*}
		\begin{align*}
			t_{x,\xi}&=N\:h(x,\xi) \zpi , \qquad \text{ and }\\
			\tilde t_{x,\xi}&=N\:h(x,\xi) \zpii,
		\end{align*}
	\end{linenomath*}
	where $N$ is the positive integer. For a fixed $(x,\xi)$ we split the time interval as 
	\[
		[0,T] = [0,t_{x,\xi}] \cup (t_{x,\xi}, \tilde t_{x,\xi}] \cup (\tilde t_{x,\xi}, T].
	\]
	We define the regions as below:
	\begin{linenomath*}
		\begin{align*}
			\label{zone1}
			\zint(N)&=\{(t,x,\xi)\in J : 0 \leq t \leq t_{x,\xi}\}, \\
			\zmid(N) &=\{(t,x,\xi)\in J : t_{x,\xi} < t \leq \tilde t_{x,\xi}\}, \\
			\zext(N) &=\{(t,x,\xi)\in J : \tilde t_{x,\xi} < t \}.
		\end{align*}
	\end{linenomath*}
	Note that for all $(x,\xi) \in \R^{2n}$, 
	\begin{equation}
		\label{zpts}
		\h \leq t_{x,\xi} \leq \tilde t_{x,\xi} \leq T.
	\end{equation}
	\begin{rmk}
		When $\tilde \theta(t)$ and $\psi(t)$ are bounded functions we see that $t_{x,\xi} \sim \tilde t_{x,\xi}.$ In such a case, given $(x,\xi)$ we split the time interval and the extended phase space as
		\[
		\begin{aligned}
			[0,T] &= [0,t_{x,\xi}] \cup ( t_{x,\xi}, T],\\
			J &= \zint(N) \cup \zext(N).
		\end{aligned}
		\]
		We do not need the region $\zmid(N).$ 
		An example of such a case is the oscillatory behavior (including very fast oscillations) i.e., $\theta(t) \sim (\ln (1+1/t))^{\gamma},   \tilde{\theta}(t) = \psi(t) = 1, \gamma \in [0,+\infty).$ In \cite{CSR} where fast oscillating coefficients depending only on $t$ are dealt, the authors subdivide the extended phase space into two regions - $\zint(N)$ and $\zext(N)$ using the time splitting point $t_{x,\xi}.$ Similar is the case in \cite{Cico1,RahulNUK3} where $\theta(t) = 1$ and there is no condition on the second time derivative of the coefficients.
	\end{rmk}
	
	\subsection{Parameter Dependent Global Symbol Classes}	\label{Symbol classes}
	In this section, we define parameter dependent global symbol classes whose geometry is governed by the metrics $\go$ and $\gt.$ Let $m_j \in \R$ for $j=1,\dots,6.$
	\begin{defn}
		$\G(m_1,m_2)$ is the space of all functions $a=a(x,\xi)\in C^\infty( \R^{2n})$ satisfying the symbolic estimate
		\begin{linenomath*}
			\begin{equation}
				\vert \partial_\xi^\alpha D_x^\beta a(x,\xi) \vert 
				\leq C_{\alpha \beta}   \japxik^{m_1-|\alpha|} \om^{m_2} \Phi(x)^{-|\beta|}	,
			\end{equation}
		\end{linenomath*}	
		for constants $C_{\alpha \beta}>0$ and all $\alpha,\beta \in \mathbb{N}^n_0$.
	\end{defn}
	We denote by $G^{-\infty}$ the class of symbols in
	\[
	\bigcap_{m_1,m_2 \in \R} \G(m_1,m_2).
	\]
	Note that $C_{\alpha\beta}(>0)$ is a generic constant which may vary from equation to equation.
	\begin{defn}\label{def2}
		$\GT(m_1,m_2)$ is the space of all functions $a=a(x,\xi)\in C^\infty( \R^{2n})$ satisfying the symbolic estimate
		\begin{linenomath*}
			\begin{equation}
				\vert \partial_\xi^\alpha D_x^\beta a(x,\xi) \vert 
				\leq C_{\alpha \beta}   \japxik^{m_1-|\alpha|} \om^{m_2} \Phi(x)^{-|\beta|} \T^{(\vert \alpha\vert+\vert \beta \vert)}	,
			\end{equation}
		\end{linenomath*}	
		for constants $C_{\alpha \beta}>0$ and all $\alpha,\beta \in \mathbb{N}^n_0$.
	\end{defn}
	\begin{defn}
		$\Gint(m_1,m_2,N,m_3)$ is the space of all functions $a(t,x,\xi)$ in \\$ C^2((0,T];C^\infty(\mathbb{R}^{2n}))$ satisfying 
		\begin{linenomath*}
			\begin{equation}
				\label{symest1}
				|\partial_\xi^\alpha  D_x^\beta a(t,x,\xi)| \leq C_{\alpha\beta} \japxik^{m_1-|\alpha|} \om^{m_2}\P^{-|\beta|} \tilde \theta(t)^{m_3}
			\end{equation}	
		\end{linenomath*}	
		for constants $C_{\alpha \beta}>0$ and for all $(t,x,\xi) \in \zint(N) $ and all $\alpha,\beta \in \mathbb{N}^n_0$.	
	\end{defn}
	
	\begin{defn}
		$\Gmid(m_1,m_2,m_3,N,m_4,m_5)$ is the space of all functions $a(t,x,\xi)\in C^2((0,T]; C^\infty(\mathbb{R}^{2n}))$ satisfying the symbolic estimate
		\begin{linenomath*}
			\begin{equation}\label{symest2}
				\vert \partial_\xi^\alpha D_x^\beta a(t,x,\xi) \vert 
				\leq C_{\alpha \beta}  \japxik^{m_1-|\alpha|} \om^{m_2}\P^{-|\beta|}		
				\bigg( \frac{\theta(t)}{t}\bigg)^{m_3} \tilde \theta(t)^{m_4+m_5(|\alpha| + |\beta|)}.
			\end{equation}
		\end{linenomath*}	
		for constants $C_{\alpha \beta}>0$ and for all $(t,x,\xi) \in \zmid(N)$ and all $\alpha,\beta \in \mathbb{N}^n_0$.
	\end{defn}

	\begin{defn}
		$\Gext(m_1,m_2,m_3,m_4,N,m_5,m_6)$ for $m_3 \geq m_4,$ is the space of all functions $a(t,x,\xi)\in C^2((0,T] ; C^\infty(\mathbb{R}^{2n}))$ satisfying the symbolic estimate
		\begin{linenomath*}
			\begin{equation}\label{symest3}
				\vert \partial_\xi^\alpha D_x^\beta a(t,x,\xi) \vert 
				\leq C_{\alpha \beta} \japxik^{m_1-|\alpha|} \om^{m_2}\P^{-|\beta|}		
				\bigg( \frac{\theta(t)}{t}\bigg)^{m_3}  e^{m_4\psi(t)} \tilde \theta(t)^{m_5+m_6(|\alpha| + |\beta|)}.
			\end{equation}
		\end{linenomath*}	
		for constants $C_{\alpha \beta}>0$ and for all $(t,x,\xi) \in \zext(N) $ and all $\alpha,\beta \in \mathbb{N}^n_0$.
	\end{defn}	
	Similar to Definition \ref{def2}, we can define $\GTint(m_1,m_2,N,m_3),$\\$ \GTmid(m_1,m_2,m_3,N,m_4,m_5)$ and $\GText(m_1,m_2,m_3,m_4,N,m_5,m_6).$
	\begin{rmk}
	When $\tilde\theta(t)$ is a bounded function, we have
	\begin{equation}\label{inclu}
		\begin{rcases}
		\begin{aligned}
			\Gint(m_1,m_2,N,m_3) &\equiv \Gint(m_1,m_2,N,0),\\
			\Gmid(m_1,m_2,m_3,N,m_4,m_5) &\equiv  \Gmid(m_1,m_2,m_3,N,0,0), \\
			\Gext(m_1,m_2,m_3,m_4,N,m_5,m_6)&\equiv  \Gext(m_1,m_2,m_3,m_4,N,0,0).\\
		\end{aligned}
		\end{rcases}
	\end{equation}
	\end{rmk}
	Given a $t$-dependent global symbol $a(t,x,\xi)$, we can associate a pseudodifferential operator $Op(a)=a(t,x,D_x)$ to $a(t,x,\xi)$ by the following oscillatory integral
	\begin{align*}
		a(t,x,D_x)u(t,x)& =\iint\limits_{\mathbb{R}^{2n}}e^{i(x-y)\cdot\xi}a(t,x,\xi){u}(t,y)dyd \xi\\
		& = \int\limits_{\mathbb{R}^n}e^{ix\cdot\xi}a(t,x,\xi)\hat{u}(t,\xi) \textit{\dj}\xi.
	\end{align*}
	where $\textit{\dj}\xi = (2\pi)^{-n}d \xi.$
	
	As for the calculus of symbol classes $G^{m_1,m_2}(\omega,g_{\Phi,k})$, we refer to \cite[Section 1.2 \& 3.1]{nicRodi}. The calculus for the operators with symbols in the additive form 
	\begin{linenomath*}
		\begin{equation}\label{symadd}
			\begin{rcases}
		\begin{aligned}
			a(&t,x,\xi) = a_1\var + a_2\var + a_3\var, \quad \text{for} \quad \\
			a_1 &\in \Gint(\tilde m_1,\tilde m_2,N,\tilde m_3), \\
			a_2 & \in \Gmid(m_1',m_2',m_3',N,m_4',m_5')\\
			a_3 & \in \Gext(m_1,m_2,m_3,m_4,N,m_5,m_6)
		\end{aligned}
		\end{rcases}
		\end{equation}
	\end{linenomath*}
	is given in Appendix I of this paper. The calculus for the $\gt$ versions of the symbol classes follows in similar lines. This requires that the function $\T$ is sub-additive and sub-multiplicative both $x$ and $\xi$ variables separately i.e.,
		\begin{alignat*}{3}
			\Theta(x+y,\xi) &\leq C\big( \Theta(x,\xi) + \Theta(y,\xi) \big), \qquad && \Theta(x+y,\xi) &&\leq C \Theta(x,\xi)  \Theta(y,\xi), \\
			\Theta(x,\xi+\eta) &\leq C\big( \Theta(x,\xi) + \Theta(x,\eta) \big), && \Theta(x,\xi+\eta) &&\leq C \Theta(x,\xi)  \Theta(x,\eta).
		\end{alignat*}
	In fact, the sub-multiplicative property can be derived from the sub-additivity as $\Theta \geq1.$
	 
	While dealing with the case of infinite order loss, we need to keep track of the weight sequences with respect to both $x$ and $\xi$. To this end we replace the generic constant $C_{\alpha\beta}$ by $CC'_{|\alpha|}K'_{|\beta|}$ such that
	\begin{equation}\label{wgtseq}
		\inf_{j \in \mathbb{N}} \frac{C'_jK_j'}{(\P\japxik)^j} \lesssim e^{-\delta_0\T}, \quad \delta_0>0.
	\end{equation}
	The calculus of the operators with symbols governed by such weight sequences can be developed in similar lines to the calculus given in the Appendix I of this paper along with the standard techniques from the book \cite[Section 6.3]{nicRodi}, \cite[Appendix I \& II]{Rahul_NUK2} and \cite[Appendix]{AscaCappi2}.

		\section{Statements of the Main Results}  \label{stmt}
	Before we state the main result of this paper, we introduce our model strictly hyperbolic equation. 
	We study the Cauchy problem 
	\begin{linenomath*}
		\begin{equation}
			\begin{cases}
				\label{eq1}
				P(t,x,D_t,D_x)u(t,x)= f(t,x), \qquad (t,x) \in (0,T] \times \R^n, \\
				u(0,x)=f_1(x), \quad \partial_tu(0,x)=f_2(x)
			\end{cases}
		\end{equation}
	\end{linenomath*}
	with the strictly hyperbolic operator $P(t,x,D_{t},D_{x}) = -D_t^2 + a(t,x,D_x)+ b(t,x,D_x)$ where
	\begin{linenomath*}			
		\[
		a(t,x,\xi)  = \sum_{j,l=1}^{n} a_{j,l}(t,x)\xi_j\xi_l \quad \text{ and } \quad
		b(t,x,\xi)  = i\sum_{j=1}^{n} b_{j}(t,x)\xi_j + b_{n+1}(t,x).
		\]
	\end{linenomath*}
	The matrix $(a_{j,l}(t,x))$ is real symmetric with $a_{j,l} \in C^2((0,T];C^\infty(\R^n))$ and $b_j \in C([0,T];C^\infty(\R^n))$. 
	The assumptions on $P$ are as follows
	\begin{linenomath*}
		\begin{align}
			\label{elli}
			a(t,x,\xi) &\geq C_0 \japxik^2 \om^2, \\
			\label{Lower}
			\vert \partial_\xi^\alpha \partial_x^\beta b(t,x,\xi) \vert &\leq C_{\alpha\beta} \japxik^{1-\vert \alpha \vert} \om \P^{-\vert \beta \vert},
		\end{align}
		for $(t,x,\xi) \in [0,T] \times \R^n\times \R^n$ and 
		\begin{align}
			\label{bound}
			\vert \partial_\xi^\alpha \partial_x^\beta a(t,x,\xi) \vert &\leq C_{\alpha\beta}   \japxik^{2-\vert \alpha \vert}  \om^2 \P^{-\vert \beta \vert} \tilde{\theta}(t), \\
			\label{B-up1}
			\vert \partial_\xi^\alpha \partial_x^\beta \partial_t a(t,x,\xi) \vert & \leq C_{\alpha\beta} \japxik^{2-\vert \alpha \vert} \om^2 \P^{-\vert \beta \vert} \frac{\theta(t)}{t},\\
			\label{B-up2}
			\vert \partial_\xi^\alpha \partial_x^\beta \partial_t^2 a(t,x,\xi) \vert & \leq C_{\alpha\beta} \japxik^{2-\vert \alpha \vert} \om^2 \P^{-\vert \beta \vert} \frac{\theta(t)^2}{t^2} e^{\psi(t)},
		\end{align}
	\end{linenomath*}		
	for $(t,x,\xi) \in (0,T] \times \R^n\times \R^n, \alpha,\beta \in \mathbb{N}_0^n$.

	We now state the main results of this paper whose proof is presented in Section \ref{Proof1}.
	Let $e=(1,1)$. 
	
	\begin{thm}[zero, arbitrarily small or finite loss]\label{result1}
		Consider the strictly hyperbolic Cauchy problem  (\ref{eq1}) satisfying the conditions (\ref{elli}) - (\ref{B-up2}) and (\ref{ineq}). Let the initial datum $f_j$ belong to $\sobo(s+{(2-j)}e) $, $j=1,2$ and the right hand side $f \in C([0,T];\sobo(s))$.
		Then, for every $\varepsilon \in (0,1)$ there exist $\kappa_0,\kappa_1 >0$ such that for every $s \in \R^2$ there is a unique global in time solution
		\begin{linenomath*}
			\[
				u \in \bigcap\limits_{j=0}^{1}C^{1-j}\Big([0,T];\sobol(s+je,{-\kappa(t)}) \Big),
			\]
		\end{linenomath*}
		where 
		\[
			\kappa(t)= \begin{cases}
			\begin{aligned}
				&\kappa_0+\kappa_1t^\varepsilon/\varepsilon, \quad \text{when } \tilde\theta \text{ is unbounded}\\
				& \kappa_0+\kappa_1t, \qquad \text{ otherwise.}
			\end{aligned}
			\end{cases}
		\]
		More specifically, the solution satisfies the a-priori estimate		
		\begin{linenomath*}
			\begin{equation}
				\begin{aligned}
					\label{est2}
					C\sum_{j=0}^{1} \Vert \partial_t^ju(t,\cdot) \Vl(s+{(1-j)}e,{-\kappa(t)}) \; &\leq \sum_{j=1}^{2} \Vert f_j\V(s+{(2-j)}e)\\
					&\qquad + \int_{0}^{t}\Vert f(\tau,\cdot)\Vl(s,-{\kappa(\tau)})\;d\tau
				\end{aligned}
			\end{equation}
		\end{linenomath*}
		for $0 \leq t \leq T, \; C=C_s>0$.				
	\end{thm}
	In view of the Remark \ref{sobormk}, we see that when $\theta(t), \tilde \theta(t)$ and $\psi(t)$ are all bounded i.e., $\T \sim 1$, we have no loss. When $\vartheta(1/t)  \sim \ln(1+1/t)$ i.e., $\T \sim \ln (1+\P\japxik)$, we have at most finite loss of regularity. In between both the cases, the loss is arbitrarily small. The above result not only extends \cite[Theorem 8]{Reis} to the case of coefficients depending on $x$ and unbounded in $t$ but also to a global setting and hence settles the well-posedness issue for the oscillatory behavior case ($\gamma \in[0,1]$ in Definition \ref{osci}).
	\begin{thm}[Infinite Loss]\label{result3}
		Consider the strictly hyperbolic Cauchy problem  (\ref{eq1}) satisfying the conditions (\ref{elli}) - (\ref{B-up2}) and (\ref{ineq2}) with generic constant $C_{\alpha\beta}$ replaced by $CC'_{|\alpha|} K'_{|\beta|}$ as in (\ref{wgtseq}). Let the initial datum $f_j$ belong to $\sobol(s+{(2-j)}e,\delta_1) $, $\delta_1>0, j=1,2$ and the right hand side $f \in C([0,T];\sobol(s,\delta_2)), \delta_2>0.$
		Then, for every $\varepsilon\in(0,1)$ there exist $\delta^*,\tilde\kappa_0,\tilde\kappa_1 >0$ such that for every $s \in \R^2$ there exists a unique solution
		\begin{linenomath*}
			\[
			u \in \bigcap\limits_{j=0}^{1}C^{1-j}\Big([0,T^*];\sobol(s+je,{\tilde\kappa(t)})) \Big),
			\]
		\end{linenomath*}
		where 
			\[
		\tilde \kappa(t)= \begin{cases}
			\begin{aligned}
				&\tilde\kappa_0 + \tilde\kappa_1 ({T_0}^\varepsilon-t^\varepsilon)/\varepsilon, \quad \text{when } \tilde\theta \text{ is unbounded}\\
				& \tilde\kappa_0 + \tilde\kappa_1 ({T_0}-t), \;\qquad \text{ otherwise,}
			\end{aligned}
		\end{cases}
		\]
		and $T_0=\min\{\delta^*,\delta_1,\delta_2\}$. More specifically, the solution satisfies the a-priori estimate		
		\begin{linenomath*}
			\begin{equation}
				\begin{aligned}
					\label{est}
					C\sum_{j=0}^{1} \Vert \partial_t^ju(t,\cdot) \Vl(s+{(1-j)}e,{\tilde\kappa(t)}) \; &\leq \sum_{j=1}^{2} \Vert f_j\Vl(s+{(2-j)}e,{\tilde\kappa(0)})\\
					&\qquad+ \int_{0}^{t}\Vert f(\tau,\cdot)\Vl(s,{\tilde\kappa(\tau)})\;d\tau
				\end{aligned}
			\end{equation}
		\end{linenomath*}
		for $0 \leq t \leq T^*, \; C>0$.		
	\end{thm}

	\section{Conjugation by an Infinite Order Pseudodifferential Operator}\label{CONJ}
		
		It is well-known in the theory of hyperbolic operators (see \cite{kn,CicoLor}) that the low-regularity in $t$ needs to be compensated by a higher regularity in the $x$ variable. This \enquote{balancing} operation is the key to understand the loss of regularity. It is clear (see Section \ref{Energy1} and $\ref{Energy2}$) that the conjugation is the operation that brings this balance and even encodes the quantity of the loss. 
		
		In the global setting, solutions experience an infinite loss of both derivatives and decay when (\ref{ineq2}) is satisfied. In order to overcome the difficulty of tracking a precise loss in our context we introduce a class of parameter dependent infinite order pseudodifferential operators of the form $e^{\nu(t) \Theta(x,D_x)}$ for the purpose of conjugation. These operators compensate, microlocally, the loss of regularity of the solutions. The operator $\Theta(x,D_x)$ explains the compensation for the singularity in $t$ and decides the quantity of the loss while the monotone continuous function $\nu(t)$ gives a scale for the loss. Hence, we call the conjugating operator as the \enquote{loss operator}. The form of the $\nu(t)$ is dealt in Sections \ref{Energy1} and \ref{Energy2}. 		
		The operators with loss of regularity (both derivatives and decay) are transformed to \enquote{good} operators by conjugation with such an infinite order operator. This helps us to derive \enquote{good} a priori estimates of solutions in the Sobolev space associated with the loss operator.
		
		In this section, we perform a conjugation of operators with symbols of the form (\ref{symadd}) by $e^{\nu(t) \Theta(x,D_x)}.$ Here we assume that $\nu(t)$ is a  continuous function for $t \in [0,T]$. When $e^{\nu(t) \Theta(x,D_x)}$ is an infinite order pseudodifferential operator, we need to consider an appropriate weight sequence so that the conjugation is well-defined. For this reason one can replace the generic contant $C_{\alpha\beta}$ appearing in the definitions of the symbols with $CC'_{|\alpha|} K'_{|\beta|}$ satisfying the condition (\ref{wgtseq}).

		The following proposition gives  an upper bound on the function $\nu(t)$ for the conjugation to be well defined.

	\begin{thm}\label{conju}
		Consider a symbol $a\var$ as in (\ref{symadd}) where the generic constant $C_{\alpha\beta}$ in the symbol estimates (\ref{symest1})- (\ref{symest3}) is replaced by $CC'_{|\alpha|} K'_{|\beta|}$ satisfying the condition (\ref{wgtseq}). Let $\nu = \nu(t) \in C([0,T]) \cap C^1((0,T])$. Then, there exists $\delta^*>0$ such that for $\nu(t) > 0$ with  $\nu(t)<\delta^*$,
		\begin{linenomath*}
			\begin{align}
				e^{\nu(t) \Theta(x,D_x)}a(t,x,D_x) e^{-\nu(t) \Theta(x,D_x)}  = a(t,x,D)+ \sum_{j=1}^{3} r_{\nu}^{(j)}(t,x,D_x),
			\end{align}
		\end{linenomath*}
		where $r_{\nu}^{(j)}(t,x,D_x),j=1,2,3,$ are such that 
		\begin{align*}
			\begin{rcases}
				\begin{aligned}
				\Theta(x,\xi)^{-1} r_{\nu}^{(1)}\var &\in L^{\infty}\left([0, T]; G^{-\infty,l_2^*-1}(\Phi,\gt) \right) \\
				\Theta(x,\xi)^{-1}r_{\nu}^{(2)} \var &\in L^{\infty}\left([0, T]; G^{l_1^*-1,-\infty}(\Phi,\gt)\right)
			\end{aligned}
		\end{rcases},&
		\quad \text{ if } \tilde\theta \text{ is bounded,}\\
			\begin{rcases}
				\begin{aligned}
					t^{1-\varepsilon}\Theta(x,\xi)^{-1}r_{\nu}^{(1)} \var &\in C\left([0, T]; G^{-\infty,l_2^*-1}(\Phi,\gt) \right)  \\
					t^{1-\varepsilon}\Theta(x,\xi)^{-1}r_{\nu}^{(2)} \var &\in C\left([0, T]; G^{l_1^*-1,-\infty}(\Phi,\gt) \right) 
				\end{aligned}
			\end{rcases},&
			\quad \text{ otherwise,}
		\end{align*}
		while $	\Theta(x,\xi)^{-1}r_{\nu}^{(3)}(t,x,\xi) \in L^{\infty}([0,T];G^{-\infty})$ for every $\varepsilon \in (0,1)$ and 
		$l_i^*= \max\{\tilde m_i,m_i'+m_3',m_i+m_3\}, i=1,2.$
	\end{thm}  
	
	To prove Theorem \ref{conju}, we need the following lemma, which can be
	given an inductive proof.
	
	\begin{lem}
		Let $\delta \neq 0$. Then, for every $\alpha,\beta \in \mathbb{Z}^n_+$, we have
		\begin{linenomath*}
			\[
			\partial_x^\beta\partial_\xi^\alpha e^{\delta \Theta(x,\xi)} \leq (C\delta)^{|\alpha|+|\beta|}  \alpha! \beta!  e^{\delta \Theta(x,\xi)} \P^{-|\beta|} \japxik^{-|\alpha|} \Theta(x,\xi)^{\vert \alpha \vert + \vert \beta \vert}.
			\]
		\end{linenomath*}
	\end{lem}
	
	\begin{proof}[Proof of Theorem \ref{conju}]
		Throughout this proof we write $\nu$ in place of $\nu(t)$ for the sake of simplicity of the notation. 
		Let $a_{\nu,\Theta}(t,x,\xi)$ be the symbol of the operator
		\begin{linenomath*}
			\[
			\exp\{\nu \Theta(x,D_x)\}a(t,x,D_x) \exp\{-\nu \Theta(x,D_x)\}.
			\] 
		\end{linenomath*}
		Then  $a_{\nu,\Theta}(t,x,\xi)$ can be written in the form of an oscillatory integral as follows:
		\begin{linenomath*}
			\begin{align}\label{conj}
				a_{\nu,\Theta}(t,x,\xi)= \int \dots \int e^{-iy\cdot\eta}e^{-iz\cdot\zeta}&e^{\nu \Theta(x,\xi+\zeta+\eta)} a(t,x+z,\xi+\eta) \\
				& \times e^{-\nu \Theta(x+y,\xi)}dz\textit{\dj}\zeta dy\textit{\dj}\eta, \nonumber
			\end{align}
		\end{linenomath*}
		Taylor expansions of $\exp\{\nu \T\}$ in the first and second variables, respectively, are
		\begin{linenomath*} 
			\begin{align*}
				e^{-\nu \Theta(x+y,\xi)} &= e^{-\nu \Theta(x,\xi)} + \sum_{j=1}^{n} \int_{0}^{1}y_j \partial_{w_j'}e^{-\nu \Theta(w',\xi)} \Big|_{w'=x+\theta_1y}d\theta_1, \text{ and }\\
				e^{\nu \Theta(x,\xi+\zeta+\eta)} &= e^{\nu \Theta(x,\xi)} + \sum_{i=1}^{n} \int_{0}^{1}(\zeta_i + \eta_i) \partial_{w_i}e^{\nu \Theta(x,w)} \Big|_{w=\xi+\theta_2(\eta+\zeta)}d\theta_2.		
			\end{align*}
		\end{linenomath*}
		We can write $a_{\nu, \Theta}$ as 
		\begin{linenomath*}
			\[
			a_{\nu, \Theta}(t,x,\xi)=a(t,x,\xi) + \sum_{l=1}^{3} r_{\nu}^{(l)}(t,x,\xi)\quad \text{ where }
			\]	
		\end{linenomath*}
		\begin{linenomath*}	 
			\[
			r_{\nu}^{(l)}(x,\xi) = \int \cdots \int e^{-iy\cdot\eta}e^{-iz\cdot\zeta} I_l a(t,x+z,\xi+\eta) dz\textit{\dj}\zeta dy\textit{\dj}\eta,
			\]	
		\end{linenomath*}	
		and $I_l$, $l=1,2,3$ are as follows:
		\begin{equation*}
			\begin{aligned}
			I_1 &= e^{\nu \Theta(x,\xi)} \sum_{j=1}^{n} \int_{0}^{1}y_j \partial_{w_j'}e^{-\nu \Theta(w',\xi)} \Big|_{w'=x+\theta_1y}d\theta_1,\\
			I_2 &=  e^{-\nu \Theta(x,\xi)}  \sum_{i=1}^{n} \int_{0}^{1}(\zeta_i+ \eta_i) \partial_{w_i}e^{\nu \Theta(x,w)} \Big|_{w=\xi+\theta_2(\zeta + \eta)}d\theta_2,\\
			I_3 &= 
			\Bigg(\sum_{i=1}^{n} \int_{0}^{1}(\zeta_i+\eta_i) \partial_{w_i}e^{\nu \Theta(x,w)} \Big|_{w=\xi+\theta_2(\zeta+\eta)}d\theta_2\Bigg)\\
			& \qquad \times \Bigg(\sum_{j=1}^{n} \int_{0}^{1}y_j \partial_{w'_j}e^{-\nu \Theta(w',\xi)}\Big|_{w'=x+\theta_1y}d\theta_1\Bigg).
			\end{aligned}
		\end{equation*}
		
		Denote the indicator functions for the regions $\zint(N)$, $\zmid(N)$ and $ \zext(N)$ by $\chi_1$, $\chi_2$ and  $\chi_3,$respectively. Let $m_1^* = \max\{\tilde m_1, m_1', m_1\}$ and $m_2^* = \max\{\tilde m_2,m_2',m_2\}$.
		
		We will now determine the growth estimate for $r^{(1)}_\nu(t,x,\xi)$ using integration by parts. For $\alpha, \beta,\kappa \in (\mathbb{Z}^+_0)^n$ and $l \in \mathbb{Z}^+$ we have			
		\begin{linenomath*}
			\[
			\begin{aligned}
				&\partial_\xi^\alpha\partial_x^\beta r^{(1)}_\nu(t,x,\xi)\\
				&= \sum_{j=1}^{n} \sum_{\tiny{\beta'+ \beta'' \leq \beta}}\sum_{\scriptsize{\alpha'+\alpha'' \leq \alpha}} \int \dots \int y^{-\kappa} y^{\kappa} e^{-iy\cdot\eta}e^{-iz\cdot\zeta} 
				(\partial_\xi^{\alpha'} \partial_x^{\beta'} D_{\xi_j}a)(t,x+z,\xi+\eta)\\	
				&\quad\times \int_{0}^{1} \partial_\xi^{\alpha''}\partial_x^{\beta''}\partial_{w'_j}	e^{\nu (\T- \Theta(w',\xi))} 	\Big|_{w'=x+\theta_1y} d\theta_1 dz\textit{\dj}\zeta dy\textit{\dj}\eta\\
				&= \sum_{j=1}^{n} \sum_{\tiny{\beta'+ \beta'' \leq \beta}}\sum_{\scriptsize{\alpha'+\alpha'' \leq \alpha}} \int \dots \int y^{-\kappa} e^{-iy\cdot\eta}e^{-iz\cdot\zeta} \la \eta \ra_k ^{-2l} \la y\rak^{-2l}\la z \rak ^{-2l}  \la D_\zeta\rak^{2l} \la \zeta \rak^{-2l}  \\
				&\quad\times \la D_\eta\rak^{2l}\la D_z\rak^{2l} D_\eta^\kappa (\partial_\xi^{\alpha'} \partial_x^{\beta'} D_{\xi_j}a)(t,x+z,\xi+\eta)\\	
				&\quad\times \int_{0}^{1} \la D_y \rak^{2l} \partial_\xi^{\alpha''}\partial_x^{\beta''}\partial_{w'_j}	e^{\nu (\T- \Theta(w',\xi))} 	\Big|_{w'=x+\theta_1y} d\theta_1 dz\textit{\dj}\zeta dy\textit{\dj}\eta.
			\end{aligned}
			\]
		\end{linenomath*}
		Using the easy to show inequality $\Phi(x+z)^r \leq 2^{|r|} \P^r \Phi(z)^{|r|}, \forall r \in \R,$ 
		\[
			\begin{aligned}
				 \partial_{w'_j}e^{\nu (\T- \Theta(w',\xi))} \Big|_{w'=x+\theta_1y} & \leq E_1(t,x,y,\xi) \Phi(x+\theta_1y)^{-1}\Theta(x+\theta_1y,\xi)\\
				 & \leq C E_1(t,x,y,\xi)\P^{-1} \Phi(y) \T \Theta(y,\xi).
			\end{aligned}
		\]
		where $E_1(t,x,y,\xi) = \exp\{\nu (\T- \Theta(x+\theta_1y,\xi))\}$.  From the following  the estimation
		
		\[
			\begin{aligned}
				\Big|\sum_{j=0}^{2l} \partial_{y_i}^j e^{\nu (\T- \Theta(w',\xi))} \Big|_{w'=x+\theta_1y} \Big| 
				& \leq CE_1(t,x,y,\xi)  \sum_{j=0}^{2l}  \left(\frac{\Theta(x+\theta_1y,\xi)}{\Phi(x+\theta_1y)}\right)^{j}\\
				&\leq CE_1 \sum_{j=0}^{2l} \left(\frac{(\ln(2\Phi(x+\theta_1y)\japxik))^{\varrho_2}}{\Phi(x+\theta_1y)}\right)^{j}\\
				&= C E_1\sum_{j=0}^{2l} \left(   \frac{(\ln(2\Phi(x+\theta_1y)))^j}{\Phi(x+\theta_1y)^{j/\varrho_2}} + \frac{(\ln\japxik)^j}{\Phi(x+\theta_1y)^{j/\varrho_2}}  \right)^{\varrho_2} \\
				&\leq C E_1 (\ln \japxik)^{2\varrho_2 l},
			\end{aligned}
		\]	
		we have 
		\[
			\la D_y \rak^{2l} E_1(t,x,y,\xi) \leq C E_1(t,x,y,\xi) (\ln \japxik)^{2\varrho_2 l}.
		\]
		Let 
		$$
			\begin{aligned}
				G(t,x,\xi) & = \chi_1 \om^{\tilde{m}_2}  \japxik^{\tilde{m}_1} \tilde\theta(t)^{\tilde{m}_3} 
				+ \chi_2 \om^{{m}_2'} \japxik^{{m}_1'} \left(\frac{\theta(t)}{t}\right)^{m_3'} \tilde{\theta}(t)^{m_4'+m_5'(|\alpha|+|\beta|+1+4l) }\\
				  & \quad + \chi_3 \om^{{m}_2} \japxik^{{m}_1} \left(\frac{\theta(t)}{t}\right)^{m_3} e^{m_4 \psi(t)}  \tilde{\theta}(t)^{m_5+m_6(|\alpha|+|\beta|+1+4l)} 
			\end{aligned}
		$$
		Note that for $|y| \geq 1$ we have $\la y \ra \leq \sqrt{2}|y|$ and in the case $|y| <1$ we have $\la y\ra < \sqrt{2}$. Using these estimates along with the fact that $\la y \ra^{-|\kappa|} \leq \Phi(y)^{-|\kappa|}$ we have
		\begin{linenomath*}
			\[
			\begin{aligned}
				|\partial_\xi^\alpha\partial_x^\beta r^{(1)}_\nu(t,x,\xi)|&\leq CC'_{|\alpha|+1} K'_{|\beta|+1} \T^{1+\vert \alpha \vert+\vert \beta \vert} \P^{-1-|\beta|} \japxik^{-1-|\alpha|} G(t,x,\xi)\\
				& \quad  \times \sum_{\tiny{\beta'+ \beta'' \leq \beta}}\sum_{\scriptsize{\alpha'+\alpha'' \leq \alpha}}  \int \dots \int \omega(z) ^{|m_2'|} \Phi(z)^{|\beta'|} \Phi(y)^{1+|\beta''|} \\
				& \quad \times  \Theta(y,\xi)^{1+|\beta''| + |\alpha''|} \la \eta \rak^{|m_1'-1-|\alpha'||+ |\kappa| -2l}  (\ln(\japxik))^{2\varrho_2 l} \tilde\theta(t)^{\kappa}\\ 
				& \quad \times   \frac{C'_{|\kappa|} K'_{|\kappa|}}{ (\Phi(y)\japxik)^{|\kappa|} } E_1(t,x,y,\xi) \la z \rak^{-2l}  \la y \rak^{-2l} \la \zeta \rak^{-2l}dz\textit{\dj}\zeta dy\textit{\dj}\eta.
			\end{aligned}
			\]
		\end{linenomath*}
		Given $\alpha, \beta$ and $\kappa$, we choose $l$ such that $2l > n + \max\{m_1^*,m_2^*\}+|\alpha|+|\beta|+ |\kappa|$. So that
		\begin{linenomath*}
			\[
			\begin{aligned}
				|\partial_\xi^\alpha\partial_x^\beta r^{(1)}_\nu(t,x,\xi)| & \leq  CC'_{|\alpha|+1} K'_{|\beta|+1} \T^{1+\vert \alpha \vert+\vert \beta \vert} \P^{-1-|\beta|} \japxik^{-1-|\alpha|} G(t,x,\xi)  \\
				& \quad \times \tilde\theta(t)^{\kappa} \int  \Big( (\ln(\japxik))^{2\mu l} \Theta(y,\xi)^{1+|\beta''| + |\alpha''|}  \\
				& \quad \times   \frac{C'_{|\kappa|} K'_{|\kappa|}}{ (\Phi(y)\japxik)^{|\kappa|} } E_1(t,x,y,\xi) \Big) dy.
			\end{aligned}
			\]
		\end{linenomath*}
		Noting the inequality (\ref{wgtseq}), we have $  (\ln(\japxik))^{2\varrho_2 l}  \Theta(y,\xi)^{1+|\beta''| + |\alpha''|} e^{-\delta_0 \Theta(y,\xi)} \leq C e^{-\frac{\delta_0}{2} \Theta(y,\xi)}$. Thus,
		\begin{linenomath*}
			\[
			\begin{aligned}
				|\partial_\xi^\alpha\partial_x^\beta r^{(1)}_\nu(t,x,\xi)| & \leq  CC'_{|\alpha|+1} K'_{|\beta|+1} \T^{1+\vert \alpha \vert+\vert \beta \vert} \P^{-1-|\beta|} \japxik^{-1-|\alpha|} G(t,x,\xi)\\
				 & \quad \times  \int  \exp\Big\{\nu(\T- \Theta( x+\theta_1y,\xi)) - \frac{\delta_0}{2} \Theta(y,\xi)\Big\} dy.
			\end{aligned}
			\]
		\end{linenomath*}
		Since
		\begin{linenomath*}
			\begin{equation}\label{Pineq}
				\begin{aligned}
					\Theta(x,\xi) - \Theta(x+\theta_1y,\xi) &\leq \T - (\T-\Theta(\theta_1y,\xi)) \\
					&\leq \Theta(\theta_1 y,\xi)\leq\Theta(y,\xi),
				\end{aligned}  			
			\end{equation}
		\end{linenomath*} 		
		and $\delta_0$ is independent of $\nu$, there exists $\delta^*_1>0$ (in fact $\delta^*_1= \frac{\delta_0}{2}$) such that, for $\nu(t) <\delta^*_1$ and $\varrho_1$ as in (\ref{ineq2}) we obtain the estimate
			\begin{linenomath*}
			\begin{equation}\label{rem}
			\begin{aligned}
				|\partial_\xi^\alpha\partial_x^\beta r^{(1)}_\nu(t,x,\xi)| & \leq  CC'_{|\alpha|+1} K'_{|\beta|+1} \T^{1+\vert \alpha \vert+\vert \beta \vert} \P^{-1-|\beta|} \japxik^{-1-|\alpha|} G(t,x,\xi) \\
				& \quad  \times  e^{-\delta_0'(\ln\japxik)^{\varrho_1} }, \quad \delta_0'>0.
			\end{aligned}
			\end{equation}
		\end{linenomath*}
		In view of Propositions (\ref{p1})-(\ref{p3}), we have
		\begin{alignat*}{3}
			\Theta(x,\xi)^{-1}r_{\nu}^{(1)} &\in L^{\infty}\left([0, T]; G^{-\infty,l_2^*-1}(\Phi,\gt) \right), && \qquad \text{ if } \tilde\theta \text{ is bounded}\\
			t^{1-\varepsilon}\Theta(x,\xi)^{-1}r_{\nu}^{(1)} &\in C\left([0, T]; G^{-\infty,l_2^*-1}(\Phi,\gt)\right), && \qquad \text{ otherwise}
		\end{alignat*}
		for $l_2^*= \max\{\tilde m_2,m_2'+m_3',m_2+m_3\}.$

		In a similar fashion, we will determine the growth estimate for $r^{(2)}_\nu(t,x,\xi)$. Let $\alpha, \beta,\kappa \in (\mathbb{Z}^+_0)^n$ and $l \in \mathbb{Z}^+$. Then		
				
		\begin{linenomath*}
			\[
			\begin{aligned}
				&\partial_\xi^\alpha\partial_x^\beta r^{(2)}_\nu(t,x,\xi)\\
				&= \sum_{i=1}^{n} \sum_{\tiny{\beta'+ \beta'' \leq \beta}}\sum_{\scriptsize{\alpha'+\alpha'' \leq \alpha}} \int \dots \int \eta^{-\kappa} \eta^\kappa e^{-iy\cdot\eta} \zeta^{-\kappa} \zeta^\kappa e^{-iz\cdot\zeta} \la z \rak ^{-2l} \la \eta \rak^{-2l}\la D_y\rak^{2l} \\
				&\quad\times  \la y \rak^{-2l} \la \zeta \rak^{-2l}  \la D_z\rak^{2l} \la D_\zeta\rak^{2l}   \la D_\eta\rak^{2l} (\partial_\xi^{\alpha'} \partial_x^{\beta'} D_{x_i}a)(t,x+z,\xi+\eta)\\
				& \quad \times\int_{0}^{1} \partial_\xi^{\alpha''}\partial_x^{\beta''}\partial_{w_i} e^{\nu (\Theta(x,w)-\Theta(x,\xi))}\Big|_{w=\xi+\theta_2(\eta+\zeta)}	
				d\theta_2 dz\textit{\dj}\zeta dy\textit{\dj}\eta,\\
				&= \sum_{i=1}^{n} \sum_{\tiny{\beta'+ \beta'' \leq \beta}}\sum_{\scriptsize{\alpha'+\alpha'' \leq \alpha}} \int \dots \int \eta^{-\kappa} e^{-iy\cdot\eta} \zeta^{-\kappa} e^{-iz\cdot\zeta} D_y^\kappa D_z^\kappa \la z \rak ^{-2l} \la \eta \rak^{-2l} \la D_y\rak^{2l} \\
				&\quad\times \la y \rak^{-2l}  \la \zeta \rak^{-2l}  \la D_z\rak^{2l}  \la D_\zeta\rak^{2l}   \la D_\eta \rak^{2l} (\partial_\xi^{\alpha'} \partial_x^{\beta'} D_{x_i}a)(t,x+z,\xi+\eta)\\
				& \quad \times\int_{0}^{1} \partial_\xi^{\alpha''}\partial_x^{\beta''}\partial_{w_i} e^{\nu (\Theta(x,w)-\Theta(x,\xi))}\Big|_{w=\xi+\theta_2(\eta+\zeta)}	
				d\theta_2 dz\textit{\dj}\zeta dy\textit{\dj}\eta.
			\end{aligned}
			\]	
		\end{linenomath*}
		Let $E_2(t,x,\xi,\eta,\zeta) = \exp\{\nu( \Theta(x,\xi+\theta_2(\eta+\zeta))-\T)\}$.	
		We have	
		\begin{linenomath*}
			\[
			\begin{aligned}
				|\partial_\xi^\alpha\partial_x^\beta r^{(2)}_\nu(t,x,\xi)| & \leq  CC'_{|\alpha|+1} K'_{|\beta|+1} \T^{1+\vert \alpha \vert+\vert \beta \vert} \P^{-1-|\beta|} \japxik^{-1-|\alpha|} G(t,x,\xi) \\
				& \quad \times \sum_{\tiny{\beta'+ \beta'' \leq \beta}}\sum_{\scriptsize{\alpha'+\alpha'' \leq \alpha}} \int \dots \int \omega(z)^{|m_2'|} \Phi(z)^{|\beta'|+1} \la \eta \rak^{|m_1'-|\alpha'||}\\
				& \quad \times
				 \la \eta \rak^{-2l} \la \zeta \rak^{-2l}\la \eta + \zeta \rak^{1+|\alpha''|} \left(\frac{C'_{|\kappa|}K'_{|\kappa|}}{(\P \la \zeta \rak \la \eta \rak)^{|\kappa|}}\right) \tilde\theta(t)^{\kappa}\\		
				& \quad \times \la z \rak^{-2l+|\kappa|} \la y\rak^{-2l-|\kappa|} E_2(t,x,\xi,\eta,\zeta) dz\textit{\dj}\zeta dy\textit{\dj}\eta,
			\end{aligned}
			\]
		\end{linenomath*}
		where we have noted estimate
		\[
			\la D_{\zeta}\rak^{2l} \la D_{\eta}\rak^{2l} E_2(t,x,\xi,\eta,\zeta) \leq C E_2(t,x,\xi,\eta,\zeta) (\ln\P)^{4\varrho_2l},
		\]
		as in the case of $r^{(1)}_\nu$.
		In this case we choose $l$ such that $2l > 2(n+1) + \max\{m_1^*,m_2^*\}+|\alpha|+|\beta|+ |\kappa|$. 			
		Noting $(\la \eta \rak \la \zeta \rak)^{-1} \leq \la \zeta + \eta \rak^{-1}$ and $(\ln\P)^{4\varrho_2l}   e^{-\delta_0\Theta(x,\eta+\zeta)} \leq e^{-\frac{\delta_0}{2}\Theta(x,\eta+\zeta)},$ from (\ref{wgtseq}) we get	
		\begin{linenomath*}
			\[
			\begin{aligned}
				|\partial_\xi^\alpha\partial_x^\beta r^{(2)}_\nu(t,x,\xi)| 
				&\leq  CC'_{|\alpha|+1} K'_{|\beta|+1} \T^{1+\vert \alpha \vert+\vert \beta \vert} \P^{-1-|\beta|} \japxik^{-1-|\alpha|} \\
				& \quad \times G(t,x,\xi) \tilde\theta(t)^{\kappa} \int \int
				(\la \eta \rak \la \zeta \rak )^{-2(n+1)} \\
				& \quad \times \exp\left\{ \nu(\Theta(x,\xi+\eta+\zeta) - \T)-\frac{\delta_0}{2}\Theta(x,\eta+\zeta)) \right\} \textit{\dj}\zeta \textit{\dj}\eta.
			\end{aligned}
			\] 
		\end{linenomath*}	
		Since $\Theta(x,\xi+\eta+\zeta) - \T \leq \Theta(x,\eta+\zeta)$ and
		and $\delta_0$ is independent of $\nu$, there exists $\delta^*_2>0$ (in fact $\delta^*_2= \frac{\delta_0}{2}$) such that, for $\nu(t) <\delta^*_2$ we obtain the estimate
		\begin{linenomath*}
			\[
			\begin{aligned}
				|\partial_\xi^\alpha\partial_x^\beta r^{(2)}_\nu(t,x,\xi)| & \leq  CC'_{|\alpha|+2} K'_{|\beta|+1} \T^{1+\vert \alpha \vert+\vert \beta \vert} \P^{-1-|\beta|} \japxik^{-1-|\alpha|} G(t,x,\xi) \\
				& \quad \times e^{-\delta_0''(\ln \P)^{\varrho_1} }, \quad \delta_0'' >0.
			\end{aligned}
			\]
		\end{linenomath*}	
		In view of Propositions (\ref{p1})-(\ref{p3}), we have
		\begin{alignat*}{3}
			\Theta(x,\xi)^{-1}r_{\nu}^{(2)} &\in L^{\infty}\left([0, T]; G^{l_1^*-1,-\infty}(\Phi,\gt) \right), && \qquad \text{ if } \tilde\theta \text{ is bounded}\\
			t^{1-\varepsilon}\Theta(x,\xi)^{-1}r_{\nu}^{(2)} &\in C\left([0, T]; G^{l_1^*-1,-\infty}(\Phi,\gt)\right), && \qquad \text{ otherwise}
		\end{alignat*}
		for $l_1^*= \max\{\tilde m_1,m_1'+m_3',m_1+m_3\}.$
			
		By similar techniques used in the case of $ r^{(1)}_\nu$ and $ r^{(2)}_\nu$, one can show that $r^{(3)}_\nu \in C([0,T]; G^{-\infty})$. Taking $\delta^{*} = \min\{\delta^{*}_1, \delta^{*}_2\}$, proves the theorem.  
	\end{proof}
	
	\begin{rmk}
		When $\tilde\theta$ is bounded, we get a more refined symbol estimate for the remainder $r_\nu^{(1)}$ as seen from (\ref{rem}), suggesting that
			\begin{equation*}
			\begin{aligned}
				&\Theta^{-1}r_\nu^{(1)} \in G^{-\infty,1}\{0\}(\omega^{\tilde m_2}\Phi^{-1},\gt)^{(1)}_N \cap G^{-\infty,1}\{m_3',0,0\}(\omega^{m_2'}\Phi^{-1},\gt)^{(2)}_N  \\
				& \qquad\cap  G^{-\infty,1}\{m_3,m_4,0,0\}(\omega^{m_2}\Phi^{-1},\gt)^{(3)}_N
			\end{aligned}
		\end{equation*} 
		Similar is the case for the remainder $r_\nu^{(2)}.$
	\end{rmk}
	
	Substituting $a(t,x,D_x)$ in Theorem \ref{conju} with identity operator $I$ and taking $k$ sufficiently large, one can easily arrive at the following two corollaries.
	
	\begin{cor}\label{cor1}
		There exists a $k^*>1$ such that for $k \geq k^*$
		\[
			e^{\pm\nu(t)\Theta(x,D_x)} e^{\mp\nu(t)\Theta(x,D_x)} = I + R^{(\pm)}(t,x,D_x)
		\]
		where $ I + R^{(\pm)}$ are invertible operators with $\T^{-1} \sigma(R^{(\pm)}) \in C([0,T];\GT(-1,-1)).$
	\end{cor}
	
	\begin{cor}\label{cor2}
		Let $0 \leq \varepsilon \leq \varepsilon'<\delta^*$ where $\delta^*$ is as in Theorem \ref{conju}. Then
		\[
			e^{\varepsilon\Theta(x,D_x)} e^{-\varepsilon'\Theta(x,D_x)} = e^{(\varepsilon-\varepsilon')\Theta(x,D_x)}\big(  I+ \hat{R}(x,D_x)  \big)
		\]
		where $\T^{-1} \sigma(\hat{R}) \in \GT(-1,-1)$ and for sufficiently large $k,$ $I+ \hat{R}$ is invertible.
	\end{cor}
	We use the above corollaries to prove the continuity of the operator $e^{\varepsilon \Theta(x,D_x) }$ on the spaces $\sobol(s,\varepsilon').$ The following proposition is helpful in making change of variable in Section \ref{Energy2}.
	\begin{prop}\label{conti}
		The operator $e^{\varepsilon \Theta(x,D_x) } :\sobol(s,\varepsilon') \to \sobol(s,\varepsilon'-\varepsilon)$ is continuous for $k\geq k_0$ and $0 \leq \varepsilon \leq \varepsilon' < \delta^*$ where $k_0$ sufficiently large and $\delta^*$ is as in Thereom \ref{conju}.
	\end{prop}
	\begin{proof}
		Consider $w $ in $\sobol(s,\varepsilon')$. From Corollaries \ref{cor1} and \ref{cor2}, we have
		\begin{linenomath*}
			$$
			\begin{aligned}
				e^{-\varepsilon' \Theta(x,D_x) } e^{\varepsilon' \Theta(x,D_x)} &= I+R_1(x,D_x), \\
				e^{\varepsilon \Theta(x,D_x)} e^{-\varepsilon' \Theta(x,D_x)} &= e^{(\varepsilon-\varepsilon') \Theta(x,D_x)}(I+R_2(x,D_x)),\\
				e^{(\varepsilon' -\varepsilon) \Theta(x,D_x)} e^{-(\varepsilon' -\varepsilon) \Theta(x,D_x)} &= I+R_3(x,D_x).
			\end{aligned}
			$$	
		\end{linenomath*}
		where $\T^{-1}\sigma(R_j) \in \GT(-1,-1)$.	For $k\geq k_0$, $k_0$ sufficiently large, the operators $I+R_j(x,D_x),j=1,2,3$ are invertible.
		Then, one can write 
		\begin{linenomath*}
			\begin{equation*}
				e^{\varepsilon \Theta(x,D_x) } w = e^{\varepsilon (\Theta(x,D_x) }  \left(e^{-\varepsilon' \Theta(x,D_x) } e^{\varepsilon' \Theta(x,D_x) } -R_1 \right) w.
			\end{equation*}
		\end{linenomath*}
		This implies that 
			\begin{linenomath*}
			\begin{equation}
				\label{c2}
				e^{\varepsilon \Theta(x,D_x) } (I+R_1)w = e^{(\varepsilon-\varepsilon') \Theta(x,D_x) } (I+R_2) e^{\varepsilon' \Theta(x,D_x) } w.
			\end{equation}
		\end{linenomath*}
		From (\ref{c2}), we have
		\begin{linenomath*}
			\[
			e^{(\varepsilon' -\varepsilon) \Theta(x,D_x) } e^{\varepsilon \Theta(x,D_x) } (I+R_1)w = (I+R_3)(I+R_2) e^{\varepsilon' \Theta(x,D_x) } w.
			\]
		\end{linenomath*}
		Note that $(I+R_j),j=1,2,3,$ are bounded and invertible operators. Substituting  $w = (I+R_1)^{-1}v$ and taking $L^2$ norm on both sides of the above equation yields 
		\begin{linenomath*}
			\[
			\Vert e^{\varepsilon \Theta(x,D_x) } v \Vl(s,\varepsilon'-\varepsilon) \leq C_1 \Vert (I+R_1)^{-1}v  \Vl(s,\varepsilon')  \leq C_2 \Vert v  \Vl(s,\varepsilon')\;,
			\]
		\end{linenomath*}
		for all $v \in \sobol(s,\varepsilon')$ and for some $C_1,C_2>0$. This proves the proposition.
	\end{proof}

 \section{Proofs } \label{Proof1}

 	In this section, we give proofs of the main results. There are four key steps in the proofs of Theorem \ref{result1} and Theorem \ref{result3}. First, we employ the diagonalization procedure available in the literature \cite{yag,KuboReis} to arrive at a first order system with diagonalized principal part. Second, we localize the singularity to the regions $\zmid(N)$ and $\zext(N)$. Third, using the locazation achieved in the previous step we define a function that majorizes the symbol with singularity. This helps in making appropriate change of variable to handle the singularity. Lastly, using sharp G\r{a}rding’s inequality we arrive at an energy estimate that proves the well-posedness of the problem in the Sobolev spaces $\sobol(s,\delta)$.
 
	\subsection{First Order System with Diagonalized Principal Part}
	
	Let $\Lambda(x,D_x)= \P \la D_x \rak$ and $\tilde \Lambda(x,D_x)$ be such that $\sigma(\tilde \Lambda) = \P^{-1}\japxik^{-1}.$ Observe that 
	\[
		\tilde \Lambda(x,D_x) \Lambda(x,D_x) = I + K(x,D_x)
	\] 
	where $\sigma(K) \in \G(-1,-1).$ We choose $k>k_0$ for large $k_0$ so that the operator norm of $K$ is strictly lesser than $1.$ This guarantees
	\[
		(I + K(x,D_x))^{-1} = \sum_{j=0}^{\infty} (-1)^j K(t,x,D_x)^j \in \G(0,0).
	\]
	The transformation $U=\left(U_{1}, U_{2}\right)^{T}=\left(\Lambda(x,D_x) u, D_{t} u\right)^{T}$ transfers the Cauchy problem (\ref{eq1}) to
	\begin{equation}
		\label{FOS1}
		D_{t} U-A U=F, \quad U(0, x)=\left(\begin{array}{c}
			\Lambda(x,D_{x})f_1(x) \\
			-i f_2(x)
		\end{array}\right)
	\end{equation}
	where $F := \left( 0,-f \right)^T$ and
	$$
	A:=\left(\begin{array}{cc}
		0 & \Lambda(x,D_{x}) \\
		\left(a(t,x,D_x)+b(t,x,D_x)\right)(I+K(x,D_x))^{-1}\tilde\Lambda(x,D_x) & 0
	\end{array}\right).
	$$
	
	Consider a function $\chi \in C_{0}^{\infty}(\mathbb{R})$ such that $\chi(s) \equiv 1$ for $|s| \leq 1, \chi(s) \equiv 0$ for $|s| \geq 2$ and $0 \leq \chi(s) \leq 1$. Let $N>k$ in the definition of $\zi$ and $\zii$, where $k$ is appropriately chosen later in the discussion. We define the functions
	\begin{equation}
		\lambda_{j}(t, x, \xi)= d_{j} \CHI \om \japxik + \CHIC \tau_{j}(t, x, \xi), \quad j=1,2 
	\end{equation}
	where $d_{2}=-d_{1}$ is a positive constant and
	\begin{equation}
			\tau_{j}(t, x, \xi)=(-1)^{j} \sqrt{a(t, x, \xi)}, \quad a(t, x, \xi):=\sum_{l,m=1}^{n} a_{l,m}(t, x) \xi_{l} \xi_{m}.
	\end{equation}
    Note that $ a \in \Gint(2,2,N,1) \cap \Gmid(2,2,0,N,1,0) \cap \Gext(2,2,0,0,N,1,0).$ Let $\chi_1$, $\chi_2$ and  $\chi_3$ be the indicator functions for the regions $\zint(2N)$, $\zmid(N)$ and $ \zext(N),$ respectively. Observe that 
    \begin{enumerate}[label=\roman*)]
    	\item $\lambda_j\var $ is $G_{\omega}$-elliptic symbol i.e. for all $\var \in [0,T] \times \R^n \times \R^n$ \[
    	\vert \lambda_j(t,x,\xi) \vert \geq C\om\japxik, 
    	\]
    	for some $C>0$ independent of $k.$
    	
    	\item $\lambda_{j} \in \Gint(1,1,2N,0) \cap \Gmid(1,1,0,N,1,1) \cap \Gext(1,1,0,0,N,1,1).$ More precisely, for $|\alpha| + |\beta| > 0,$
    	\begin{equation}\label{root1}
    		\begin{rcases}
    			\begin{aligned}
    				\vert \lambda_j \var \vert & \leq C_0 \om \japxik \left\{  \chi_1 + \left(  \chi_2+\chi_3  \right) \tilde\theta(t) \right\} \\
    				\vert \partial_{\xi}^{\alpha} \partial_{x}^{\beta} \lambda_j \var \vert & \leq C_{\alpha\beta} \om \P^{-|\beta|} \japxik^{1-|\alpha|} \left\{  \chi_1 + \left(  \chi_2+\chi_3  \right) \tilde\theta(t)^{|\alpha| + |\beta|} \right\}
    			\end{aligned}
    		\end{rcases}
    	\end{equation}
    	
    	\item $\partial_t\lambda_{j} \in \Gint(1,1,2N,1) \cap \Gmid(1,1,1,N,0,1) \cap \Gext(1,1,1,0,N,0,1).$ More precisely, for $|\alpha| + |\beta| > 0,$
    	\begin{equation}\label{root2}
    		\begin{rcases}
    			\begin{aligned}
    				\partial_t \lambda_j & \sim 0 \quad \text{ in } \zint(N) \\
    				\vert \partial_t \lambda_j \var \vert & \leq C_0 \om \japxik \left\{  \chi_1 \tilde\theta(t) + \left(  \chi_2+\chi_3  \right) \frac{\theta(t)}{t} \right\} \\
    				\vert \partial_{\xi}^{\alpha} \partial_{x}^{\beta} \partial_t \lambda_j \var \vert & \leq C_{\alpha\beta} \om \P^{-|\beta|} \japxik^{1-|\alpha|} \left\{  \chi_1 \tilde\theta(t)+ \left(  \chi_2+\chi_3  \right)  \frac{\theta(t)}{t} \tilde\theta(t)^{|\alpha| + |\beta|} \right\}
    			\end{aligned}
    		\end{rcases}
    	\end{equation}
    \end{enumerate}

	We begin the diagonalization procedure by defining the matrix pseudodifferential operators $\cn(1)$ and $\tcn(1)$ with symbols
	$$
	\begin{aligned}
		\cn(1) &= \left(\begin{array}{cc}
		I & I \\
		\lambda_{1}(t, x, D_x) \tilde\Lambda(x,D_x) & \lambda_{2}(t, x, D_x) \tilde\Lambda(x,D_x) 
    	\end{array}\right), \quad \text{ and } \\
	  \tcn(1)  &= \frac{1}{2} \Lambda(x,D_x) \tilde \lambda_2(t,x,D_x)I \left(\begin{array}{cc}
			  \lambda_{2}(t, x, D_x) \tilde\Lambda(x,D_x) & -I \\
			- \lambda_{1}(t, x, D_x) \tilde\Lambda(x,D_x)&  I 
		\end{array}\right),
	\end{aligned}
	$$
	where $\sigma\big( \tilde \lambda_2(t,x,D_x) \big) = \lambda_2(t,x,\xi)^{-1}.$ Notice that $\mathcal{N}_1$ and $\tilde{\mathcal{N}}_1$ are elliptic and satisfy
	\[
		{\mathcal{N}}_1(t,x,D_x) \tilde{\mathcal{N}}_1(t,x,D_x) = I + K_1(t,x,D_x)
	\] 
	with $\sigma(K_1)$ in 
	$$ \Gint(-1,-1,N,0) \cap \Gmid(-1,-1,0,N,2,1) \cap \Gext(-1,-1,0,0,N,2,1) .$$
 	 From Propositions \ref{p1}-\ref{p3} and Remark \ref{lb}, $\sigma(K_1) \in G^{-1+\varepsilon,-1+\varepsilon}(\Phi,\go).$ We choose $k>k_1$ for large $k_1$ so that the operator norm of $K_1$ is strictly lesser than $1.$ This guarantees
	\[
		(I + K_1(t,x,D_x))^{-1} = \sum_{j=0}^{\infty} (-1)^j K_1(t,x,D_x)^j \in C([0,T];OP \G(0,0)) .
	\]
	We make the following change of variable 
	\begin{equation}
		\label{COV1}
		V_1(t,x)=\tilde{\mathcal{N}}_1(t,x,D_x) U(t,x). 
	\end{equation}
	Implying 
	\begin{align*}
	    D_{t} V_1 &=  \tilde{\mathcal{N}}_1 D_{t} U  + D_{t}  \tilde{\mathcal{N}}_1 U\\
		&=  \tilde{\mathcal{N}}_1 A (I+K_1)^{-1} {\mathcal{N}}_1 V_1 + D_{t} \tilde{\mathcal{N}}_1 (I+K_1)^{-1} {\mathcal{N}}_1 V_1 + F_1 . 
	\end{align*}
	Here $F_1 = \tilde{\mathcal{N}}_1  F $ and
	\begin{itemize}
		\item $\tilde{\mathcal{N}}_1 A (I+K_1)^{-1}{\mathcal{N}}_1 = A_1+A_2+A_3$ where 
		\begin{align*}
			\sigma(A_1) &= diag(\lambda_1(t,x,\xi), \lambda_2(t,x,\xi)), \\
			\sigma(A_2) & \in \Gint(1,1,2N,1) , \quad \sigma(A_2) \equiv 0 \text{ in } \zmid(N) \cup \zext(N),\\
			\sigma(A_3) & \in \Gint(0,0,N,0) \cap \Gmid(0,0,0,N,0,0) \cap \Gext(0,0,0,0,N,0,0) 
		\end{align*}
		
		\item $D_{t} \tilde{\mathcal{N}}_1 (I+K_1)^{-1} {\mathcal{N}}_1=-\tilde{\mathcal{N}}_1 D_{t}{\mathcal{N}}_1+ B_0 ,$ where
		$
			B_0  = D_t K_1 + D_t \tilde{\mathcal{N}}_1 \left( \sum_{j=1}^{\infty}(-K_1)^j \right){\mathcal{N}}_1,
		$
		\small{
		\begin{align*}
			\sigma(\tilde{\mathcal{N}}_1) \sigma\left(D_{t}{\mathcal{N}}_1\right) &= \frac{1}{2} \left(\begin{array}{cc}
				\frac{\Lambda(x,\xi)}{\lambda_1(t,x,\xi)}D_t \frac{ \lambda_1(t,x,\xi)}{\Lambda(x,\xi)}& \frac{-\Lambda(x,\xi)}{\lambda_2(t,x,\xi)}D_t \frac{ \lambda_2(t,x,\xi)}{\Lambda(x,\xi)} \\
				\frac{-\Lambda(x,\xi)}{\lambda_1(t,x,\xi)}D_t \frac{ \lambda_1(t,x,\xi)}{\Lambda(x,\xi)} & \frac{\Lambda(x,\xi)}{\lambda_2(t,x,\xi)}D_t \frac{ \lambda_2(t,x,\xi)}{\Lambda(x,\xi)}
			\end{array}\right), \\
			\sigma(\tilde{\mathcal{N}}_1 D_{t}{\mathcal{N}}_1) &\in \Gint(0,0,2N,1) \cap  \Gmid(0,0,1,N,0,1) \\
			& \qquad  \cap  \Gext(0,0,1,0,N,0,1) ,\\
			\sigma(B_0) &	\in	 \Gint(-1,-1,2N,1) \cap \Gmid(-1,-1,1,N,2,1) \\
			& \qquad \cap \Gext(-1,-1,1,0,N,2,1),	
		\end{align*}}
		and $\sigma(\tilde{\mathcal{N}}_1 D_{t}{\mathcal{N}}_1)=\sigma(B_0)=0$ in $\zint(N)$ by (\ref{root2}).
	\end{itemize}	
	In summary, we have the following proposition.
	\begin{prop}\label{D1}
		The solution $U(t,x)$ to the first order system (\ref{FOS1}) is given by
		\[
			U(t,x) = (I+K_1(t,x,D_x))^{-1} \mathcal{N}_1(t,x,D_x)V_1(t,x)
		\]
		where $V_1(t,x)$ is the solution to the following system
		\begin{equation}
			\label{FOS2}
			(D_{t} -\mathcal{D} +P_{1} +P_{2} +Q_1) V_1= F_1, \; V_1(0,x)= \tilde{\mathcal{N}}_1(0,x,D_x)U(0,x) 
		\end{equation}
		The matrix pseudodifferential operators $\mathcal{D}, P_{1}, P_{2}, Q$ possess the following properties:
		\begin{itemize}
			\item  The operator $\mathcal{D}=\mathcal{D}_1 + \mathcal{D}_2$ is such that
			\[
				\begin{aligned}
					\sigma(\mathcal{D}_1) &= \text{diag}\left\{\lambda_{1}\var , \lambda_{2}\var \right\} ,\\
					\sigma(\mathcal{D}_1)  &\in \Gint(1,1,2N,0) \cap \Gmid(1,1,0,N,1,1) \\  
					&\qquad \cap \Gext(1,1,0,0,N,1,1),  \\
					\sigma(\mathcal{D}_2) &= \frac{1}{2}\text{diag}\left\{\frac{\Lambda(x,\xi)}{\lambda_1\var}{D_{t} \frac{\lambda_{1}\var}{\Lambda(x,\xi)}},\frac{\Lambda(x,\xi)}{\lambda_2\var}{D_{t} \frac{\lambda_{2}\var}{\Lambda(x,\xi)}}\right\}, \\
					\sigma(\mathcal{D}_2) &\in \Gint(0,0,2N,1) \cap \Gmid(0,0,1,N,0,1) \\
					& \qquad \cap \Gext(0,0,1,0,N,0,1), \\
					\sigma(\mathcal{D}_2) &=0 \text{ in } \zint(N).
				\end{aligned}
			\] 
			
			\item $P_{1}$ is diagonal while $P_{2}$ is anti-diagonal and $	\sigma({P}_2) =0 \text{ in } \zint(N).$
			
			\item $\sigma\left(P_{1}\right) \in \Gint(0,0,N,0) \cap \Gmid(0,0,0,N,0,0) \cap \Gext(0,0,0,0,N,0,0)$.
			
			\item $\sigma\left(P_{2}\right) \in \Gint(0,0,2N,1) \cap \Gmid(0,0,1,N,0,1)\cap \Gext(0,0,1,0,N,0,1)$.
			
			\item $\sigma(Q_1) \in \Gint(1,1,2N,1), \sigma(Q) \equiv 0$ in $\zmid(2N) \cup \zext(2N)$.
		\end{itemize}
	\end{prop}

	\subsection{Localization in $Z_{ext}$}
	The main goal of this section is to localize the singularity arising from the first $t$-derivative of coefficients to $\zmid(N)$ and the one from the second $t$-derivative to $\zext(N)$. We will not directly localize the first order system in Proposition \ref{D1} but a modified one after an application of suitable transformation.	
	To this end, consider the elliptic pseudodifferential operators $\mathcal{N}_2, \tilde{\mathcal{N}}_2$ with symbols
	\[
		\sigma(\mathcal{N}_2) = \sigma(\tilde{\mathcal{N}}_2)^{-1} =\left(\begin{array}{cc}
			\left(\frac{\lambda_{1}(t, x, \xi)}{\Lambda(x,\xi)}\right)^{1/2} & 0 \\
			0 & 	\left(\frac{\lambda_{2}(t, x, \xi)}{\Lambda(x,\xi)}\right)^{1/2} 
		\end{array}\right),
	\]
	$\sigma(\mathcal{N}_2), \sigma(\tilde{\mathcal{N}}_2) \in \Gint(0,0,N,0) \cap \Gmid(0,0,0,N,1,1) \cap \Gext(0,0,0,0,N,1,1).$ Note that the symbols is constant in $\zint(N)$. Let
	\[
		\cn(2)\tcn(2)(t,x,D_x) = I + K_2(t,x,D_x),
	\]
	with $\sigma(K_2)$ in
 	$$ \Gint(-1,-1,N,0) \cap \Gmid(-1,-1,0,N,2,1) \cap \Gext(-1,-1,0,0,N,2,1).$$
	 We choose $k>k_2$ for large $k_2$ so that the operator norm of $K_2$ is strictly lesser than $1.$ This guarantees
	\[
	(I + K_2(t,x,D_x))^{-1} = \sum_{j=0}^{\infty} (-1)^j K_2(t,x,D_x)^j \in C([0,T]; OP\G(0,0)).
	\]	
	We make the following change of variable
	\[
		{V}_2(t,x) = \tilde{\mathcal{N}}_2(t,x,D_x) V_1(t,x).
	\]
	Implying
	\[
		\begin{aligned}
			D_t V_2 &= \tilde{\mathcal{N}}_2 D_tV_1 + D_t \tilde{\mathcal{N}}_2 V_1 \\
			 &= \left( \tilde{\mathcal{N}}_2(\mathcal{D}_1+\mathcal{D}_2-P_1-P_2-Q) +  D_t \tilde{\mathcal{N}}_2\right) (I+K_2)^{-1} \mathcal{N}_2 {V}_2 + F_2,
		\end{aligned}
	\]
	where $F_2 = \tilde{\mathcal{N}}_2 F_1.$ Observe that $\sigma(\tilde{\mathcal{N}}_2)\sigma(\mathcal{D}_2)+\sigma(D_t \tilde{\mathcal{N}}_2)=0$ as { \small
	\[
		\sigma(\tilde{\mathcal{N}}_2)\sigma(\mathcal{D}_2)= -\sigma(D_t \tilde{\mathcal{N}}_2)= \left(\begin{array}{cc}
			\frac{1}{2}\frac{\Lambda(x,\xi)^{1/2}}{\lambda_1(t,x,\xi)^{3/2}} D_t \lambda_1(t,x,\xi) & 0\\
			0 & \frac{1}{2} \frac{\Lambda(x,\xi)^{1/2}}{\lambda_2(t,x,\xi)^{3/2}} D_t \lambda_2(t,x,\xi)
		\end{array}\right).
	\]}
	Here 
	$$
	\begin{aligned}
		\sigma(\tilde{\mathcal{N}}_2\mathcal{D}_2 + D_t \tilde{\mathcal{N}}_2) &\in \Gint(-1,-1,2N,2) \cap \Gmid(-1,-1,1,N,2,1)  \\ 
		& \qquad  \cap \Gext(-1,-1,1,0,N,2,1).
	\end{aligned}
	$$
	To be precise, due to the presence of $D_t \lambda_j,j=1,2,$ in $\sigma(D_t \tilde{\mathcal{N}}_2)$ the singularity of order $2$ in $\zint(2N)$ appears only in the region $\zint(2N) / \zint(N)$ as $D_t \lambda_j=0$ in $\zint(N).$ 
	From the estimate (\ref{root2}) Remark \ref{lb},
	$$
	\begin{aligned}
		\sigma(\tilde{\mathcal{N}}_2\mathcal{D}_2 + D_t \tilde{\mathcal{N}}_2) &\in \Gint(0,0,2N,0) \cap \Gmid(0,0,1,N,0,0)  \\ 
		& \qquad  \cap \Gext(-1,-1,1,0,N,2,1).
	\end{aligned}
	$$
	Let 
	\begin{equation}
		\label{ops}
		\begin{rcases}
			\begin{aligned}
				\tilde P_1 & =  \tilde{\mathcal{N}}_2 P_1  (I+K_2)^{-1}  \mathcal{N}_2 + \mathcal{D}_1 -  \tilde{\mathcal{N}}_2 \mathcal{D}_1 (I+K_2)^{-1}  \mathcal{N}_2 ,\\
				\tilde P_2 & = \tilde{\mathcal{N}}_2 P_2  (I+K_2)^{-1} \mathcal{N}_2 ,\\
				\tilde Q_1 & = \tilde{\mathcal{N}}_2 Q_1 (I+K_2)^{-1} \mathcal{N}_2,\\
				\tilde Q_2 & = -	\big(\tilde{\mathcal{N}}_2\mathcal{D}_2 + D_t \tilde{\mathcal{N}}_2 \big)(I+K_2)^{-1}  \mathcal{N}_2.
			\end{aligned}
	\end{rcases}
	\end{equation}
	Here $\sigma( \tilde P_1)$ is in
	\[
		\Gint(0,0,N,0) \cap \Gmid(0,0,0,N,2,1) \cap \Gext(0,0,0,0,N,2,1).
	\]	
	By Propositions \ref{p1}-\ref{p3} and Remark \ref{lb}, $t^{1-\varepsilon}\sigma( \tilde P_1) \in \G(0,0)$ for every $\varepsilon \in (0,1).$
	It is easy to see that $\tilde P_1$ is of diagonal structure while $\tilde P_2$ is of anti-diagonal and 
	\[
		\begin{aligned}
			\sigma(\tilde P_1)  &\in \Gint(0,0,N,0) \cap \Gmid(0,0,0,N,2,1)  \\ & \qquad \cap \Gext(0,0,0,0,N,2,1), \\
			 \sigma(\tilde P_2) &\in \Gint(0,0,N,1) \cap \Gmid(0,0,1,N,0,1)  \\ & \qquad \cap \Gext(0,0,1,0,N,0,1), \\
			\sigma(\tilde Q_1) &\in \Gint(1,1,N,1), \sigma(\tilde Q_1) \equiv 0 \text{ in } \zmid(2N) \cup \zext(2N),\\
			\sigma(\tilde Q_2) &\in \Gint(0,0,2N,0) \cap \Gmid(0,0,1,N,0,0)  \\ 
			& \qquad  \cap \Gext(-1,-1,1,0,N,2,1).
		\end{aligned}
	\]
	Summarizing the above discussion we have the following proposition.
	\begin{prop}\label{D2}
		The solution $V_1(t,x)$ to first order system (\ref{FOS2}) is given by
		\[
		V_1(t,x) = (I + K_2(t,x,D_x))^{-1}\mathcal{N}_2(t,x,D_x) {V}_2(t,x)
		\]
		where ${V}_2(t,x)$ satisfies the following system
		\begin{equation}
			\label{FOS20}
			(D_{t} -\mathcal{D}_1 +\tilde P_{1} + \tilde P_{2} + \tilde Q_1 + \tilde Q_2) V_2= F_2, \; V_2(0,x)=  \tilde{\mathcal{N}}_2(0,x,D_x)V_1(0,x).
		\end{equation}
		The matrix pseudodifferential operators $\tilde P_{1}, \tilde P_{2},  \tilde Q_1,  \tilde Q_2$ are as in (\ref{ops}) and $\mathcal{D}_1$ is as in Proposition \ref{D1}.
	\end{prop}
	
	Next, we aim to localize the singularity.
	Let
	\[
		\sigma\left(\tilde P_{2}\right)=\left(\begin{array}{cc}
			0 & p_{12} \\
			p_{21} & 0
		\end{array}\right).
	\]
	We define the elliptic pseudodifferential operators $\cn(3)$ and $\tcn(3)$ with symbols
	\[
		\sigma(\cn(3)) = \sigma(\tcn(3))^{-1} = I+\eta\var
	\]
	where 
	\[
		\eta(t, x, \xi):=  \left(1-\chi(t/\tilde{t}_{x,\xi})\right) \left(\begin{array}{cc}
			0 & \frac{p_{12}}{2\lambda_{1}} \\
			\frac{p_{21}}{2\lambda_{2}} & 0
		\end{array}\right) 	\in \Gext(-1,-1,1,0,N,0,1).\\		
	\]
	The form of $\sigma(\tcn(3))$ is 
	\[
		\sigma(\tcn(3)) = \frac{1}{1-r}\left( \begin{array}{cc}
			1 & - (1-\chi)\frac{p_{12}}{2\lambda_{1}}\\
			-(1-\chi)\frac{p_{21}}{2\lambda_{2}} & 1
			\end{array}\right), \; r = \left(1-\chi(t/\tilde{t}_{x,\xi})\right)^2\frac{p_{12}p_{21}}{4\lambda_1\lambda_2}.
	\]
	The symbol $\sigma(\tcn(3))$ is well defined as we see that
	\[
		\vert r\var\vert \leq \frac{C}{\P^2\japxik^2} \left(\frac{\theta(t)}{t}\right)^2 \leq \frac{C}{N^2} \text{ in } \zext(N)
	\]
	and a large $N$ ensures that $\vert r\var\vert \leq 1/2$ in $\J.$ Let
	\[
	\cn(3)\tcn(3)= I + K_3(t,x,D_x),
	\]
	with $\sigma(K_3) =0$ in $\zint(N)\cup \zmid(N)$ and $\sigma(K_3) \in \Gext(-2,-2,1,0,N,0,1).$ From Propositions \ref{p1}-\ref{p3} and Remark \ref{lb}, $\sigma(K_3) G^{-1+\varepsilon,-1+\varepsilon}(\Phi,\go).$ We choose $k>k_3$ for large $k_3$ so that the operator norm of $K_3$ is strictly lesser than $1.$ This guarantees
	\[
	(I + K_3(t,x,D_x))^{-1} = \sum_{j=0}^{\infty} (-1)^j K_3(t,x,D_x)^j \in C([0,T]; OP\G(0,0)).
	\]	
	
	We make a following change of variable 
	\[
		V_3(t,x) = \tcn(3) V_2(t,x)
	\]
	Implying
	\[
		\begin{aligned}
			D_tV_3  &= \tilde{\mathcal{N}}_3D_tV_2 + D_t \tilde{\mathcal{N}}_3 V_2\\
			&= \left( \tilde{\mathcal{N}}_3 ( \mathcal{D}_1 -\tilde P_{1} - \tilde P_{2} - \tilde Q_1 - \tilde Q_2) + D_t \tilde{\mathcal{N}}_3 \right) (I+K_3)^{-1}\mathcal{N}_3V_3 + F_3  
		\end{aligned}
	\]
	where $F_3 = \tilde{\mathcal{N}}_3 F_2.$  Let us write
	\[
				\tilde{\mathcal{N}}_3  \big( \mathcal{D}_1 - \tilde{P}_2\big) (I+K_3)^{-1}\mathcal{N}_3  = \tilde{\mathcal{N}}_3 \big( \mathcal{D}_1 - \tilde{P}_2\big) \mathcal{N}_3  + \tilde{\mathcal{N}}_3 \big( \mathcal{D}_1 - \tilde{P}_2\big) \left(\sum_{j=1}^{\infty}(-K_3)^j\right) \mathcal{N}_3.
	\]
	We see that
	\[
		\begin{aligned}
			&\sigma(\tilde{\mathcal{N}}_3)  \sigma(\mathcal{D}_1) \sigma(\mathcal{N}_3) \\
			&=  \frac{1}{1-r}\left( \begin{array}{cc}
				1 & - (1-\chi)\frac{p_{12}}{2\lambda_{1}}\\
				-(1-\chi)\frac{p_{21}}{2\lambda_{2}} & 1
			\end{array}\right) 
			\left(\begin{array}{cc}
				\lambda_1 & 0\\
				0 & \lambda_2
			\end{array}\right)
			\left( \begin{array}{cc}
				1 & (1-\chi)\frac{p_{12}}{2\lambda_{1}}\\
				(1-\chi)\frac{p_{21}}{2\lambda_{2}} & 1
			\end{array}\right) \\
			&= \frac{1}{1-r} \left\{ \mathcal{D}_1  + (1-\chi)^2 \frac{p_{12}p_{21}}{4\lambda_2} diag\left\{  1,  -1\right\} + (1-\chi)\tilde{P}_2\right\}\\
			&= \mathcal{D}_1 + (1-\chi)\tilde{P}_2 + \frac{r}{1-r} \left( \mathcal{D}_1 + \tilde{P}_2 \right) + \frac{1}{1-r} (1-\chi)^2 \frac{p_{12}p_{21}}{4\lambda_2} diag\left\{  1,  -1\right\}\\
			&= \mathcal{D}_1 + (1-\chi)\tilde{P}_2 \; \text{ mod } \; \Gext(-1,-1,2,0,N,2,1).
		\end{aligned}
	\]
	Similarly,
	\[
		\begin{aligned}
			&\sigma(\tilde{\mathcal{N}}_3)  \sigma(\tilde{P}_2) \sigma(\mathcal{N}_3) \\
			&=  \frac{1}{1-r}\left( \begin{array}{cc}
				1 & - (1-\chi)\frac{p_{12}}{2\lambda_{1}}\\
				-(1-\chi)\frac{p_{21}}{2\lambda_{2}} & 1
			\end{array}\right) 
			\left(\begin{array}{cc}
				0 & p_{12}\\
				p_{21} & 0
			\end{array}\right)
			\left( \begin{array}{cc}
				1 & (1-\chi)\frac{p_{12}}{2\lambda_{1}}\\
				(1-\chi)\frac{p_{21}}{2\lambda_{2}} & 1
			\end{array}\right) \\
			&= \tilde{P}_2 + \frac{r}{1-r} \tilde{P}_2 -  \frac{1}{1-r} \left((1-\chi)\frac{p_{12}p_{21}}{2\lambda_1} - \frac{p_{12}p_{21}}{2\lambda_2} \right) I   -  \frac{(1-\chi)}{1-r} \left(
			\begin{array}{cc}
				0 &  \frac{p_{12}^2p_{21}}{4\lambda_1^2}\\
				  \frac{p_{12}p_{21}^2}{4\lambda_2^2} & 0
			\end{array}
			\right)\\
			& = \tilde{P}_2 \; \text{ mod } \; \Gext(-1,-1,2,0,N,2,1).
		\end{aligned}
	\]
	Thus  $	\sigma(\tilde{\mathcal{N}}_3)  \left(\sigma(\mathcal{D}_1) -  \sigma(\tilde{P}_2)\right) \sigma(\mathcal{N}_3) =  \mathcal{D}_1 -\chi \tilde{P}_2 \; \text{ mod } \; \Gext(-1,-1,2,0,N,2,1) .$ Note that $	\sigma(\chi\tilde{P}_2) = 0$ in $\zint(N)$ and 
	\[
		\sigma(\chi\tilde{P}_2) \in \Gint(0,0,2N,1) \cap \Gmid(0,0,1,2N,0,1), \; \sigma(\chi\tilde{P}_2) \equiv 0 \text{ in } \zext(N).
	\]
	From the structure of $\tilde{P}_2$ and the estimate on the second derivative in time of the characteristics, 
	\[
		\sigma(D_t\tilde{\mathcal{N}}_3) \in \Gext(-1,-1,2,1,N,2,1), \; \sigma(D_t\tilde{\mathcal{N}}_3) \equiv 0 \text{ in } \zint(N) \cup \zmid(N).
	\] 
	Let
	\begin{equation}\label{p34I}
		\begin{aligned}
				P_3 &= \tilde{\mathcal{N}}_3 \tilde{P}_1 (I+K_3)^{-1} \mathcal{N}_3, \\
				P_4 &=  \tilde{\mathcal{N}}_3 \left(\tilde{P}_2 + \tilde Q_2- \mathcal{D}_1\right)(I+K_3)^{-1} \mathcal{N}_3  + \mathcal{D}_1   - D_t\tilde{\mathcal{N}}_3.\\
				Q_3 &= \tilde{\mathcal{N}}_3 \tilde Q_1 (I+K_3)^{-1} \mathcal{N}_3.
		\end{aligned}
	\end{equation}
	Since
	\[
		\Gext(-1,-1,2,0,N,2,1) \subset \Gext(-1,-1,2,1,N,2,1),
	\]
	it is easy to see that 
	\begin{equation}\label{p34II}
		\begin{aligned}
			\sigma(P_3)  &\in \Gint(0,0,N,0) \cap \Gmid(0,0,0,N,2,1)\\
			& \qquad \cap \Gext(0,0,0,0,N,2,1),\\
			\sigma(P_4)  &\in \Gint(0,0,N,1) \cap \Gmid(0,0,1,N,0,1)\\
			& \qquad \cap \Gext(-1,-1,2,1,N,2,1),\\
			\sigma(Q_3) &\in \Gint(1,1,N,1), \sigma(Q_3) \equiv 0 \text{ in } \zmid(2N) \cup \zext(2N).
		\end{aligned}
	\end{equation}
	Let us summarize the above discussion in the following proposition.
		\begin{prop}\label{D3}
		The solution $V_2(t,x)$ to the first order system (\ref{FOS20}) is given by
		\[
		V_2(t,x) = (I + K_3(t,x,D_x))^{-1}\mathcal{N}_3(t,x,D_x){V}_3(t,x)
		\]	
		where $V_3(t,x)$ is solution to first order system
		\begin{equation}\label{FOS3}
			(D_{t} -\mathcal{D}_1 +P_{3} +P_4 +Q_3)V_{3}= F_3, \; V_3(0,x)= \tilde{\mathcal{N}_3}(0,x,D_x)V_2(0,x) 
		\end{equation}
		The matrix pseudodifferential operator $\mathcal{D}_1$ is as in Proposition \ref{D1} while $P_{3}, P_4, Q_3$ are as in (\ref{p34I})-(\ref{p34II}).
	\end{prop}
	\begin{rmk}
		The localization technique in our case has led to such a factorization where the regularity in the $t$-variable was lost along the way. Infact, after the localization procedure no new control of the $t$-derivative of the symbol $P_4$ was available.
	\end{rmk}

	\begin{rmk}
		Let us explain the philosophy of our approach. We use a careful amalgam of a localization technique on the extended phase space (defined in Section \ref{zones}) and the diagonalization procedure already available in the literature (see \cite{yag,KuboReis}) to handle the singularity. Note that we have restricted the singularity arising from the first $t$-derivative to $\zmid(N)$ and the one from the second $t$-derivative to $\zext(N).$ This kind of localization of the singularities allows one to come with a function (see (\ref{psi2})) with {\itshape good} estimate as in (\ref{psi3}) that majorizes the symbol of the operator $P_4 + Q_3$. It is woth noting that we do not need the so called \enquote{perfect diagnalization} in our analysis. 
	\end{rmk}

	\subsection{An Upper Bound for the Lower Order Terms with Singularity} \label{major}
	In this section, we define a function that majorizes $ \sigma(P_4) + \sigma(Q_3)$.	
	 Consider a smooth function $\mathfrak{M}_{1} \var$ of the form
	\begin{equation}\label{psi2}
		\begin{aligned}
				\mathfrak{M}_{1}\var = &
				= \kappa \Big( \chi(t/t_{x,\xi}) \om\japxik \tilde \theta(t)+ (1-\chi(t/t_{x,\xi}) \Big( 	\chi(t/\tilde{t}_{x,\xi})\frac{\theta(t)}{t} \\
				& \qquad \qquad + \frac{(1-\chi(t/\tilde{t}_{x,\xi}))}{\P\japxik} 	\frac{\theta(t)^2}{t^2} e^{\psi(t)} \tilde \theta(t)^2  \Big)\Big)
		\end{aligned}
	\end{equation}
   where $\kappa>0$ is chosen in such way that we have 
   \[
   		|\sigma(P_4)| + |\sigma(Q_3)|\leq \mathfrak{M}_{1}.
   \]
   From Propositions \ref{p1} - \ref{p3}, $t^{1-\varepsilon} \mathfrak{M}_{1} \in  C([0,T]; G^{1,1}(\Phi,\go)),$ for every $\varepsilon \in(0,1).$ 
   
   When $\tilde \theta(t)$ is unbounded near $t=0$, we use following estimate
   \begin{equation}\label{rel}
   	\int_{0}^{t} {\tilde \theta}(s)ds \lesssim t \tilde{\theta}(t).
   \end{equation}
   Observe that the estimate (\ref{rel}) is natural in the context of logarithmic-type functions. 
   Using the estimate (\ref{rel}), we readily have
    \begin{align*}
   	\int_0^{2t_{x,\xi}} \vert \partial_\xi^\alpha D_x^\beta \mathfrak{ M}_{1}(s,x,\xi) \vert ds 
   	&\leq 	\kappa_{\alpha\beta}  \om \P^{-\vert \beta \vert} \japxik^{1-\vert \alpha \vert} 	\int_0^{2t_{x,\xi}} \tilde \theta(t)dt\\
   	&\leq 	\kappa_{\alpha\beta}  \om \P^{-\vert \beta \vert} \japxik^{1-\vert \alpha \vert} 2t_{x,\xi} \tilde \theta(2t_{x,\xi}) \\
   	&\leq 	\kappa_{\alpha\beta}  \P^{-\vert \beta \vert} \japxik^{-\vert \alpha \vert} \tilde \theta(h) \theta(h) .
   	\end{align*}
   Similarly using the definitions of $t_{x,\xi}$ and $\tilde t_{x,\xi},$ we have the following estimates
   
   \begin{align*}
   		\int_{t_{x,\xi}}^{2\tilde t_{x,\xi}} \vert \partial_\xi^\alpha D_x^\beta  \mathfrak{ M}_{1}(s,x,\xi) \vert ds &\leq \kappa_{\alpha\beta}  \P^{-\vert \beta \vert} \japxik^{-\vert \alpha \vert} \left| \int_{t_{x,\xi}}^{2\tilde t_{x,\xi}} \frac{\theta(s)}{s} ds \right| \\
   		&\leq \kappa_{\alpha\beta} \P^{-\vert \beta \vert} \japxik^{-\vert \alpha \vert} 	 \theta(h) \left| \int_{t_{x,\xi}}^{2\tilde t_{x,\xi}} \frac{1}{s} ds \right|\\
   		&\leq \kappa_{\alpha\beta} \P^{-\vert \beta \vert} \japxik^{-\vert \alpha \vert}   \theta(h)(\ln 2 + \vert\ln\tilde\theta(h)\vert + \psi(h)) ,
   	 \end{align*}
   	 \begin{align*}
   		\int^{T}_{\tilde t_{x,\xi}} \vert \partial_\xi^\alpha D_x^\beta  \mathfrak{ M}_{1} (s,x,\xi) \vert ds 
   		&\leq \kappa_{\alpha\beta}  \P^{-1-\vert \beta \vert} \japxik^{-1-\vert \alpha \vert} 	\left| \int^{T}_{\tilde t_{x,\xi}} \frac{\theta(s)^2}{s^2} e^{\psi(s)} \tilde \theta(s)^2ds \right| \\
   		&\leq \kappa_{\alpha\beta}  \P^{-1-\vert \beta \vert} \japxik^{-1-\vert \alpha \vert}  \theta(h)^2 \tilde\theta(h)^2e^{\psi(h)}  \left| \int^{T}_{\tilde t_{x,\xi}} \frac{1}{s^2} ds \right|   \\
   		&\leq \kappa_{\alpha\beta}  \P^{-1-\vert \beta \vert} \japxik^{-1-\vert \alpha \vert}  \theta(h)^2 \tilde\theta(h)^2e^{\psi(h)} \frac{\P\japxik}{N\tilde\theta(h) \theta(h)e^{\psi(h)}}\\
   		&\leq \kappa_{\alpha\beta}  \P^{-\vert \beta \vert} \japxik^{-\vert \alpha \vert} \tilde\theta(h) \theta(h). 
   \end{align*}
	Thus we have,
	\begin{equation}
		\label{psi3}
		\int_{0}^{T} \vert D_x^\beta \partial_\xi^\alpha  \mathfrak{ M}_{1}(t,x,\xi)\vert dt \leq \kappa_{\alpha\beta} \Theta(x,\xi) \P^{-\vert \beta \vert} \japxi^{-\vert \alpha \vert} .
	\end{equation}
	The function $\mathfrak{M}_1$ plays an important role in performing conjugation and thus in quantifying the loss of regularity.
	
	\subsection{Infinite Loss of Regularity via Energy Estimate}\label{Energy1}
	In this section, we derive an energy estimate when (\ref{ineq2}) is satisfied and show that the loss is infinite. The case for finite loss i.e., when (\ref{ineq}) is satisfied follows in similar lines.
	
	Consider the operator $L$ defined by
	\begin{equation}
		L= D_t - \mathcal{D}_1 + P_3 + P_4 + Q_3,
	\end{equation}
	the matrix pseudodifferential operators are as in Proposition \ref{D3}. 
	Then the first order system (\ref{FOS3}) is equivalent to 
	\begin{equation}
		\label{FOS4}
		LV_3= F_3, \; V_3(0,x) =\tilde{ \mathcal{N}}_3(0,x,D_x) V_2(0,x).
	\end{equation}
	Note that to prove the estimate (\ref{est}), it is sufficient to consider the case $s=(0,0)$ as the operator $\P^{s_2} \la D\rak^{s_1}L\la D\rak^{-s_1}\P^{-s_2}$, where $s=(s_1,s_2)$ is the index of the Sobolev space, satisfies the same hypotheses as $L$. 
	
	We perform a change of variable, which allows us to control lower order terms. We set 
	\begin{equation}\label{one}
		V_4(t,x)= W_{1}(t,x,D_x){V_3}(t,x),
	\end{equation}
	where $W_{1}$ is a infinite order pseudodifferential operator with 
	\[
	\sigma(W_{1})(t,x,\xi) =\exp{\Big(\int_t^{T^*_1}{\mathfrak{M}_{1}(r,x,\xi)dr\Big)}}
	\]
    for $\mathfrak{M}_{1}$ is as in (\ref{psi2}).
	We have $V_4(0,x) = W_1(0,x,D_x) V_3(0,x)$ and for $0<t\leq T$
	\begin{equation*}
		||{V_3}(t)\Vl(0e,\tilde\kappa_0) \leq C ||V_4(t)||_{L^2},
	\end{equation*}
	where $ C>0$ and $e=(1,1)$. Here $\tilde\kappa_{0}$ is the constant $\kappa_{\alpha\beta}$ in (\ref{psi3}) with $\alpha=\beta=0$. Let $\tilde{W}_{1}(t,x,D_x)$ be such that $\sigma(\tilde{W}_{1}) = \exp{\Big(-\int_t^{T^*_1} \mathfrak{M}_{1}(r,x,\xi)dr\Big)}$. Then, by Theorem \ref{conju},
	\[
	\begin{aligned}
		\tilde{W}_{1}(t,x,D_x) {W}_{1}(t,x,D_x) &= I + K_4^{(1)}(t,x,D_x), \\
		{W}_{1}(t,x,D_x) \tilde{W}_{1}(t,x,D_x) &= I + K_4^{(2)}(t,x,D_x)
	\end{aligned}
	\]
	where $\T^{-1} \sigma(K_4^{(l)}) \in C([0,T]; \GT(-1,-1)), l=1,2.$ We choose $k>k_4$ for large $k_4$ so that the operator norm of $K_4^{(l)}$ is strictly lesser than $1.$ This guarantees
	\[
	(I + K_4^{(l)}(t,x,D_x))^{-1} = \sum_{j=0}^{\infty} (-1)^j K_4^{(l)}(t,x,D_x)^j \in C([0,T]; OP\GT(0,0)).
	\]		
	For the sake of simplicity, let us denote the operators $\big(I + K_4^{(1)}\big)^{-1}$ and $\big(I + K_4^{(2)}\big)^{-1}$ by $\mathcal{K}_{1}$ and $\tilde{\mathcal{K}}_{1},$ respectively.
	
	In the following we arbitrarily fix $\varepsilon \in (0,1)$ when $\tilde\theta$ is unbounded while for the bounded case we set $\varepsilon=1.$ We see that the pseudodifferential system (\ref{FOS4}) is equivalent to
	\begin{equation}
		\label{FOS5}
		\tilde{L}(t,x,D_t,D_x)V_4(t,x)= F_4(t,x), 
	\end{equation} 
	where $	\tilde{L} = L- i\mathfrak{M}_{1}(t,x,D_x) I + P_5(t,x,D_x),$ $F_4 = W_{1}(t,x,D_x) F_3$ and 
	\begin{align*}  
		\sigma(P_5) &=   \tilde{\mathcal{K}}_{1}W_{1}\mathcal{K}_{1}^{-1}\left( \big( P_3+P_4+Q_3 -\mathcal{D}_1\big)\mathcal{K}_{1} + D_t\mathcal{K}_{1}  \right) \tilde{W}_{1} + (\tilde{\mathcal{K}}_{1}W_{1}D_t\tilde{W}_{1} + i \mathfrak{M}_{1}  I )	\\
		&\qquad \qquad \quad - (P_3 +P_4+Q_3 -\mathcal{D}_1).		
	\end{align*}
	From Theorem \ref{conju}, $t^{1-\varepsilon}\T^{-1} \sigma(P_5) \in {C([0,T]; G^{0,0}(\Phi,\gt)}.$ $T_1^*$ is chosen in such a way that all the above conjugations with operator $W_1$ are valid in view of Theorem \ref{conju}.
	
	Let $\kappa_1>0$ be such that 
	\begin{equation}
		\label{mj}
		\vert \sigma(P_5) \vert \leq \mathfrak{M}_{2} \var =\kappa_1 t^{-1+\varepsilon}\Theta(x,\xi).
	\end{equation}	
	We make a further change of variable
	\begin{equation}\label{two}
		V_5(t,x) = 	W_2(t,x,D_x)V_4(t,x)
	\end{equation}
	where $W_{2}$ is an infinite order pseudodifferential operator with 
	\[
	\sigma(W_{2})(t,x,\xi) =  \exp\left\{ \int^{T_2^*}_t \mathfrak{M}_{2}(s,x,\xi) ds \right\}.
	\]	
	We have $V_5(0,x) = W_2(0,x,D_x) V_4(0,x)$ and 
	\begin{equation*}
		||{V_4}(t)\Vl(0e,\kappa_1^*) \leq C ||V_5(t)||_{L^2}, 
	\end{equation*}	
    where $\kappa_1^*(t) = \tilde\kappa_1({T^*_2}^\varepsilon-t^\varepsilon)/\varepsilon.$
	Let $\tilde{W}_{2}(t,x,D_x)$ be such that $\sigma(\tilde{W}_{2}) = e^{-\int^{T_2^*}_t\mathfrak{M}_{2}(s,x,\xi) ds}.$ Then, by Theorem \ref{conju}, 
	\[
	\begin{aligned}
		\tilde{W}_{2}(t,x,D_x) {W}_{2}(t,x,D_x) &= I + K_5^{(1)}(t,x,D_x), \\
		{W}_{2}(t,x,D_x) \tilde{W}_{2}(t,x,D_x) &= I + K_5^{(2)}(t,x,D_x),
	\end{aligned}
	\]
	where $\Theta(x,\xi)^{-1} \sigma(K_5^{(l)}) \in C([0,T]; G^{-1,-1}(\Phi,\gt)  ), l=1,2.$ 	
	We choose $k>k_5$ for large $k_5$ so that the operator norm of $K_5^{(l)}$ is strictly lesser than $1.$ This guarantees
	\[
	(I + K_5^{(l)}(t,x,D_x))^{-1} = \sum_{j=0}^{\infty} (-1)^j K_5^{(l)}(t,x,D_x)^j \in C([0,T]; OP\GT(0,0)).
	\]	
	Let us denote the operators $\big(I + K_5^{(1)}\big)^{-1}$ and $\big(I + K_5^{(2)}\big)^{-1}$ by $\mathcal{K}_{2}$ and $\tilde{\mathcal{K}}_{2},$ respectively.
	
	We see that the pseudodifferential system (\ref{FOS5}) is equivalent to
	\begin{equation}
		\label{FOS6}
		\big(\tilde{L}(t,x,D_t,D_x) - i\mathfrak{M}_{2}(t,x,D_x) I + P_6(t,x,D_x) \big)V_5(t,x)= F_5(t,x), 
	\end{equation} 
	where $F_5 = W_{2}(t,x,D_x) F_4$ and 
	\begin{align*}  
		\sigma(P_6) &=   \tilde{\mathcal{K}}_{2}W_{2}\mathcal{K}_{2}^{-1}\left( \big(  P_3+P_4+Q_3+P_5 -\mathcal{D}_1 \big) \mathcal{K}_{2} + D_t\mathcal{K}_{2}  \right) \tilde{W}_{2} \\& \qquad- (P_3 +P_4 +Q_3+P_5 -\mathcal{D}_1)
		 + ( \tilde{\mathcal{K}}_{2}W_{2}D_t\tilde{W}_{2} + i \mathfrak{M}_2I ),
	\end{align*}
	with $  t^{1-\varepsilon}\sigma(P_6)  \in {C([0,T]; \GT(0,0)}.$ $T_2^*$ is chosen in such a way that all the above conjugations with operator $W_2$ are valid in view of Theorem \ref{conju}.
	
	Let us write down the first order system (\ref{FOS6}) explicitly as below
	\begin{equation*}
		\partial_tV_2 = \left(i\mathcal{D}_1-\left( \mathfrak{M_1}I + iP_4 + iQ_3 + \mathfrak{M_2}I +iP_5  \right)- iP_6 -iP_3 \right)V_5+ i F_5.
	\end{equation*}
	Here, the diagonal matrix operator, $i\mathcal {D}_1$ is of the form 
	\[
	i\mathcal{D}_1 = \text{diag}\left\{i\lambda_1, i\lambda_2\right\}.
	\]
	Observe that the symbol $d(t,x,\xi)$ of the operator $i\mathcal{D}_1-i\mathcal{D}_1^*$ is such that 
	\begin{linenomath*}
		\[
		d \in \Gint(0,0,N,0) \cap \Gmid(0,0,0,N,1,1) \cap \Gext(0,0,0,0,N,1,1).
		\]
	\end{linenomath*}
	It follows from Propositions \ref{p1}-\ref{p3} and Remark \ref{lb} that 
	\begin{linenomath*}
		\[
		t^{1-\varepsilon}d \in C([0,T];\G(0,0)).
		\]
	\end{linenomath*}
	Hence, by Calderon-Vaillancourt theorem ,
	\begin{equation}
		\label{diag}
		\begin{aligned}
			2 \Re \la i\mathcal{D}_1V_5,V_5 \ra &\leq \frac{C}{t^{1-\varepsilon}} \la V_5,V_5 \ra, \; C>0.
		\end{aligned}
	\end{equation}

	Further,
	by the choice of the functions, $\mathfrak{M_1}\var+\mathfrak{M_2}(x,\xi)$
	\[
	\Re(  (\mathfrak{M_1}\var+\mathfrak{M_2}(t,x,\xi))I + i\sigma(P_4) +i\sigma(Q_3) +i\sigma(P_5)  )\geq 0.
	\]
	Assuming
	\[
	\om \lesssim \P,
	\]
	we apply sharp G\r{a}rding inequality (see Thereom \ref{sg} in Appendix II and \cite[Theorem 18.6.14]{Horm}) to $2\Re(  \mathfrak{M_1}I + iP_4 +iQ_3)$ with the metric $\go$ and Planck function $h(x,\xi)=(\Phi(x) \japxik) ^{-1}$ and to $2\Re(  \mathfrak{M_2}I + iP_5  )$ with the metric $\gt$ and Planck function $\tilde h(x,\xi)=\T(\Phi(x) \japxik) ^{-1}.$ We obtain
	\begin{equation}
		\label{psiA}
		2\Re \la (  (\mathfrak{M_1}+\mathfrak{M_2})I + iP_4+iQ_3+iP_5  )V_5,V_5 \ra_{L^2} \geq - \frac{C}{t^{1-\varepsilon}} \la V_5,V_5 \ra_{L^2}, \; C>0.
	\end{equation}
	Since $t^{1-\varepsilon}\sigma(iP_3+iP_6)$ is uniformly bounded, by Calderon-Vaillancourt theorem we have
	\begin{equation}
		\label{R2}
		-2\Re \la i(P_3+ P_6)V_5,V_5\ra \leq \frac{C}{t^{1-\varepsilon}} \la V_5,V_5\ra.
	\end{equation}
	From (\ref{diag})-(\ref{R2}) it follows that
	\begin{equation*}
		\label{K2}
		2 \Re \la ( i\mathcal{D}_1-\left( (\mathfrak{M_1}+\mathfrak{M_2})I + iP_4+iQ_3 +iP_5 \right)- iP_6-iP_3 ) V_5,V_5 \ra_{L^2} \leq  \frac{C}{t^{1-\varepsilon}} \la V_5,V_5 \ra_{L^2}.
	\end{equation*}
	This yields 
	\[
	\partial_t ||V_5(t, \cdot)||_{L^2}^2 \leq C( t^{-1+\varepsilon}\Vert V_5(t,\cdot)\Vert_{L^2}^2 + \Vert F_5(t, \cdot) \Vert_{L^2}^2), \; 0\leq t \leq T.
	\]
	Considering the above inequality as a differential inequality, we apply Gronwall's lemma and obtain that
	$$
	\|V_5(t, \cdot)\|_{L^{2}}^{2} \leq C_{\varepsilon} \left( \|V_5(0, \cdot)\|_{L^{2}}^{2}+ \int_{0}^{t}\left\|F_5(\tau, \cdot)\right\|_{L^{2}}^{2} d \tau \right), \quad t \in [0,T^*].
	$$
	If $\tilde\theta$ is unbounded $C_\varepsilon = C'e^{T^\varepsilon/\varepsilon}$ for a fixed $\varepsilon \in (0,1),$ else $C_\varepsilon = C''$ for some $C',C''>0.$ $T^*=\min\{T_1^*,T_2^*,\delta_1,\delta_2\}$ where $\delta_1$ and $\delta_2$ are related to the initial datum and the right hand side of the Cauchy problem (\ref{eq1}).
	
	This proves the well-posedness of the auxiliary Cauchy problem (\ref{FOS6}). Note that the solution $V_3$ to (\ref{FOS4}) belongs to $C\left([0, T] ; \sobol(s,{\tilde\kappa(t)})\right),\tilde\kappa(t) = \tilde\kappa_0 + \tilde\kappa_1^*(t)  .$ Returning to our original solution $u=u(t, x)$ we obtain the estimate (\ref{est}) with
	$$
	u \in C\left([0, T] ;  \sobol(s+e,{\tilde\kappa(t)}) \right) \bigcap C^{1}\left([0, T] ; \sobol(s,{\tilde\kappa(t)})  \right).
	$$
	This proves Theorem  \ref{result3}.
	
	\subsection{Finite Loss of Regularity via Energy Estimate}\label{Energy2}
	
	Let us assume that (\ref{ineq}) is satisfied. The energy estimate (\ref{est2}) in such case follows in similar lines to the infinite loss case except that we need to replace the changes of variable (\ref{one}) and (\ref{two}) with the following counterparts
	\[
	\begin{aligned}
		V_4(t,x) &= W'_{1}(t,x,D_x){V_3}(t,x),\\
		V_5(t,x) &= W'_2(t,x,D_x)V_4(t,x),
	\end{aligned}
	\] 
	where
	\[
	\begin{aligned}
		\sigma(W'_{1})(t,x,\xi) &= \exp\left\{-\int_{0}^{t} \mathfrak{M}_{1}(r,x,\xi)dr\right\}, \\
		\sigma(W'_{2})(t,x,\xi) &=   \exp\left\{-\int_{0}^{t} \mathfrak{M}_{2}(r,x,\xi)dr \right\}
	\end{aligned}
	\]
	for $\mathfrak{M}_{1}$ and $\mathfrak{M}_{2}$ as in (\ref{psi2}) and (\ref{mj}). Note that operators $W_1'$ and $W_2'$ are finite order pseudodifferential operators in view of (\ref{ineq}). Hence, depending on the order these operators we have zero, arbitrarily small or finite loss.

 \section{Cone Condition}\label{cone}
 Existence and uniqueness follow from the a priori estimate established in the previous section. It now remains to prove the existence of cone of dependence. 
 
 We note here that the $L^1$ integrability of the singularity plays a crucial in arriving at the finite propagation speed. The implications of the discussion in \cite[Section 2.3 \& 2.5]{JR} to the global setting suggest that if the Cauchy data in (\ref{eq1}) is such that $f \equiv 0$ and $f_1,f_2$ are supported in the ball $\vert x \vert \leq R$, then the solution to Cauchy problem (\ref{eq1}) is supported in the ball $\vert x \vert \leq R+\gamma_0 \om \tilde\theta(t)$. The quantity $t\tilde\theta(t)$ is bounded in $[0,T]$. The constant $\gamma_0$ is such that the quantity $\gamma_0 \om\tilde\theta(t)$ dominates the characteristic roots, i.e.,
 \begin{equation}
 	\label{speed}
 	\gamma_0= \sup\Big\{\sqrt{a(t,x,\xi)}\om^{-1} \tilde\theta(t)^{-1}:(t,x,\xi) \in[0,T] \times \R^n_x \times \R^n_\xi,\:|\xi|=1\Big\}.
 \end{equation}
 Note that the support of the solution increases as $|x|$ increases since $\om$ is monotone increasing function of $|x|$.
 
 In the following we prove the cone condition for the Cauchy problem $(\ref{eq1})$. Let $K(x^0,t^0)$ denote the cone with the vertex $(x^0,t^0)$:
 \begin{linenomath*}
 	\[
 	K(x^0,t^0)= \{(t,x) \in [0,T] \times \R^n : |x-x^0| \leq \gamma_0 \om \tilde\theta(t^0-t) (t^0-t)\}.
 	\]
 \end{linenomath*}
 Observe that the slope of the cone is anisotropic, that is, it varies with both $x$ and $t$.
 
 \begin{prop}
 	The Cauchy problem (\ref{eq1}) has a cone dependence, that is, if
 	\begin{equation}\label{cone1}
 		f\big|_{K(x^0,t^0)}=0, \quad f_i\big|_{K(x^0,t^0) \cap \{t=0\}}=0, \; i=1, 2,
 	\end{equation}
 	then
 	\begin{equation}\label{cone2}
 		u\big|_{K(x^0,t^0)}=0.
 	\end{equation}
 \end{prop}
 \begin{proof}
 	Consider $t^0>0$, $\gamma_0>0$ and assume that  (\ref{cone1}) holds. We define a set of operators $P_\varepsilon(t,x,\partial_t,D_x), 0 \leq \varepsilon \leq \varepsilon_0$ by means of the operator $P(t,x,\partial_t,D_x)$ in (\ref{eq1}) as follows
 	\begin{linenomath*}
 		\[
 		P_\varepsilon(t,x,D_t,D_x) = P(t+\varepsilon,x,D_t,D_x), \: t \in [0,T-\varepsilon_0], x \in \R^n,
 		\]
 	\end{linenomath*}
 	and $\varepsilon_0 < T-t^0$, for a fixed and sufficiently small $\varepsilon_0$. For these operators we consider Cauchy problems
 	\begin{linenomath*}
 		\begin{alignat}{2}
 			P_\varepsilon v_\varepsilon & =f,  &&  t \in [0,T-\varepsilon_0], \; x \in \R^n,\\
 			\partial_t^{k-1}v_\varepsilon(0,x)& =f_k(x),\qquad && k=1,2.
 		\end{alignat}
 	\end{linenomath*}
 	Note that $v_\varepsilon(t,x)=0$ in $K(x^0,t^0)$ and $v_\varepsilon$ satisfies the a priori estimate (\ref{est2}) for all $t \in[0,T-\varepsilon_0]$. Further, we have 
 	\begin{linenomath*}
 		\begin{alignat}{2}
 			P_{\varepsilon_1} (v_{\varepsilon_1}-v_{\varepsilon_2}) & = (P_{\varepsilon_2}-P_{\varepsilon_1})v_{\varepsilon_2},\qquad  &&  t \in [0,T-\varepsilon_0], \; x \in \R^n,\\
 			\partial_t^{k-1}(v_{\varepsilon_1}-v_{\varepsilon_2})(0,x)& = 0,\qquad && k=1,2.
 		\end{alignat}
 	\end{linenomath*}
 	Since our operator is of second order, for the sake of simplicity we denote $b_{j}(t,x)$, the coefficients of lower order terms, as $a_{0,j}(t,x), 1 \leq j \leq n,$ and $b_{n+1}(t,x)$ as $a_{0,0}(t,x)$. Let $a_{i,0}(t,x) =0, \; 1 \leq i \leq n.$ Substituting $s-e$ for $s$ in the a priori estimate (\ref{est2}), we obtain
 	\begin{equation}\label{cone3}
 		\begin{aligned}
 			&\sum_{j=0}^{1} \Vert  \partial_t^j(v_{\varepsilon_1}-v_{\varepsilon_2})(t,\cdot) \Vl(s-je,{-\kappa(t)}) \\
 			&\leq C \int_{0}^{t}\Vert (P_{\varepsilon_2}-P_{\varepsilon_1})v_{\varepsilon_2}(\tau,\cdot) \Vl(s-e,-{\kappa(\tau)}) \;d\tau\\
 			&\leq C \int_{0}^{t} \sum_{i,j=0 }^{n}\Vert (a_{i,j}(\tau+\varepsilon_1,x) - a_{i,j}(\tau+\varepsilon_2,x))D_{ij} v_{\varepsilon_2}(\tau,\cdot) \Vl(s-e,-{\kappa(\tau)}) \;d\tau,
 		\end{aligned}
 	\end{equation}
 	where $D_{00}=I, D_{i0}=0, i\neq 0, D_{0j}=\partial_{x_j},j \neq 0$ and $D_{ij}=\partial_{x_i} \partial_{x_j}, i,j \neq 0$. Similar estimate holds if we had used the a priori estimate (\ref{est}) instead of (\ref{est2}).
 	
 	Using the Taylor series approximation in $\tau$ variable, we have
 	\begin{linenomath*}
 		\begin{align*}
 			|a_{i,j}(\tau+\varepsilon_1,x) - a_{i,j}(\tau+\varepsilon_2,x)| &= \Big|\int_{\tau+\varepsilon_2}^{\tau+\varepsilon_1} (\partial_ta_{i,j})(r,x)dr \Big|\\
 			&\leq \om^{2} \Big|\int_{\tau+\varepsilon_2}^{\tau+\varepsilon_1}\frac{\theta(r)}{r} dr\Big|\\
 			&\leq \om^{2}|E(\tau,\varepsilon_1,\varepsilon_2)|,
 		\end{align*}
 	\end{linenomath*}
 	where
 	\begin{linenomath*}
 		\[
 		E(\tau,\varepsilon_1,\varepsilon_2) =	\frac{1}{2}	
 		\left(\ln \Bigg(1+\frac{\varepsilon_1-\varepsilon_2}{\tau+\varepsilon_2}\Bigg)\right)^{\varrho_2+1} . 
 		\]
 	\end{linenomath*}
 	Note that $\om \lesssim \P$ and $E(\tau,\varepsilon,\varepsilon)=0$.
 	Then right-hand side of the inequality in (\ref{cone3}) is dominated by
 	\begin{linenomath*}
 		\begin{equation*}\label{cone4}
 			C \int_{0}^{t} |E(\tau,\varepsilon_1,\varepsilon_2)| \Vert  v_{\varepsilon_2}(\tau,\cdot)\Vl(s+e,{-\kappa(\tau)})\;d\tau,
 		\end{equation*}
 	\end{linenomath*}
 	in the case of finite loss of regularity while for the case of infinite loss by
 	\begin{linenomath*}
 		\begin{equation*}\label{cone5}
 			C \int_{0}^{t} |E(\tau,\varepsilon_1,\varepsilon_2)| \Vert  v_{\varepsilon_2}(\tau,\cdot)\Vl(s+e,{\tilde \kappa(\tau)})\;d\tau, 
 		\end{equation*}
 	\end{linenomath*}
 	where $C$ is independent of $\varepsilon$. By definition, $E$ is $L^1$-integrable in $\tau$.
 	
 	The sequence $v_{\varepsilon_k}$, $k=1,2,\dots$ corresponding to the sequence $\varepsilon_k \to 0$ is in the space
 	\begin{linenomath*}
 		\[
 		C\Big([0,T^*];\sobol(s,-{\kappa(t)}) \Big) \bigcap C^{1}\Big([0,T^*];  \sobol(s-e,-{\kappa(t)}) 	\Big), \quad T^*>0,
  		\]
 	\end{linenomath*}
 	or
 	\begin{linenomath*}
 		\[
 			C\Big([0,T^*];\sobol(s,{\tilde\kappa(t)})  \Big) \bigcap C^{1}\Big([0,T^*]; \sobol(s-e,{\tilde\kappa(t)}) \Big), \quad T^*>0,
 		\]
 	\end{linenomath*}
 	depending on the loss
 	and $u=\lim\limits_{k\to\infty}v_{\varepsilon_k}$ in the above space and hence, in $\mathcal{D}'(K(x^0,t^0))$. In particular,
 	\begin{linenomath*}
 		\[
 		\la u,\varphi\ra = \lim\limits_{k\to\infty} \la v_{\varepsilon_k}, \varphi \ra =0,\; \forall \varphi \in \mathcal{D}(K(x^0,t^0))
 		\]
 	\end{linenomath*}
 	gives (\ref{cone2}) and completes the theorem. 
 	
 \end{proof}
 	\begin{rmk}
 		It is worth noting that while the regularity of the solution is dictated by $\P$, the cone condition and there by the support of the solution is controlled by the weight function $\om.$
 	\end{rmk}

\section{Existence of Counterexamples}\label{CE}

From Theorem \ref{result3} it is evident that when the condition (\ref{ineq}) is violated one can expect an infinite loss quantified by an infinite order pseudodifferential operator. In this section we show that the set of coefficients for which the associated Cauchy problems exhibit infinite loss is residual in the metric space $\mathcal{C}(\mu_1,\mu_2,\theta,\psi)$ (see Definition \ref{ps}). To this end, we extend the techniques developed by Ghisi and Gobbino \cite[Section 4]{GG} to our global setting.

We consider a Cauchy problem of the form
\begin{equation}
	\label{ce}
	\begin{aligned}
		&\partial_t^2 u(t,x) + c(t)A(x,D_x)u=0, \quad (t,x) \in [0,T] \times \R^n,\\
		&u(0,x) = 0, \quad \partial_tu(0,x) = f(x),
	\end{aligned}
\end{equation}
where $A(x,D_x) = \japx(I-\triangle_x)\japx$ is a G-elliptic, positive, self-adjoint operator with the domain $D(A)=\{u \in L^2(\R^n): Au \in L^2(\R^n)\}$ and $c \in \mathcal{C}(\mu_1,\mu_2,\theta,\psi)$, which is defined below. 

\begin{defn}\label{ps}
	We denote $\mathcal{C} \left(\mu_{1}, \mu_{2}, \theta, \psi \right)$ as the set of functions $c \in C \left(\left[0, T\right]\right) \cap C^{2}\left(\left(0, T\right]\right)$ that satisfy the following growth estimates 
	\begin{align}
		\label{e1}
		0<\mu_1& \leq c(t) \leq  \mu_2, \quad t \in [0,T],\\
		\label{e2}
		\vert c'(t)\vert &\leq C \frac{\theta(t)}{t}, \qquad t \in (0,T], \\
		\label{e4}
		\vert c''(t)\vert &\leq C \frac{\theta(t)^2}{t^2}\psi(t), \qquad t \in (0,T],
	\end{align}
	for monotone decreasing functions $\theta,\psi:(0,+\infty) \to (0,+\infty)$ satisfying
	\begin{equation}
		\label{e3}
		\lim_{t\to 0^+} \frac{\theta(t)\psi(t)}{\vert \ln t \vert} = +\infty.
	\end{equation}	
\end{defn}
The set $\mathcal{C} \left(\mu_{1}, \mu_{2}, \theta,\psi \right)$ is a complete metric space with respect to the metric
\begin{linenomath*}
	\[
	\begin{aligned}
		d_{\mathcal{C}} \left(c_{1}, c_{2}\right):=\sup _{t \in\left(0, 	T\right)}&\left|c_{1}(t)-c_{2}(t)\right|+\sup _{t \in\left(0, T\right)}\left\{\frac{t^{2}}{\theta(t)}\left|c_{1}^{\prime}(t)-c_{2}^{\prime}(t)\right|\right\} \\
		&+ \sup _{t \in\left(0, T\right)}\left\{\frac{t^{3}e^{-\psi(t)}}{\theta(t)^2}\left|c_{1}^{\prime\prime}(t)-c_{2}^{\prime\prime}(t)\right|\right\}.
	\end{aligned}
	\]
\end{linenomath*}
A sequence $c_{n}$ converges to $c_{\infty}$ with respect to $d_{\mathcal{C}}$ if and only if $c_{n} \rightarrow c_{\infty}$ uniformly in $\left[0, T\right]$, and for every $\tau \in\left(0, T\right)$,  $c_{n}^{\prime} \rightarrow c_{\infty}^{\prime}$ and $c_{n}'' \rightarrow c_{\infty}''$uniformly in $\left[\tau, T\right]$. In view of $L^1$ integrability of $c(t)$ in our setting, convergence with respect to $d_{\mathcal{C}}$ implies convergence in $L^{1}\left(\left(0, T\right)\right)$.

The main aim of this section is to prove the following result.
\begin{thm} \label{result2}
	The interior of the set of all $c \in \mathcal{C} \left(\mu_{1}, \mu_{2},\theta,\psi\right)$ for which the Cauchy problem (\ref{ce}) exhibits an infinite loss is nonempty.	
\end{thm}

Since the operator $A$ is positive and G-elliptic and the symbol
\[
\sigma(A^{\alpha}) \sim \japx^{2\alpha} \japxik^{2\alpha} + \textnormal{ lower order terms, }
\]
the Sobolev spaces associated to the operator $A$ are $H^{2\alpha,2\alpha}_{\japx,k},\alpha \in \R,$ defined in (\ref{Sobo}).
We characterize the Sobolev spaces $H^{2m,2m}_{\japx,k}(\R^n),m \in \mathbb{Z},$ using the spectral theorem \cite[Theorem 4.2.9]{nicRodi} for the pseudodifferential operators on $\R^n.$  The theorem guarantees the existence of an orthonormal basis $(e_i(x))_{i=1}^{\infty}, e_i \in \mathcal{S}(\R^n),$ of $L^2(\R^n)$ and a nondecreasing sequence $(\ro_i)_{i=1}^{\infty} $ of nonnegative real numbers diverging to $+\infty$ such that $Ae_i(x)=\ro_i^2e_i(x).$

Using $\ro_{i}$s we identify $v(x)\in H^{2m,2m}_{\japx,k}(\R^n)$ with a sequence $(v_i)$ in weighted $ \ell^2$, where $v_i=\la v,e_i\ra_{L^2}$. One can prove the following proposition using Riesz representation theorem showing the correspondence between $H^{2m,2m}_{\japx,k}(\R^n)$ and a weighted $\ell^2$ space.

\begin{prop}
	Let $(v_{i})$ be a sequence  of real numbers and $m \in \mathbb{Z}.$ Then 
	\[
	\sum_{i=1}^{\infty}v_ie_i(x)\in H^{2m,2m}_{\japx,k}(\R^n) \quad \text{ if and only if} \quad 	\sum_{i=1}^{\infty}\ro_{i}^{2 m} v_{i}^{2}<+\infty.
	\]
\end{prop}

The solution to (\ref{ce}) is $u(t,x) =\sum_{i=1}^{\infty}u_i(t)e_i(x)$ where the functions $u_i(t)$ satisfy the decoupled system of ODEs
\begin{equation}
	\label{ce2}
	\begin{aligned}
		&u_i''(t) + c(t) \ro_i^2u_i(t) = 0, \quad i \in \mathbb{N}, \; t \in [0,T],\\
		&u_i(0)= 0, \quad u_i'(0) = f_{i},
	\end{aligned}
\end{equation}
for $\partial_tu(0,x)=f(x)=\sum_{i=1}^{\infty}f_ie_i(x)$.

We say that the solution to (\ref{ce}) experiences infinite loss of decay and derivatives if the initial velocity $f\in H^{2m,2m}_{\japx,k}$ for all $m \in \mathbb{Z}^+$ but $(u,\partial_tu) \notin H^{-2m'+1,-2m'+1}_{\japx,k} \times H^{-2m',-2m'}_{\japx,k}$ for any $m' \in \mathbb{Z}^+$ and for $t \in (0,T].$

In order to prove Theorem \ref{result2} we define a dense subset of $\mathcal{C} \left(\mu_{1}, \mu_{2}, \theta,\psi \right).$
\begin{defn}  We call $\mathcal{D}\left(\mu_{1}, \mu_{2}\right)$ the set of functions $c:\left[0, T\right] \rightarrow\left[\mu_{1}, \mu_{2}\right]$ for which there exists two real numbers $T_{1} \in\left(0, T\right)$ and $\mu_{3} \in\left(\mu_{1}, \mu_{2}\right)$ such that $c(t)=\mu_{3}$ for every $t \in\left[0, T_{1}\right]$.
\end{defn}
For the sake of simplicity, let us denote $\mathcal{D}\left(\mu_{1}, \mu_{2}\right)$ and $\mathcal{C}\left(\mu_{1}, \mu_{2}, \theta,\psi\right)$ by $\mathcal{D}$ and $\mathcal{C}$, respectively. From \cite[Prposition 4.7]{GG}, $\mathcal{D} \cap \mathcal{C}$ is dense in $\mathcal{C}.$ We note here that the weight factors $\frac{t^2}{\theta(t)}$ and $\frac{t^3e^{-\psi(t)}}{\theta(t)^2}$ appearing in the definition of the metric $d_{\mathcal{C}}$ plays a crucial role in proving the denseness.

Following the terminology of Ghisi and Gobbino \cite{GG}, we now introduce special classes of propagation speeds: universal and asymptotic activators. Let $\phi:(0,+\infty) \rightarrow(0,+\infty)$ be a function. 
\begin{defn} 
	A universal activator of the sequence $(\ro_{i})$ with rate $\phi$ is a propagation speed $c \in L^{1}\left(\left(0, T\right)\right)$ such that the corresponding sequence $(u_{i}(t))$ of solutions to 
	\begin{linenomath*}
		\[
		u_i^{''}(t) + c(t) \ro_i^2u_i(t) = 0, \quad 
		u_i(0)= 0, \quad u_i'(0) = 1,
		\]
	\end{linenomath*}
	satisfies
	\begin{linenomath*}
		$$
		\limsup _{i \rightarrow+\infty}\left(\left|u_{i}^{\prime}(t)\right|^{2}+\ro_{i}^{2}\left|u_{i}(t)\right|^{2}\right) \exp \left(-\phi\left(\ro_{i}\right)\right) \geq 1,  \quad \forall t \in\left(0, T\right].
		$$
	\end{linenomath*}
\end{defn}	
Then the solution $u$ to problem (\ref{ce}) is given by
\begin{linenomath*}
	\[
	u(t,x) = \sum_{i=1}^{\infty} f_iu_i(t)e_i(x).
	\]
\end{linenomath*}
\begin{defn}
	A family of asymptotic activators with rate $\phi$ is a family of propagation speeds $\left\{c_{\ro}(t)\right\} \subseteq L^{1}\left(\left(0, T\right)\right)$ with the property that, for every $\delta \in\left(0, T\right),$ there exist two positive constants $M_{\delta}$ and $\ro_{\delta}$ such that the corresponding family $\left\{u_{\ro}(t)\right\}$ of solutions to 
	\begin{linenomath*}
		\[
		u_{\ro}''(t) + c_{\ro}(t) \ro^2u_{\ro}(t) = 0, \quad 
		u_{\ro}(0)= 0, \quad u_{\ro}'(0) = 1,
		\]
	\end{linenomath*}
	satisfies
	\begin{linenomath*}
		$$
		\qquad\left|u_{\ro}^{\prime}(t)\right|^{2}+\ro^{2}\left|u_{\ro}(t)\right|^{2} \geq M_{\delta} \exp (2 \phi(\ro)), \quad \forall t \in\left[\delta, T\right], \quad \forall \ro \geq \ro_{\delta}.
		$$
	\end{linenomath*}
\end{defn}
Due to denseness of $\mathcal{D}$ in $\mathcal{C}$, for every $c \in \mathcal{D}$ there exists a family of asymptotic activators $(c_{\ro}) \subseteq \mathcal{C}$ with rate $\phi$ such that $c_{\ro} \rightarrow c$ with respect to $d_{\mathcal{C}} .$
The existence of families of asymptotic activators converging to elements of a dense set implies the existence of a residual set of  universal activators. Since the problem (\ref{ce}) exhibits an infinite loss of derivatives and decay whenever $c(t)$ is a universal activator, construction of couterexample amounts to the existence of such asymptotic activators.
Once the asymptotic activators are constructed and if $\phi$ is such that $\phi(\ro) \rightarrow+\infty$ as $\ro \rightarrow+\infty$, then \cite[Proposition 4.5]{GG} guarantees that the set of elements in $\mathcal{C}$ that are universal activators of the sequence $(\ro_{i})$ with rate $\phi$ is residual in $\mathcal{C}.$ In addition, if the function $\phi$ is such that
\begin{equation}\label{logphi}
	\lim _{\ro \rightarrow+\infty} \frac{\phi(\ro)}{\ln \ro}=+\infty,
\end{equation}
then by Proposition 4.3 in \cite{GG} one can show that for each of the universal activator $c(t)$ of sequence $(\ro_n)$ with rate $\phi$ the solution to the problem (\ref{ce}) exhibits infinite loss of regularity.

We are left with construction of asymptotic activators with rate $\phi$ satisfying (\ref{logphi}). We consider $T_{1}$ and $ \tilde \gamma$ such that $0<T_{1}<T$ and
$0<\mu_{1}<\tilde\gamma^{2}<\mu_{2},$ and define a initially constant speed $c_{*}:\left[0, T\right] \rightarrow\left[\mu_{1}, \mu_{2}\right]$ such that
\begin{linenomath*}
	$$
	c_{*}(t)=\tilde\gamma^{2},  \quad \forall t \in\left[0, T_{1}\right].
	$$
\end{linenomath*}
We set
\[ 
\begin{aligned}
	\theta_{\ro} &:=\min \left\{\theta\left(b_{\ro}\right), \ln \ro\right\}, \\
	\Gamma_{\ro} &:= \frac{\theta(\sqrt{\ro}) \psi(\sqrt{\ro})}{\ln \ro} , \\
	\psi_{\ro} &:= \min \left\{ \frac{1}{8} \ln \ro, \frac{1}{4}\psi\left( \frac{1}{\sqrt{\ro}} \right)  + \frac{1}{4} \ln \Gamma_\ro \right\}.
\end{aligned}
\]
For every large enough real number $\ro,$ let $a_{\ro}$ and $b_{\ro}$ be real numbers such that
\begin{linenomath*}
	\[
	a_{\ro}:=\frac{2 \pi}{\tilde\gamma \ro}\left\lfloor \ln \ro \exp(\psi_\ro)\right\rfloor, \quad \quad b_{\ro}:=\frac{2 \pi}{\tilde\gamma \ro}\left\lfloor\ln \ro \exp(2\psi_\ro)\right\rfloor.
	\]
\end{linenomath*}
where $\lfloor \alpha \rfloor$ stands for integer part of a real number $\alpha.$ Observe that 
\begin{linenomath*}
	$$
	0<a_{\ro}<2 a_{\ro}<\frac{b_{\ro}}{2}<b_{\ro}<T_{1}, \qquad
	\frac{\tilde\gamma \ro a_{\ro}}{2 \pi} \in \mathbb{N} \quad \text { and } \quad \frac{\tilde\gamma \ro b_{\ro}}{2 \pi} \in \mathbb{N}.
	$$
\end{linenomath*}
Let us choose a cutoff function $\tilde\nu: \mathbb{R} \rightarrow \mathbb{R}$ of class $C^{\infty}$ such that $0 \leq \tilde\nu(r)\leq 1$, $\tilde\nu(r)=0 $ for $r\leq 0$ and $\tilde\nu(r)=1$ for $r\geq1$. We define $\varepsilon_{\ro}:\left[0, T\right] \rightarrow \mathbb{R}$ as
\begin{linenomath*}
	$$
	\varepsilon_{\ro}(t):=\left\{\begin{array}{ll}
		0 & \text { if } t \in\left[0, a_{\ro}\right] \cup\left[b_{\ro}, T_{0}\right] \\
		\frac{\theta_{\ro}}{t} & \text { if } t \in\left[2 a_{\ro}, b_{\ro} / 2\right] \\
		\frac{\theta_{\ro}}{t} \cdot \tilde\nu\left(\frac{t-a_{\ro}}{a_{\ro}}\right) & \text { if } t \in\left[a_{\ro}, 2 a_{\ro}\right] \\
		\frac{\theta_{\ro}}{t} \cdot \tilde\nu\left(\frac{2\left(b_{\ro}-t\right)}{b_{\ro}}\right) & \text { if } t \in\left[b_{\ro} / 2, b_{\ro}\right]
	\end{array}\right.
	$$
\end{linenomath*}
Using the functions $c_*(t)$ and $\varepsilon_{\ro}(t)$ we define $c_{\ro}:\left[0, T\right] \rightarrow \mathbb{R}$ as
\begin{linenomath*}
	$$
	c_{\ro}(t) := c_{*}(t)-\frac{\varepsilon_{\ro}(t)}{4 \tilde\gamma \ro} \sin (2 \tilde\gamma \ro t)-\frac{\varepsilon_{\ro}^{\prime}(t)}{8 \tilde\gamma^{2} \ro^{2}} \sin ^{2}(\tilde\gamma \ro t)-\frac{\varepsilon_{\ro}(t)^{2}}{64 \tilde\gamma^{4} \ro^{2}} \sin ^{4}(\tilde\gamma \ro t).
	$$
\end{linenomath*}
By \cite[Propositions 4.8-4.9]{GG}, $(c_\ro(t))$ is a family of asymptotic activators with rate
\begin{linenomath*}
	$$
	\phi(\ro):=\frac{\theta_{\ro}}{32 \tilde\gamma^{2}} \ln \left(\frac{a_\ro}{b_\ro}\right),
	$$
\end{linenomath*}
and 
\begin{linenomath*}
	\[
	\lim _{\ro \rightarrow+\infty} d_{\mathcal{C}}\left(c_{\ro}, c_{*}\right)=0.
	\]
\end{linenomath*}
Since $c_{*}$ is a generic element of a dense subset,  we see that these universal activators cause an infinite loss of decay and derivatives.

\addcontentsline{toc}{section}{Acknowledgements}
	\section*{Acknowledgements}
	The first author is funded by the University Grants Commission, Government of India, under its JRF and SRF schemes.
	
	\appendix
	\section{Appendix I: Calculus}\label{A1}
	In this section we develop a calculus for the operators with symbols in additive form given in (\ref{symadd}). The following two propositions give their relations to the symbol classes $\G(m_1,m_2).$
	Let $\c(1,N),\c(2,N)$ and $\c(3,N)$ denote the indicator functions for the regions $\zint(N),\zmid(N) $ and $ \zext(N),$ respectively.
	
	\begin{prop}\label{p1}
		Let $a=a(t, x, \xi)$ be a symbol with
		\[
		    \begin{aligned}
		    	a &\in  \Gint(\tilde m_1,\tilde m_2,N,0) \cap
		    	\Gmid(m_1',m_2',0,N,0,0) \\
		    	&\qquad \cap
		    	\Gext(m_1,m_2,m_3,m_4,N,m_5,m_6),
		    \end{aligned}
		\]
		for $m_3 \geq 0$ and $m_3 \geq m_4.$
		Then, for any $\varepsilon \in (0,1),$
		   \begin{alignat*}{2}
		   		a &\in L^{\infty}\left([0, T] ; \Gt({m}_1^*,{m}_2^*,\Phi) \right), \quad &&\text{ if } m_5 \leq m_3 \text { and } m_6 \leq 0\\
		   		t^{1-\varepsilon}a &\in C\left([0, T] ; \Gt({m}_1^*,{m}_2^*,\Phi) \right), \qquad && \text { otherwise}
		   \end{alignat*}
	   where ${m}_i^* = \max\{\tilde m_i,m_i',m_i+m_3\},$ $i=1,2.$
	\end{prop}	
		\begin{proof} 
			The definition of the regions and straightforward calculations yield
			\[
				\begin{aligned}
					\vert &D_{x}^{\beta} \partial_{\xi}^{\alpha} a(t, x, \xi)\ \vert \\
					&\leq C_{\alpha, \beta} \P^{-|\beta|} \japxik^{-|\alpha|}   \Big( \c(1,N) \japxik^{\tilde m_1} \om^{\tilde m_2} + \c(2,N)\japxik^{m_1'} \om^{ m_2'} \\
					&\quad \qquad+ \c(3,N)  \japxik^{m_1} \om^{m_2}\t(m_3) \p(m_4) 
					 \tilde\theta(t)^{m_5+m_6(|\alpha|+|\beta|)} \Big)\\
					& \leq C_{\alpha, \beta} \P^{-|\beta|} \japxik^{-|\alpha|}  \Big( \c(1,N)\japxik^{\tilde m_1} \om^{\tilde m_2} + \c(2,N) \japxik^{m_1'} \om^{ m_2'} \\
					&\quad \qquad + \c(3,N) \japxik^{m_1} \om^{m_2} \left(  \frac{\theta(h)}{t_{x,\xi}}  \right)^{m_3} e^{m_4\psi(h)} 
					 \tilde\theta(h)^{m_3} \tilde\theta(t)^{m_5-m_3+m_6(|\alpha|+|\beta|)} \Big)\\
					& \leq C_{\alpha, \beta} \P^{m_2^*-|\beta|} \japxik^{m_1^*-|\alpha|} t^{-1+\varepsilon}.
				\end{aligned}
			\]
			The last estimate follows from the fact that $\om \lesssim \P$ and
			\[
				\tilde\theta(t)^{m_5-m_3+m_6(|\alpha|+|\beta|)} \leq t^{-1+\varepsilon}, \quad \varepsilon \in (0,1),
			\]
			since the singularities in our consideration are of logarithmic type.
		\end{proof}

		Similarly, we can prove the following two propositions.
		\begin{prop}\label{p2}
			Let $a=a(t, x, \xi)$ be a symbol with
			\[
			\begin{aligned}
				a &\in  \Gint(\tilde m_1,\tilde m_2,N,0) \cap
				\Gmid(m_1',m_2',m_3',N,m_4',m_5') \\
				&\qquad \cap
				\Gext(m_1,m_2,0,0,N,0,0),
			\end{aligned}
			\]
			for $m_3' \geq 0.$
			Then, for any $\varepsilon \in (0,1),$
			\begin{alignat*}{2}
				a &\in L^{\infty}\left([0, T]  ; \Gt({m}_1^*,{m}_2^*,\Phi ) \right), \quad &&\text{ if } m_4', m_5' \leq 0\\
				t^{1-\varepsilon}a &\in C\left([0, T] ; \Gt({m}_1^*,{m}_2^*,\Phi )\right), \qquad && \text { otherwise}
			\end{alignat*}
			where ${m}_i^* = \max\{\tilde m_i,m_i'+m_3',m_i\},$ $i=1,2.$
		\end{prop}	
	
		\begin{prop}\label{p3}
			Let $a=a(t, x, \xi)$ be a symbol with
			\[
			\begin{aligned}
				a &\in  \Gint(\tilde m_1,\tilde m_2,N,\tilde m_3) \cap
				\Gmid(m_1',m_2',0,N,0,0) \\
				&\qquad \cap
				\Gext(m_1,m_2,0,0,N,0,0),
			\end{aligned}
			\]
			for $\tilde m_3 \geq 0.$
			Then, for any $\varepsilon \in (0,1),$
			\begin{alignat*}{2}
				a &\in L^{\infty}\left([0, T]  ; \Gt({m}_1^*,{m}_2^*,\omega) \right), \quad &&\text{ if } \tilde m_3= 0\\
				t^{1-\varepsilon}a &\in C\left([0, T] ; \Gt({m}_1^*,{m}_2^*,\omega) \right), \qquad && \text { otherwise}
			\end{alignat*}
			where ${m}_i^* = \max\{\tilde m_i,m_i',m_i\}, i=1,2.$
		\end{prop}	
		\begin{rmk}\label{lb}
			Since the function $\tilde\theta$ is of logarthmic type, we have
			\[
				\tilde\theta(\tilde t_{x,\xi})^l \leq \tilde\theta(t_{x,\xi})^l \leq\tilde\theta(\h)^l \leq (\P\japxik)^{\varepsilon},
			\]
			for any $l>0$ and $\varepsilon \ll 1.$ Similar is the case for the functions $\theta$ and $\psi.$
		\end{rmk}
		
		For $\mu>0$ and $r \geq 2$, we set 
		\begin{linenomath*}
			\[
			Q_{r,\mu} = \{(x,\xi) \in \R^{2n} : \P < r, \japxik < r\}, \qquad Q_{r,\mu}^c = \R^{2n} \setminus Q_{r,\mu}.
			\]
		\end{linenomath*}
		
		\begin{prop}[Asymptotic Expansion]\label{p4}
			Let $\{a_j\}, j\geq0$ be a sequence of symbols with 
			\[
				\begin{aligned}
					a_j &\in G^{\tilde{m}_1,1}\{\tilde m_3\} (\omega^{\tilde{m}_2}\Phi^{-j},\go)^{(1)}_N \cap G^{{m}_1',1}\{m_3',m_4'+2m_5'j,m_5'\}(\omega^{{m}_2'}\Phi^{-j},\go)^{(2)}_N \\
					&\qquad \cap G^{{m}_1,1}\{m_3,m_4,m_5+2m_6j,m_6\}(\omega^{{m}_2}\Phi^{-j},\go)^{(2)}_N.
				\end{aligned}
			\]
			Then, there is a symbol
			$$
			\begin{aligned}
				a &\in \Gint(\tilde m_1,\tilde m_2,N,\tilde m_3) \cap
				\Gmid(m_1',m_2',m_3',N,m_4',m_5')\\ 
				& \qquad \cap
				\Gext(m_1,m_2,m_3,m_4,N,m_5,m_6)
			\end{aligned}
			$$
			such that
			$$
			a(t, x, \xi) \sim \sum_{j=0}^{\infty} a_{j}(t, x, \xi),
			$$
			that is for all $j_{0} \geq 1,$ $a(t, x, \xi)-\sum_{j=0}^{j_{0}-1} a_{j}(t, x, \xi) $ belongs to 
			$$
			\begin{aligned}
				&G^{\tilde{m}_1-j_0,1}\{\tilde m_3\}(\omega^{\tilde{m}_2}\Phi^{-j_0},\go)^{(1)}_N \cap G^{{m}_1'-j_0+\varepsilon,1}\{m_3',m_4',m_5'\}(\omega^{{m}_2'}\Phi^{-j_0+\varepsilon},\go)^{(2)}_N \\
				&\qquad \cap G^{{m}_1-j_0+\varepsilon,1}\{m_3,m_4,m_5,m_6\}(\omega^{{m}_2}\Phi^{-j_0+\varepsilon},\go)^{(2)}_N,
			\end{aligned}
			$$
			where $\varepsilon \ll 1.$ The symbol is uniquely determined modulo $C\left((0, T] ; G^{-\infty}\right) .$
		\end{prop}
				
		\begin{proof}
			Let us fix $\varepsilon \ll 1$ and set $\mu=1-\varepsilon.$
			Consider a $C^{\infty}$ cut-off function, $\chi$ defined by
			\begin{linenomath*}
				$$
				\chi(x,\xi)=\left\{\begin{array}{ll}
					1, & (x,\xi) \in Q_{2k,\mu} \\
					0, & (x,\xi) \in Q_{4k,\mu}^c
				\end{array}\right.
				$$
			\end{linenomath*}
			and $0 \leq \chi \leq 1 .$ For a sequence of positive numbers $\varepsilon_{j} \rightarrow 0$, we define
			\begin{linenomath*}
				$$
				\begin{aligned}
					\gamma_0(x,\xi) &\equiv 1, \\
					\gamma_{{j}}(x,\xi) &=1-\chi\left(\varepsilon_{j} x,\varepsilon_{j} \xi\right), \quad j \geq 1.
				\end{aligned}
				$$
			\end{linenomath*}
			We note that $\gamma_{{j}}(x,\xi)=0$ in $Q_{2k,\mu}$ for $j \geq 1.$  We choose $\varepsilon_{j}$ such that
			\begin{linenomath*}
				$$
				\varepsilon_{j} \leq 2^{-j}
				$$
			\end{linenomath*}
			and set
			\begin{linenomath*}
				$$
				a(t, x, \xi)=\sum_{j=0}^{\infty} \gamma_{{j}}(x,\xi) a_{j}(t, x, \xi).
				$$
			\end{linenomath*}
			We note that $a(t, x, \xi)$ exists (i.e. the series converges point-wise), since for any fixed point $(t, x, \xi)$ only a finite number of summands contribute to $a(t, x, \xi) .$ Indeed, for fixed $(t, x, \xi)$ we can always find a $j_{0}$ such that $\P < \frac{1}{\varepsilon_{j_0}}$,  $\japxik<\frac{1}{\varepsilon_{j_0}}$ and hence
			\begin{linenomath*}
				$$
				a(t, x, \xi)=\sum_{j=0}^{j_{0}-1} \gamma_{j}(x,\xi) a_{j}(t, x, \xi).
				$$
			\end{linenomath*}
			Observe that
			\begin{linenomath*}
				$$
				\begin{aligned}
					|D_{x}^{\beta} \partial_{\xi}^{\alpha}\left(\gamma_{{j}}(x,\xi) a_{j}(t, x, \xi)\right)|& \leq \sum\limits_{\substack{\alpha^{\prime}+\alpha^{\prime \prime}=\alpha \\ \beta^{\prime}+\beta^{\prime \prime}=\beta}}\left(\begin{array}{c}
						\alpha \\
						\alpha^{\prime}
					\end{array}\right) \left(\begin{array}{c}
						\beta \\
						\beta^{\prime}
					\end{array}\right) |\partial_{\xi}^{\alpha^{\prime}} D_x^{\beta^\prime} \gamma_{j}(x,\xi) D_{x}^{\beta^{\prime \prime}} \partial_{\xi}^{\alpha^{\prime \prime}} a_{j}(t, x, \xi)| \\
					&\leq  \mid \gamma_{{j}}(x,\xi) D_{x}^{\beta} \partial_{\xi}^{\alpha} a_{j}(t, x, \xi) \\
					&\qquad +\sum\limits_{\substack{\alpha^{\prime}+\alpha^{\prime \prime}=\alpha,|\alpha^{\prime}|>0 \\ \beta^{\prime}+\beta^{\prime \prime}=\beta,|\beta^{\prime}|>0 } }   C_{\alpha^{\prime} \beta^{\prime}}  \frac{\tilde{\chi}_{{j}}(x,\xi)}{\P^{|\beta^{\prime}|} \japxik^{|\alpha^{\prime}|}}D_{x}^{\beta^{\prime \prime}} \partial_{\xi}^{\alpha^{\prime \prime}} a_{j}(t, x, \xi) \mid,
				\end{aligned}
				$$
			\end{linenomath*}
			where $\tilde{\chi}_{{j}}(x,\xi)$ is a smooth cut-off function supported in $Q_{2k,\mu}^c \cap Q_{4k,\mu}.$
			This new cut-off function describes the support of the derivatives of $\gamma_{j}(x,\xi) .$ In the last estimate, we also used that $\frac{1}{\varepsilon_{j}}\sim \japxik$ and $\frac{1}{\varepsilon_{j}} \sim \P$ if $\tilde{\chi}_{j}(x,\xi) \neq 0 .$ We conclude that
			\begin{linenomath*}
				\begin{equation}\label{asym1}
				\begin{aligned}
					|D_{x}^{\beta} &\partial_{\xi}^{\alpha} \gamma_{j}(x,\xi) a_{j}(t, x, \xi)| \\
					& \leq \frac{1}{2^j} \japxik^{\mu-j-|\alpha|} \P^{\mu-j-|\beta|}\Bigg[ \c(1,N)  \japxik^{\tilde m_1} \om^{\tilde m_2} \tilde \theta(t)  \\
					&\qquad +  \c(2,N) \japxik^{m_1'} \om^{m_2'} \left( \frac{\theta(t)}{t}\right)^{m_3'} \tilde \theta(t)^{m_4'+m_5'(|\alpha| + |\beta| + 2j)}\\
					& \qquad + \c(3,N) \japxik^{m_1} \om^{m_2} 
					 \left(\frac{\theta(t)}{t}\right)^{m_3} e^{m_4\psi(t)} \tilde \theta(t)^{m_5+m_6(|\alpha| + |\beta| + 2j)}  \Bigg],
				\end{aligned}
				\end{equation}
			\end{linenomath*}
			where we have estimated $\frac{\P^{\mu}}{ 2^{j}} \geq 1$and $\frac{\japxik^{\mu}}{ 2^{j}} \geq 1$ (due to the support of cut-off functions) once in each summand.
			In $\zmid(N)$ and $\zext(N),$ $\tilde \theta(t) \leq \tilde \theta(h).$ For any $r \geq 0$ and $ j\geq 1,$
			\begin{equation}\label{asym2}
				\tilde \theta(t)^{rj} \leq \tilde \theta(h)^{rj} \lesssim (\P\japxik)^{\varepsilon},
			\end{equation}
			as the singularity in our consideration is of logarithmic type.
			From (\ref{asym1}) and (\ref{asym2}), we obtain
			\begin{linenomath*}
				\begin{equation}\label{asym3}
					\begin{aligned}
						|D_{x}^{\beta} &\partial_{\xi}^{\alpha} \gamma_{j}(x,\xi) a_{j}(t, x, \xi)| \\
						& \leq \frac{1}{2^j} \japxik^{\mu-j-|\alpha|} \P^{\mu-j-|\beta|}\Bigg[ \c(1,N)  \japxik^{\tilde m_1} \om^{\tilde m_2} \tilde \theta(t)  \\
						&\qquad +  \c(2,N) \japxik^{m_1'+\varepsilon} \P^{\varepsilon} \om^{m_2'} \left( \frac{\theta(t)}{t}\right)^{m_3'} \tilde \theta(t)^{m_4'+m_5'(|\alpha| + |\beta| )}\\
						& \qquad + \c(3,N) \japxik^{m_1+\varepsilon} \P^{\varepsilon} \om^{m_2} 
						\left(\frac{\theta(t)}{t}\right)^{m_3} e^{m_4\psi(t)} \tilde \theta(t)^{m_5+m_6(|\alpha| + |\beta| )}  \Bigg]\\
						& \leq \frac{1}{2^j} \japxik^{1-j-|\alpha|} \P^{1-j-|\beta|}\Bigg[ \c(1,N)  \japxik^{\tilde m_1} \om^{\tilde m_2} \tilde \theta(t)  \\
						&\qquad +  \c(2,N) \japxik^{m_1'}  \om^{m_2'} \left( \frac{\theta(t)}{t}\right)^{m_3'} \tilde \theta(t)^{m_4'+m_5'(|\alpha| + |\beta| )}\\
						& \qquad + \c(3,N) \japxik^{m_1}  \om^{m_2} 
						\left(\frac{\theta(t)}{t}\right)^{m_3} e^{m_4\psi(t)} \tilde \theta(t)^{m_5+m_6(|\alpha| + |\beta| )}  \Bigg],
					\end{aligned}
				\end{equation}
			\end{linenomath*}
			where $j \geq 1.$ Now
			\begin{linenomath*}
				$$
				\begin{aligned}
					|D_{x}^{\beta} &\partial_{\xi}^{\alpha} a(t,x, \xi) | \\		
					\leq &\left|D_{x}^{\beta} \partial_{\xi}^{\alpha}\left(\gamma_{0}(x,\xi) a_{0}(t, x, \xi)\right)\right| + \sum_{j=1}^{j_{0}-1}\left|D_{x}^{\beta} \partial_{\xi}^{\alpha}\left(\gamma_{j}(x,\xi) a_{j}(t, x, \xi)\right)\right| .
				\end{aligned}
				$$
			\end{linenomath*}
			Combining the symbol estimate of $a_0$ and the estimate (\ref{asym3}), we readily obtain
			\begin{linenomath*}
			$$
			\begin{aligned}
				a &\in \Gint(\tilde m_1,\tilde m_2,N,\tilde m_3) \cap
				\Gmid(m_1',m_2',m_3',N,m_4',m_5')\\ 
				& \qquad \cap
				\Gext(m_1,m_2,m_3,m_4,N,m_5,m_6)
			\end{aligned}
			$$
			\end{linenomath*}
			Arguing as above, we see that $\sum_{j=j_{0}}^{\infty} \gamma_{j} a_{j}$ belongs to
			\begin{linenomath*}
				\[ 
					\begin{aligned}
					&G^{\tilde{m}_1-j_0,1}\{\tilde m_3\}(\omega^{\tilde{m}_2}\Phi^{-j_0},\go)^{(1)}_N \cap G^{{m}_1'-j_0+\varepsilon,1}\{m_3',m_4',m_5'\}(\omega^{{m}_2'}\Phi^{-j_0+\varepsilon},\go)^{(2)}_N \\
					&\qquad \cap G^{{m}_1-j_0+\varepsilon,1}\{m_3,m_4,m_5,m_6\}(\omega^{{m}_2}\Phi^{-j_0+\varepsilon},\go)^{(2)}_N,
					\end{aligned}
				\]
			\end{linenomath*}
			and thus, $a(t, x, \xi)-\sum\limits_{j=0}^{j_{0}-1} a_{j}(t, x, \xi)$ belongs to
			\begin{linenomath*}
				$$ 
					\begin{aligned}
						&G^{\tilde{m}_1-j_0,1}\{\tilde m_3\}(\omega^{\tilde{m}_2}\Phi^{-j_0},\go)^{(1)}_N \cap G^{{m}_1'-j_0+\varepsilon,1}\{m_3',m_4',m_5'\}(\omega^{{m}_2'}\Phi^{-j_0+\varepsilon},\go)^{(2)}_N \\
						&\qquad \cap G^{{m}_1-j_0+\varepsilon,1}\{m_3,m_4,m_5,m_6\}(\omega^{{m}_2}\Phi^{-j_0+\varepsilon},\go)^{(2)}_N.
					\end{aligned}
				$$
			\end{linenomath*}
			Lastly, we use Propositions \ref{p1} - \ref{p3} to conclude that 
			\begin{linenomath*}
				$$
				t^{1-\varepsilon}a_{j} \in C\left([0, T] ; G^{ m_1^*-j,1}(\omega^{ m_2^*}\Phi^{r^*-j},\go) \right)
				$$
			\end{linenomath*}
			for $m_1^*=\max\{\tilde m_1,m_1'+m_3',m_1+m_3\},$ $m_1^*=\max\{\tilde m_2,m_2',m_2\}$ and $r^*=\max\{m_3',m_3\}.$
			As $j$ tends to $+\infty,$ the intersection of all those spaces belongs to the space {\footnotesize{$C\left((0, T] ; G^{-\infty}\right) .$}} This completes the proof.
		\end{proof}
		
		\begin{lem}
		Let $A$ and $B$ be pseudodifferential operators with symbols
		\begin{linenomath*}
			$$
				\begin{aligned}
					a=\sigma(A) & \in\Gint(\tilde m_1,\tilde m_2,2N,\tilde m_3) \cap
					\Gmid(m_1',m_2',m_3',N,m_4',m_5') \\
					& \qquad \cap
					\Gext(m_1,m_2,m_3,m_4,N,m_5,m_6)
				\end{aligned}
			$$
		\end{linenomath*}
		and
		\begin{linenomath*}
			$$
			 \begin{aligned}
			 	b=\sigma(B) &\in \Gint(\tilde l_1,\tilde l_2,2N,\tilde l_3) \cap
			 	\Gmid(l_1',l_2',l_3',N,l_4',l_5')\\
			 	& \qquad \cap	 	\Gext(l_1,l_2,l_3,l_4,N,l_5,l_6).
			 \end{aligned}
			$$
		\end{linenomath*}
		Then, the pseudodifferential operator $C=A \circ B$ has a symbol $c=\sigma(C) $ in
		\begin{linenomath*}
			$$
			\begin{aligned}
				&\Gint(\tilde m_1 + \tilde l_1,\tilde m_2 + \tilde l_2,2N,\tilde m_3+ \tilde l_3) \cap
				\Gmid(m_1'+l_1', m_2'+l_2', m_3'+l_3',N,m_4+l_4',m_5'+l'_5) \\
				& \qquad \cap
				\Gext(m_1+l_1,m_2+l_2,m_3+l_3,m_4+l_4,N,m_5+l_5,m_6+l_6)
			\end{aligned}
			$$
		\end{linenomath*}
		and satisfies
		\begin{equation} \label{a4}
			c(t, x, \xi) \sim \sum_{\alpha \in \mathbb{N}^{n}} \frac{1}{\alpha !} \partial_{\xi}^{\alpha} a(t, x, \xi) D_{x}^{\alpha} b(t, x, \xi).
		\end{equation}
		The operator $C$ is uniquely determined modulo an operator with symbol from\\ $C\left((0, T] ;G^{-\infty} \right).$
	\end{lem}
	\begin{proof}
		In view of Propositions \ref{p1} - \ref{p4}, it is clear that the operator $C$ is a well-defined pseudodifferential operator. Relation (\ref{a4}) is a direct consequence of the standard composition rules (see \cite[Section 1.2]{nicRodi}).
	\end{proof}

	\begin{lem}
		Let $A$ be a pseudodifferential operator with an invertible symbol
		\begin{linenomath*}
			$$
			a=\sigma(A) \in \Gint(0,0,N,0) \cap
			\Gmid(0,0,0,N,0,0) \cap
			\Gext(0,0,0,0,N,0,0).
			$$
		\end{linenomath*}
		Then, there exists a parametrix $A^{\#}$ with symbol $a^{\#}$ in 
		\begin{linenomath*}
			$$
			\Gint(0,0,N,0) \cap
			\Gmid(0,0,0,N,0,0) \cap
			\Gext(0,0,0,0,N,0,0).
			$$
		\end{linenomath*}
	\end{lem}

	\begin{proof}
		We use the existence of the inverse of $a$ and set
		\begin{linenomath*}
			$$
			\begin{aligned}
				a_{0}^{\#}(t, x, \xi)=a(t, x, \xi)^{-1} & \in \Gint(0,0,N,0) \cap
				\Gmid(0,0,0,N,0,0)\\
				&\qquad \cap
				\Gext(0,0,0,0,N,0,0).
			\end{aligned}
			$$
		\end{linenomath*}
		In view of Propositions \ref{p1} - \ref{p3}, one can define a sequence $a_{j}^{\#}(t, x, \xi)$ recursively by
		\begin{linenomath*}
			$$
			\sum_{1 \leq|\alpha| \leq j} \frac{1}{\alpha !} \partial_{\xi}^{\alpha} a(t, x, \xi) D_{x}^{\alpha} a_{j-|\alpha|}^{\#}(t, x, \xi)=-a(t, x, \xi) a_{j}^{\#}(t, x, \xi)
			$$
		\end{linenomath*}
		with $a_{j}^{\#} $ in
		\begin{linenomath*}
			$$
			G^{-j,1}\{0\}(\Phi^{-j},\go)^{(1)}_N \cap G^{-j,1}\{0,0,0\} (\Phi^{-j},\go)^{(2)}_N  \cap G^{-j,1}\{0,0,0,0\} (\Phi^{-j},\go)^{(3)}_N.
			$$
		\end{linenomath*}
		Proposition \ref{p4} then yields the existence of a symbol
		\begin{linenomath*}
			$$
			a_{R}^{\#} \in  \Gint(0,0,N,0) \cap
			\Gmid(0,0,0,N,0,0) \cap
			\Gext(0,0,0,0,N,0,0)
			$$
		\end{linenomath*}
		and a right parametrix $A_{R}^{\#}(t, x, \xi)$ with symbol $\sigma\left(A_{R}^{\#}\right)=a_{R}^{\#} .$ We have
		\begin{linenomath*}
			$$
			A A_{R}^{\#}-I \in C\left([0, T] ; G^{-\infty} \right).
			$$
		\end{linenomath*}
		The existence of a left parametrix follows in similar lines. One can also prove the existence of a parametrix $A^{\#}$ by showing that right and left parametrix coincide up to a regularizing operator.
	\end{proof}
	
	\section{Appendix II: Sharp G\r{a}rding Inequality For Parameter Dependent Matrix}\label{app2}
	
	In this section we prove the sharp G\r{a}rding inequality for a matrix pseudodifferential operator with symbol $a \var $ in 
	\begin{equation}\label{cl}
		\Gint(1,1,N,1)  \cap \Gmid(0,0,1,N,0,1) \cap \Gext(-1,-1,2,1,N,2,1).
	\end{equation}
	This is used Sections \ref{Energy1} and  \ref{Energy2} to arrive at an energy estimate.
	By Propositions \ref{p1} and \ref{p2}, for any $\varepsilon \in (0,1),$
		\begin{alignat}{3}
			a &\in L^{\infty} \left([0,1];G^{1,1}(\Phi,\go)\right), && \qquad \text{ if } \tilde \theta \text{ is bounded,}\\
			\label{cl2}
			t^{1-\varepsilon}a &\in C\left([0,1];G^{1,1}(\Phi,\go)\right), && \qquad \text{ otherwise.}
		\end{alignat}
	We give below the sharp G\r{a}rding inequality when $\tilde\theta$ is unbounded. The case for bounded $\tilde\theta$ follows in similar lines by taking $\varepsilon=1$ in the proof.
	\begin{thm}\label{sg}
		Let  $a(t,x,\xi)$ be $2 \times 2$ positive semi-definite matrix belonging to the symbol class given in (\ref{cl}). Then, for each $t \in [0,T]$ and $\varepsilon \in (0,1)$ we have 
		\begin{equation}\label{sg4}
			\left(t^{1-\varepsilon} a^{w}(t,x, D) u, u\right) \geqq-C\|u\|^{2}, \quad u \in \mathcal{S}\left(\mathbb{R}^{n};\R^2\right).
		\end{equation}
	\end{thm}
	\begin{proof}
		We have the following region-wise seminorms for $a \var$:
			\begin{alignat*}{3}
				& \text{in } \zint(N),  \qquad && |a|_{j}^{g}(w) && \leq h^{-1} \tilde\theta,\\
				& \text{in } \zmid(N),   && |a|_{j}^{g}(w) && \leq \frac{\theta(t)}{t}\tilde\theta^j,\\
				& \text{in } \zext(N),  && |a|_{j}^{g}(w) && \leq h\left(\frac{\theta(t)}{t}\right)^2e^{\psi(t)} {\tilde\theta}^{2+j}.\\
			\end{alignat*}
		Here $w\in \R^{2n}$, $g=\go$ and the norm of the $j^{th}$ differential at $x$ given by
		\[
			|f|_j^g := \sup_{y_l \in \R^{2n}} |f^{(j)}(x;y_1,\dots,y_j)| \Big/ \prod_{1}^{j} g(y_l)^{\frac{1}{2}}.
		\] 
		Using an appropriate affine symplectic transformation one can assume that $\go=h e$ where $e$ is the Euclidean metric form. Then
		\begin{alignat}{3}
			\label{sg11}
			& \text{in } \zint(N),  \qquad && |a|_{j}^{e}(w) && \leq h^{(j-2) / 2}\tilde\theta,\\
			\label{sg12}
			& \text{in } \zmid(N),   && |a|_{j}^{e}(w) && \leq h^{j / 2}\frac{\omega(t)}{t}\tilde\theta^j,\\
			\label{sg13}
			& \text{in } \zext(N),  && |a|_{k}^{e}(w) && \leq h^{(j+2) / 2}\left(\frac{\omega(t)}{t}\right)^2e^{\psi(t)}{\tilde\theta}^{2+j}.
		\end{alignat}
		Let $a^{(0)}_t+a^{(1)}_t(x, \xi)$ be the first order Taylor expansion of $a$ at $(x,\xi)=0 .$ Let $v \in \R^2$. By the above semi-norms we have
		\begin{alignat}{2}
			\label{t1}
			&\text{in } \zint(N), \quad &&
			\begin{cases}
				\begin{array}{l}
					\left(a^{(0)}_t v, v\right)+\left(a^{(1)}_t(x, 0) v, v\right)+\frac{|x|^{2}}{2}(\tilde\theta v,v) \geqq 0, \\
					\left(a^{(0)}_t v, v\right)+\left(a^{(1)}_t(0, \xi) v, v\right)+ \frac{|\xi|^{2}}{2}(\tilde\theta v,v) \geqq 0,
				\end{array}
			\end{cases}\\
			\label{t2}
			& \text{in } \zmid(N), \quad &&
			\begin{cases}
				\begin{array}{l}
					\left(a^{(0)}_t v, v\right)+\left(a^{(1)}_t(x, 0) v, v\right)+\frac{|x|^{2}}{2}\left(\frac{\theta(t)}{t}\tilde\theta^2hv,v\right) \geqq 0, \\
					\left(a^{(0)}_tv, v\right)+\left(a^{(1)}_t(0, \xi) v, v\right)+\frac{|\xi|^{2}}{2}\left(\frac{\theta(t)}{t}\tilde\theta^2hv,v\right) \geqq 0,
				\end{array}
			\end{cases}\\
			\label{t3}
		 & \text{in } \zext(N), \quad &&
		 \begin{cases}
		 	\begin{array}{l}
		 		\left(a^{(0)}_t v, v\right)+\left(a^{(1)}_t(x, 0) v, v\right)+\frac{|x|^{2}}{2}\left(\frac{\theta(t)^2}{t^2}e^{\psi(t)}\tilde\theta^4h^2v,v\right) \geqq 0, \\
		 		\left(a^{(0)}_t v, v\right)+\left(a^{(1)}_t(0, \xi) v, v\right)+\frac{|\xi|^{2}}{2}\left(\frac{\theta(t)^2}{t^2}e^{\psi(t)}\tilde\theta^4h^2v,v\right) \geqq 0.
		 	\end{array}
		 \end{cases}
		\end{alignat}
		Let us fix $0< \varepsilon <\varepsilon'<1.$ Using the definition of the regions, we see that in the whole of extended phase space
		\[
			\begin{cases}
				\begin{array}{l}
					\left(t^{1-\varepsilon'}a_t^{(0)} v, v\right)+\left(t^{1-\varepsilon'}a_t^{(1)}(x, 0) v, v\right)+\frac{|x|^{2}}{2}(t^{1-\varepsilon'}\tilde\theta(t)^q v,v)  \geqq 0, \\
					\left(t^{1-\varepsilon'}a_t^{(0)} v, v\right)+\left(t^{1-\varepsilon'}a_t^{(1)}(0, \xi) v, v\right)+\frac{|\xi|^{2}}{2}(t^{1-\varepsilon'}\tilde\theta(t)^q v,v)  \geqq 0,
				\end{array}
			\end{cases}
		\]
		where $q=1$ in $\zint(N)$ while in $\zmid(N)$ and $\zext(N)$ $q=2.$ Since the function $\tilde\theta(t)$ is of logarithmic type, in the whole of extended phase space, we have
		\[
			\begin{cases}
				\begin{array}{l}
					t^{1-\varepsilon} \left( (a_t^{(0)}  v, v)+ (a_t^{(1)} (x, 0) v, v)+\frac{|x|^{2}}{2}\|v\|^{2} \right)  \geqq 0, \\
					t^{1-\varepsilon} \left( (a_t^{(0)}  v, v)+ (a_t^{(1)} (0, \xi) v, v)+\frac{|\xi|^{2}}{2}\|v\|^{2} \right)  \geqq 0.
				\end{array}
			\end{cases}
		\]
		From here on we proceed as in \cite[Theorem 18.6.14]{Horm} to obtain the result.
	\end{proof}
	
	\addcontentsline{toc}{section}{References}

\begin{thebibliography}{10}
		
		\bibitem{AscaCappi1}
		Ascanelli, A. and Cappiello, M. (2006) Log-{L}ipschitz regularity for {SG}
		hyperbolic systems. {\em J. Differ. Equ.\/}, {\bf 230}, 556 -- 578.
		
		\bibitem{AscaCappi2}
		Ascanelli, A. and Cappiello, M. (2008) H{\"o}lder continuity in time for {SG}
		hyperbolic systems. {\em J. Differ. Equ.\/}, {\bf 244}, 2091 -- 2121.
		
		\bibitem{Feff}
		Beals, R. and Fefferman, C. (1974) Spatially inhomogeneous pseudodifferential
		operators {I}. {\em Commun. Pur. Appl. Math.\/}, {\bf 27}, 1--24.
		
		\bibitem{Petkov}
		Bernardi, E., Bove, A., and Petkov, V. (2015) Cauchy problem for effectively
		hyperbolic operators with triple characteristics of variable multiplicity.
		{\em J. Hyperbolic Differ. Equ.\/}, {\bf 12}, 535--579.
		
		\bibitem{Cico1}
		Cicognani, M. (2003) The {C}auchy problem for strictly hyperbolic operators
		with non-absolutely continuous coefficients. {\em Tsukuba J. Math.\/}, {\bf
			27}, 1--12.
		
		\bibitem{Cico2}
		Cicognani, M. (2004) Coefficients with unbounded derivatives in hyperbolic
		equations. {\em Math. Nachr.\/}, {\bf 276}, 31--46.
		
		\bibitem{CicoLor}
		Cicognani, M. and Lorenz, D. (2018) Strictly hyperbolic equations with
		coefficients low-regular in time and smooth in space. {\em J. Pseudo Differ.
			Oper. Appl.\/}, {\bf 9}, 643--675.
		
		\bibitem{CSK}
		Colombini, F., Del~Santo, D., and Kinoshita, T. (2002) Well-posedness of the
		{Cauchy} problem for a hyperbolic equation with non-{L}ipschitz coefficients.
		{\em Ann. Scoula. Norm.-Sci.\/}, {\bf Ser. 5, 1}, 327--358.
		
		\bibitem{CSR}
		Colombini, F., {Del Santo}, D., and Reissig, M. (2003) On the optimal
		regularity of coefficients in hyperbolic cauchy problems. {\em Bulletin des
			Sciences Mathématiques\/}, {\bf 127}, 328--347.
		
		\bibitem{GG}
		Ghisi, M. and Gobbino, M. (2021) Finite vs infinite derivative loss for
		abstract wave equations with singular time-dependent propagation speed. {\em
			Bulletin des Sciences Mathématiques\/}, {\bf 166}, 102918.
		
		\bibitem{hiro}
		Hirosawa, F. (2003) Loss of regularity for the solutions to hyperbolic
		equations with non-regular coefficients—an application to kirchhoff
		equation. {\em Mathematical Methods in the Applied Sciences\/}, {\bf 26},
		783--799.
		
		\bibitem{Horm}
		H{\"o}rmander, L. (1985) {\em The Analysis of Linear Partial Differential
			Operators III\/}. Springer-Verlag.
		
		
		\bibitem{kn}
		Kajitani, K. and Nishitani, T. (1991) {\em The Hyperbolic Cauchy Problem\/}. Springer, Berlin, Heidelberg.
			
		\bibitem{KuboReis}
		Kubo, A. and Reissig, M. (2003) Construction of parametrix to strictly
		hyperbolic {C}auchy problems with fast oscillations in non{L}ipschitz
		coefficients. {\em Commun. Part. Diff. Eq.\/}, {\bf 28}, 1471--1502.
		
		\bibitem{Lern}
		Lerner, N. (2010) {\em Metrics on the Phase Space and Non-Selfadjoint
			Pseudo-Differential Operators\/}. Birkh{\"a}user Basel.
		
		\bibitem{LorReis}
		Lorenz, D. and Reissig, M. (2019) A generalized levi condition for weakly
		hyperbolic cauchy problems with coefficients low regular in time and smooth
		in space. {\em J. Pseudo Differ. Oper. Appl.\/}, {\bf 10}, 199--240.
		
		\bibitem{nicRodi}
		Nicola, F. and Rodino, L. (2011) {\em Global Pseudo-differential Calculus on
			{E}uclidean Spaces\/}. Birkh{\"a}user Basel.
		
		\bibitem{Rahul_NUK2}
		Pattar, R.~R. and Uday~Kiran, N. (2021), Global well-posedness of a class of strictly hyperbolic cauchy problems with coefficients non-absolutely continuous in time. Bulletin des Sciences Math\'ematiques , 171, 103037.
		
		\bibitem{RahulNUK3}
		Pattar, R.~R., Kiran, N.~U. (2021) Strictly hyperbolic Cauchy problems on $\R^n$
		with unbounded and singular coefficients. Annali dell'Universita di Ferrara, https://doi.org/10.1007/s11565-021-00378-2.
		
		\bibitem{JR}
		Rauch, J. (2012) {\em Hyperbolic Partial Differential Equations and Geometric
			Optics\/}. Graduate studies in mathematics, American Mathematical Society.
		
		\bibitem{Reis}
		Reissig, M. (2003) Hypebolic equations with non-{L}ipschitz coefficients. {\em
			Rend. Sem. Mat. Univ. Pol. Torino\/}, {\bf 61}, 135--182.
		
		\bibitem{NUKCori1}
		Uday~Kiran, N., Coriasco, S., and Battisti, U. (2016) Hyperbolic operators with
		non-{L}ipschitz coefficients. {\em Recent Advances in Theoretical \&
			Computational Partial Differential Equations with Applications\/}, Panjab
		University, Chandigarh.
		
		\bibitem{yag}
		Yagdjian, K. (1997) {\em The Cauchy Problem for Hyperbolic Operators: Multiple
			Characteristics. Micro-Local Approach\/}. Wiley.
		
	\end{thebibliography}

\end{document}